\documentclass{article}
\usepackage[a4paper, total={6in, 10in}]{geometry}

\usepackage{graphicx} 
\usepackage[numbers]{natbib}

\usepackage{float}
\usepackage{placeins}
\usepackage{morefloats}
\usepackage{capt-of}

\usepackage{parskip}
\usepackage{amsmath}
\usepackage{amssymb}

\usepackage{amsthm}
\usepackage{amsfonts}
\usepackage{mathrsfs}
\usepackage{xcolor}
\usepackage{tikz}
\usepackage{enumerate}
\usepackage{float}
\definecolor{darkgreen}{rgb}{0.2,0.8,0}

\usepackage[font=scriptsize]{caption}

\usepackage{hyperref}
\hypersetup{%
  colorlinks = true,
  linkcolor  = orange
}
\usepackage[english]{babel}

\newtheoremstyle{assumptionStyle}
  {}{}{\itshape}{}{\bfseries}{.}{ }
  {\thmname{#1}\thmnumber{ (#2)}\thmnote{ #3}}
\theoremstyle{assumptionStyle}
\newtheorem{assumption}{Assumption}

\usepackage{mathtools}                                                      
\mathtoolsset{showonlyrefs=true}                                    
\allowdisplaybreaks

\usepackage[textsize=footnotesize]{todonotes}

\definecolor{darkgreen}{rgb}{0.4,0.7,0}

\definecolor{darkblue}{rgb}{0,0,.5}

\numberwithin{equation}{section}
\theoremstyle{plain}
\newtheorem{Theorem}{Theorem}[section]
\newtheorem{Definition}[Theorem]{Definition}
\newtheorem{Proposition}[Theorem]{Proposition}

\newtheorem{Lemma}[Theorem]{Lemma}
\newtheorem{Corollary}[Theorem]{Corollary}
\newtheorem{Remark}[Theorem]{Remark}

\usepackage[deletedmarkup=xout]{changes}

\newcommand \bE{\mathbb{E}}

\newcommand{\bI}{\mathbb{I}}
\newcommand{\bL}{\mathbb{L}}
\newcommand{\bN}{\mathbb{N}}
\newcommand \bP{\mathbb{P}}

\newcommand \bR{\mathbb{R}}

\newcommand{\bZ}{\mathbb{Z}}

\newcommand{\cE}{\mathcal{E}}
\newcommand \cF{\mathcal{F}}

\newcommand{\cL}{\mathcal{L}}
\newcommand{\cM}{\mathcal{M}}
\newcommand{\cN}{\mathcal{N}}

\newcommand{\cP}{\mathcal{P}}

\newcommand{\cU}{\mathcal{U}}

\newcommand{\sS}{\mathscr{S}}
\newcommand{\sM}{\mathscr{M}}

\newcommand \dd{\mathrm d}

\newcommand \ds{\dd s}
\newcommand \dt{\dd t}
\newcommand \du{\dd u}
\newcommand \dx{\dd x}
\newcommand \dy{\dd y}
\newcommand \dz{\dd z}

\newcommand \dW{\dd W}

\newcommand \Exp[1]{\mathbb E\left[#1\right]}
\newcommand \E[1]{\Exp{#1}}

\renewcommand \epsilon{\varepsilon}

\newcommand \Xx{X^x}
\newcommand \Ux{U^x}
\newcommand \Yx{Y^x}
\newcommand \Zx{Z^x}

\newcommand{\xxe}{\Xx_E}
\newcommand{\xxet}{\Xx_\tE}
\newcommand{\xxen}{X^{x^n}_E}
\newcommand{\xxetn}{X^{x^n}_\tE}

\newcommand{\ta}{\tilde{a} }
\newcommand{\ttheta}{\tilde{\theta}}
\newcommand{\tE}{{\tilde{E}}}

\newcommand{\meas}{{\cM}}

\def\1{\mathbf{1}}
\newcommand{\esssup}{\mathop{\mbox{\rm ess sup}}}

\newcommand \normr[1]{\left\|#1\right\|_{\rho}}

\newcommand \normp[1]{\left|#1\right|_p}
\newcommand \norm[1]{\left|#1\right|}
\newcommand \Tr{{\rm Tr}}

\makeatletter           
\newenvironment{breakablealgorithm}
  {
   \begin{center}
     \refstepcounter{algorithm}
     \hrule height.8pt depth0pt \kern2pt
     \renewcommand{\caption}[2][\relax]{
       {\raggedright\textbf{\ALG@name~\thealgorithm} ##2\par}%
       \ifx\relax##1\relax 
         \addcontentsline{loa}{algorithm}{\protect\numberline{\thealgorithm}##2}%
       \else 
         \addcontentsline{loa}{algorithm}{\protect\numberline{\thealgorithm}##1}%
       \fi
       \kern2pt\hrule\kern2pt
     }
  }{
     \kern2pt\hrule\relax
   \end{center}
  }
\makeatother
\usepackage{algorithm}
\usepackage{algpseudocode}

\title{Numerical approximation of Markovian BSDEs in infinite horizon and elliptic PDEs}

\usepackage{authblk}
\author[1]{Emmanuel Gobet\thanks{The first author research has benefited from the support of the Chaire "Risques Financiers" of Fondation du Risque, and of the Chaire "Stress
Test, RISK Management and Financial Steering" of the Ecole Polytechnique Foundation.}}
\author[2]{Adrien Richou\thanks{The second author research has benefited from the support of the ANR Project ReLISCoP (ANR-21-CE40-0001).}}
\author[3]{Charu Shardul\thanks{The third author has benefited from the support of Chaire ``Stress Test, RISK Management and Financial Steering" of the Ecole Polytechnique Foundation and of the ANR Project ReLISCoP (ANR-21-CE40-0001).}}

\affil[1]{LPSM, Sorbonne Université, 4 Place Jussieu, 75252, Paris Cedex 05, France}
\affil[2]{Université de Bordeaux, Institut de Mathématiques de Bordeaux, UMR CNRS 5251, 351 Cours de la Libération, 33405, Talence cedex, France}
\affil[3]{CMAP, Ecole Polytechnique, Route de Saclay, 91128, Palaiseau cedex, France}
{
    \makeatletter
    \renewcommand\AB@affilsepx{: \protect\Affilfont}
    \makeatother

    \affil[ ]{Email ids}

    \makeatletter
    \renewcommand\AB@affilsepx{, \protect\Affilfont}
    \makeatother

    \affil[1]{emmanuel.gobet@sorbonne-universite.fr}
    \affil[2]{adrien.richou@math.u-bordeaux.fr}
    \affil[3]{charu.shardul@polytechnique.edu}
}

\date{\today}

\begin{document}

\maketitle

\begin{abstract}
    {We study backward stochastic differential equations (BSDEs) in infinite horizon and {design} efficient numerical schemes for solving them. We establish a probabilistic representation of the solution of the BSDE using Malliavin derivative and prove results for contraction of a Picard scheme. We develop three numerical schemes, of which the first two are based on a fixed point argument using contraction, imposing additional assumptions compared to what is needed for existence and uniqueness of the solution. The first scheme is a space grid based approximation where we establish tight numerical error bounds using a growth truncation argument; it performs well in low dimensions but computational times increase exponentially with dimension. The second scheme uses neural network approximations for which we have proved a convergence result. Using neural networks alleviates the curse of dimensionality, giving good accuracy in very high dimensions. The third scheme also uses neural networks but does not directly rely on a contraction scheme, showcasing good performance even for larger $z$-Lipschitz dependence outside the empirical domain of contraction.}

    {{\sc Keywords:}  BSDEs in infinite horizon; Probabilistic numerical scheme; Elliptic PDEs; Feynman-Kac representation

    }
\end{abstract}

\section{Introduction}
\paragraph{Statement of the problem.}
In this work, we study the numerical approximation of the {solution to the} Markovian Backward Stochastic Differential Equation (BSDE) in infinite horizon
\begin{align}
  \label{eq:BSDEinfinitehorizon:intro}
  Y_t&=Y_T+\int_t^T f(X_s,Y_s,Z_s) \ds -\int_t^T Z_s \dW_s, \qquad \forall 0\leqslant t\leqslant T<+\infty
  \end{align}
  where the solution process $(Y,Z)$ takes values in some appropriate space and $X$ is the solution of the following SDE:
  \begin{align}
    \label{eq:SDE:intro}
    X_t = x + \int_0^t b(X_s) \ds +\int_0^t \sigma(X_s) \dW_s, \quad 0 \leqslant t.
   \end{align}
   Notations and assumptions are stated below.
    {Infinite horizon BSDEs provide a probabilistic representation of the solutions of semi-linear elliptic PDEs \cite{Darling-Pardoux-97,Royer-04}, which has been used to study their stability and homogenization \cite{Briand-Hu-98,Pardoux-99}, as well as their extension to infinite dimensional state spaces \cite{Fuhrman-Tessitore-04,Briand-Confortola-08-2}. In stochastic control, they characterize the value function of discounted control problems over an infinite horizon \cite{Fuhrman-Tessitore-04}. In mathematical finance, infinite horizon BSDEs arise in recursive utilities \cite{Duffie-Epstein-92}, in the construction of forward performance processes in stochastic factor models \cite{Liang-Zariphopoulou-17}, in infinite horizon investment-consumption problems \cite{Herdegen-Hobson-Jerome-23}, and in optimal liquidation problems \cite{cheng2026stochasticcontrolproblemsinfinite}.}
   
    {
   Since solving \eqref{eq:BSDEinfinitehorizon:intro} is equivalent to solving the elliptic PDE \eqref{eq:PDEellipticRd}, numerical approximation schemes for infinite horizon BSDEs may serve as a basis for probabilistic numerical methods for elliptic PDEs, which are potentially well suited to high-dimensional problems. 
In low dimension, deterministic methods for semi-linear elliptic PDEs are well established. Monotone finite difference schemes converge to the viscosity solution under the general monotonicity, stability and consistency conditions of \cite{Barles-Souganidis-91}; we refer to \cite{Barles-Jakobsen-05} for the construction of convergent schemes and error bounds and to \cite{Kushner-Dupuis-01} for the related Markov chain approximation method. Finite element methods \cite{Ciarlet-78,brenner2007mathematical} can also handle the nonlinearity efficiently, e.g. through two-grid techniques \cite{Xu-94,Xu-96}. 
These approaches rely on a mesh of the state space whose size grows exponentially with the dimension $d$, nevertheless sparse grids \cite{Bungartz-Griebel-04} can mitigate this curse of dimensionality.

From a computational point of view, the probabilistic approach developed in this paper has also advantages over deterministic methods. The fixed point map defined in \eqref{eq:contraction:intro} only involves the simulation of the diffusion $X$ started from given points, of the independent random times $E$ and $\tilde E$, and of the associated Malliavin weights, so that the same code can be used in any dimension $d$: only the size of the simulated arrays changes. Moreover, the method is naturally parallel. The Monte Carlo samples are independent, and, at each iteration, the evaluations of the fixed point map at different points are independent of each other. All these computations can therefore be vectorized and distributed with very little communication, which makes the method well suited to graphics processing units (GPUs). Admittedly, the first scheme still relies on a spatial grid, whose size grows exponentially with the dimension; it is mainly intended to provide a rigorous convergence analysis, even if the computations at the different grid nodes can be carried out in parallel. The other schemes, in contrast, rely on modern regression tools, such as neural networks, which allows us to design mesh-free schemes that are far more effective at tackling large-scale problems. In this case, the training of the networks and the simulation of the regression targets can directly benefit from the computing power of GPUs and from standard deep learning libraries.}

\paragraph{State of the art.}
While the literature concerning the simulation of (Markovian) BSDEs in the finite horizon setting {(i.e. $T$ is fixed and $Y_T=g(X_T)$ for some known function $g$)} is huge {(see \cite{10.1214/23-PS18} for a recent broad overview, and references therein)}, there are few results concerning the infinite horizon setting. In the recent paper \cite{Beck-Gonon-Jentzen-20}, the authors provide a full-history recursive multilevel Picard approximation scheme for solving semi-linear elliptic 
PDE. Nevertheless, their framework does not contain possible nonlinearity with respect to the gradient of the solution: the generator $f$ does not depend on $z$.
A neural network {(NN for short)} approximation scheme for elliptic PDEs is also provided in \cite{Kremsner-Steinicke-Szolgyenyi-20}, or in \cite{Nusken-Richter-21} for elliptic PDEs with a boundary, but without a study of the numerical errors. 
{A natural first attempt to tackle infinite horizon BSDEs would be to approximate the problem by a finite horizon BSDE with a large horizon $T$, and to apply an existing finite horizon numerical scheme. However, error bounds for finite horizon BSDEs typically depend exponentially on $T$, which would result in at best a logarithmic convergence rate.}

\paragraph{Our contributions.}
{The approach of this work is different and rather follows the strategy developed in \cite{Gobet-Richou-Szpruch-26}. 
First, we obtain a fixed point representation for the solution for the BSDE \eqref{eq:BSDEinfinitehorizon:intro} in infinite time horizon, which we {leverage} to design numerical schemes for solving such BSDEs.  Namely, in Proposition \ref{prop:representationBSDEinfinitehorizon}, we establish a representation result for the unique solution of the BSDE using the general assumptions for existence and uniqueness (see Proposition \ref{prop:existence:uniqueness:BSDE}) where we write $(Y, Z)$ as expectations by introducing a {family of} Malliavin weights $(\Ux_t)_{t\geqslant 0}$ (defined in \eqref{eq:malliavin:weight}). To be precise, due to the Markovian setting, we establish that the solution is given by  $(Y_t, Z_t) = (u(X_t), \bar{u}(X_t))$ for all $t\geqslant 0$, where $u$ and $\bar u$ solve the following fixed point representation 
\begin{align}
\begin{split}
 u(x):= \Yx_0 &= \E{\int_0^\infty  e^{-as}(f(\Xx_s,u(\Xx_s),\bar{u}(\Xx_s))+a u(\Xx_s)) \ds}\\
&=\frac{1}{\theta }\E{  (f(\Xx_E,u(\Xx_E),\bar{u}(\Xx_E))+a u(\Xx_E))e^{-(a-\theta)E} },\\
 \bar{u}(x) := \Zx_0  &= \E{\int_0^\infty  e^{-\ta s}(f(\Xx_s,u(\Xx_s),\bar{u}(\Xx_s))+ \ta u(\Xx_s))\Ux_s \ds}\\
&=\frac{\sqrt{\pi}}{\sqrt{\ttheta} }\E{  (f(\Xx_\tE,u(\Xx_\tE),\bar{u}(\Xx_\tE))+ \ta u(\Xx_\tE))\sqrt{\tE}e^{-(\ta-\ttheta)\tE}\Ux_\tE }
\end{split}
\label{eq:contraction:intro}
\end{align}
where $E\sim \cE(\theta)$ is an exponential random variable  and $\tE\sim \Gamma(1/2, \ttheta)$ has a gamma distribution. 
{Here $a>0$ and $\tilde a>0$ are free parameters that satisfy some technical conditions (see Proposition \ref{prop:representationBSDEinfinitehorizon} for precise statements). The choice of a gamma distribution in $\bar{u}$ instead of a simple exponential one is aimed at controlling the explosion of the Malliavin weight $\Ux_s$ as the time $s$ is small.}
 {This fixed point representation \eqref{eq:contraction:intro} naturally extends the representation obtained in \cite{Gobet-Richou-Szpruch-26} that corresponds, in a moral sense, to the second equation on $\bar{u}$ with $\ta=0$ and a generator $f$ that does not depend on $u$.}
\\
In Section \ref{subsec:contraction}, we study the contraction properties of this fixed point representation \eqref{eq:contraction:intro}, using some $L_p$ norms for $p \in [1,+\infty]$. Results for verifying the contraction are given in Section \ref{section:checking:BSDE} and Section \ref{section:checking:BSDE:bis}.  {A contraction property was already present in \cite{Gobet-Richou-Szpruch-26} within the $L_\infty$ framework, but the calculations leading to sufficient conditions for contraction are much simpler and more straightforward there. Furthermore, we had to extend these results to the $L_p$ framework with $p<+\infty$ in order to study the convergence of the second numerical scheme proposed in Section \ref{section:scheme}.} \\
In Section \ref{section:scheme}, leveraging  the representation result \eqref{eq:contraction:intro}, we propose three numerical schemes.
The first  scheme 
is based on a space grid-based approximation which relies on a Picard scheme. In Theorem \ref{thm:errornum:infinitehorizon}, we state tight bounds on the numerical error.  {This scheme and its convergence analysis follow the same ideas as in \cite{Gobet-Richou-Szpruch-26}, although the calculations are once again more involved.} The scheme is highly efficient in low dimensions but suffers from the curse of dimensionality, which has motivated us to design the second scheme with NNs which are known to alleviate such problems. The second scheme, in Subsection \ref{section:second_scheme}, is also based on a Picard scheme (see Algorithm \ref{ch4:alg:contractionNN}). 
In Proposition \ref{ch4:prop:contractionNN} we prove a convergence result for this scheme. However, for parameters close to and beyond the range for contraction, these two schemes show their shortcomings. The third scheme, in Subsection \ref{section:third_scheme}, is based on a slightly different approach that does not rely directly on a Picard iteration (see Algorithm \ref{ch4:alg:directNN}). 
Some {comparative} numerical experiments are provided in Section \ref{ch4:numexperiments}. {The proofs of some technical results are postponed to Section \ref{section:proofs}.}}

\paragraph{Notations.} {The following notations will be used throughout this work.}
\begin{description}
\item \emph{Vector, matrix.}
We use the same notation $\norm{\cdot}$ for the Euclidean norm of a vector $x\in \bR^p$ and for the matrix $2$-norm (i.e. subordinated to the Euclidean norm) of a matrix $A\in \bR^{p \times q}$; $\langle \cdot,\cdot \rangle$ denotes the Euclidean scalar product; 
$x^\top$ denotes the transpose of the vector $x$ and $\Tr(A)$ denotes the trace of $A$. For a vector $x \in \bR^p$ (resp. a matrix $A \in \bR^{p \times q}$) and $R \in \bR^+$, we denote $\lfloor x \rfloor_R$ the projection of $x$ (resp. $A$) on the Euclidean ball $\bar B(0,R)$ of $\bR^p$ (resp. $\bR^{p \times q}$). For $A = (c_1,...,c_q) \in \bR^{p \times q}$ and $x \in \bR^p$, we denote $\langle A,x \rangle$ the row-valued vector $(c_1^\top x,...,c_q^\top x)$. We denote $\bI_{d}$ the identity matrix or the identity function on $\mathbb{R}^d$.

\item \emph{Function.}
{$C^0(\bR^p,\bR^q)$ denotes the set of continuous functions from $\bR^p$ to $\bR^q$.
$C^1(\bR^p,\bR^q)$ (resp. $C^1_b(\bR^p,\bR^q)$) denotes the set of functions $f\in C^0(\bR^p,\bR^q)$ that are differentiable with a continuous (resp. continuous bounded) derivative. For a bounded function $f \in C^0(\bR^p,\bR^q)$ {or $f\in C^0(\bR^p,\bR^{q\times r})$, we denote $|f|_{\infty}:=\sup_{x \in \bR^p}|f(x)|$.}

For a function $f:=(f_i)_{1 \leqslant i \leqslant p} \in C^1_b(\bR^p,\bR^q)$, we denote $\nabla_x f$ the function $\bR^p \owns x \mapsto (\partial_{x_j} f_i)_{i,j} \in \bR^{q\times p}$. In particular, when $q=1$, $\nabla_x f$ is a row-vector valued function.}

For all $r \geqslant 0$ and $\eta>0$, we define $\rho_{r,\eta}(x):=1+\eta^{-1}|x|^r$ for all $x \in \mathbb{R}^p$, and $\rho_r:=\rho_{r,1}$.

\item \emph{Random variables and stochastic processes.} {We consider a filtered probability space $(\Omega, \cF, (\cF_t)_{t\geqslant 0},\bP)$ which supports a $d$-dimensional Brownian motion $W=(W^1,\dots,W^d)^\top$. The filtration $(\cF_t)_{t\geqslant 0}$ is the one generated by $W$ augmented by the $\bP$-null sets, so that the filtration satisfies the "usual conditions".\\
For $p\geqslant 1$, $\bL_p$ denotes the set of (scalar or vector-valued) random variables $X$ with finite norm $\normp{X}:=(\E{\norm{X}^p})^{1/p}<+\infty$. $\bL_\infty$ stands for the set of essentially bounded random variables.\\
$\sS^2_T$ is the set of vector valued adapted continuous processes $Y$ on $[0,T]$ such that 
$$\E{ \sup_{s \in [0,T]} |Y_s|^2}<+\infty.$$
$\sS^2_{loc}$ denotes the set of continuous processes $Y$ on $\bR^+$ such that  $(Y_t)_{t \leqslant T} \in \sS^2_T$, for all $T>0$.
$\sM^2_T$ is the set of real matrix-valued predictable processes $Z$ on $[0,T]$ such that
$$\E{ \int_0^T \norm{Z_s}^2 \ds }<+\infty.$$
$\sM^2_{loc}$ denotes the set of predictable processes $Z$ on $\bR^+$ such that $(Z_t)_{t \leqslant T} \in \sM^2_T$, for all $T>0$.}

\item \emph{Specific distributions.} 
{For $\ell>0$ and $a>0$, we denote $\Gamma(a,\ell)$ the gamma distribution with density (with respect to the Lebesgue measure)
$$x \mapsto \frac{\ell^a}{\Gamma(a)}x^{a-1}e^{-\ell x}\1_{(0,+\infty)}(x).$$
We recall the scaling property between distributions $\Gamma(a,\ell)\stackrel{d}=\ell^{-1}\Gamma(a,1)$.}
\end{description}

\paragraph{Declarations.} AI was used only after the first ArXiv version of this article to marginally improve the linguistic quality, identify typographical errors, refine the introduction and the references, and check for mathematical consistency. This enabled us to make minor corrections to certain quantitative bounds and a few inaccuracies in the proofs. All core concepts and mathematical proofs remain the original work of the authors, who retain sole accountability for the
accuracy of the paper.

\section{Analytical results}
\subsection{{Stochastic} model and value function}
{We consider the following Markovian BSDE in infinite horizon
\begin{align}
  \label{eq:BSDEinfinitehorizon}
  Y_t&=Y_T+\int_t^T f(X_s,Y_s,Z_s) \ds -\int_t^T Z_s \dW_s, \qquad \forall 0\leqslant t\leqslant T<+\infty
\end{align}
where the solution process $(Y,Z)$ takes values in the space $\sS^2_{loc} \times \sM^2_{loc}$ and $X$ is the solution of the following $d$-dimensional SDE:
\begin{align}
  \label{eq:SDE}
  X_t = x + \int_0^t b(X_s) \ds +\int_0^t \sigma(X_s) \dW_s, \quad 0 \leqslant t.
\end{align}
{Note that the processes $X$ and $W$ have the same dimension $d$.}
An existence and uniqueness result for the solution of \eqref{eq:BSDEinfinitehorizon} is recalled in Proposition \ref{prop:existence:uniqueness:BSDE}. Throughout the paper, we take the following assumptions on $f$, $b$ and $\sigma$.

\begin{assumption}
  \label{as:generatorf}
  $f(\cdot, y, z)$ is continuous $\forall y \in \bR^{d'}, z\in \bR^{d'\times d}$ and there exist constants $r \geqslant 0$, $M_f\geqslant 0$, $\mu\geqslant 0$, $K_{f,y}\geqslant 0$, $K_{f,z}\geqslant 0$, $K_b \geqslant 0$ and $K_{\sigma} \geqslant 0$ such that, $\forall x,x' \in \bR^d, y,y' \in \bR^{d'}, z,z' \in \bR^{d' \times d}$,
  \begin{enumerate}
  \item  $|f(x,y,z)-f(x,y',z')| \leqslant K_{f,y} |y-y'| + K_{f,z}\norm{z-z'},$
  \item $|f(x,y,z)| \leqslant M_f(1 + |x|^r+|y|+\norm{z}),$
  \item $\langle f(x,y,z)-f(x,y',z),y-y'\rangle \leqslant -\mu|y-y'|^2$ {(known as $(-\mu)$-monotonicity)},%
  \item $|b(x)-b(x')| \leqslant K_b |x-x'|$,
  \item $\norm{\sigma(x)-\sigma(x')} \leqslant K_\sigma |x-x'|$.
  \end{enumerate}
\end{assumption}
Let us remark that we can assume without restriction that {$0 
\leqslant\mu \leqslant K_{f,y}$.}

Since $b$ and $\sigma$ are Lipschitz, the SDE \eqref{eq:SDE} has a unique solution for any starting point $x \in \bR^d$: whenever necessary to emphasize the $x$-dependence of the solution, we shall denote it by $X^x$ and we denote $(\Yx,\Zx)$ the solution of the associated BSDE.}

{Let us recall some known existence and uniqueness results for this BSDE.
\begin{Proposition}
  \label{prop:existence:uniqueness:BSDE}
 Assume \eqref{as:generatorf}.
 \begin{enumerate}
  \item If there exists {a parameter $\nu \in (0,2\mu-K_{f,z}^2)$}
  such that, for all $x\in \bR^d$,
  \begin{align*}
\E{\int_0^{+\infty} e^{-\nu s}|\Xx_s|^{2r}\ds} <+\infty,
  \end{align*}
then, for all $x \in \bR^{d}$, BSDE \eqref{eq:BSDEinfinitehorizon} has a unique solution $(\Yx,\Zx)\in \mathscr{S}^2_{loc} \times \mathscr{M}^2_{loc}$ such that
\begin{align*}
\lim_{T \rightarrow + \infty} \E{e^{-\nu T} |\Yx_T|^2} = 0.
\end{align*}
Moreover this solution satisfies
\begin{align*}
\E{\int_0^{+\infty} e^{-\nu s}\norm{\Zx_s}^2 \ds}+\E{\sup_{0 \leqslant s < +\infty} e^{-\nu s} |\Yx_s|^2}  \leqslant C\left(1+\E{\int_0^{+\infty} e^{-\nu s}|\Xx_s|^{2r}\ds}\right),
\end{align*}
with a constant $C$ that does not depend on $x$.
 \item If $d'=1$ (i.e. $\Yx$ is scalar valued) and $r=0$ then, for all $x \in \bR^d$, the BSDE \eqref{eq:BSDEinfinitehorizon} has a solution such that $\Yx$ is a   bounded continuous process and $(\Yx,\Zx)$ satisfies
 \begin{align}
 \esssup_{\Omega, t \in \bR^+} |\Yx_t| +  \E{\int_0^{\infty} e^{-2\mu s}\norm{\Zx_s}^2 \ds}\leqslant C,
 \end{align}
 with a constant $C$ that does not depend on $x$.
 Moreover this solution is unique in the class of processes $(\Yx,\Zx)$ such that $\Yx$ is bounded continuous   and $\Zx \in \mathscr{M}^2_{loc}$.
 \end{enumerate}
\end{Proposition}
For proofs, we refer to  \cite[Theorem 5.57]{pard:rasc:14} for the multidimensional framework and   \cite[Lemma 3.1]{Briand-Hu-98} for the scalar setting.}

{Under {the same} assumptions, the BSDE solution  has a Markovian representation that gives a probabilistic representation of the following system of elliptic PDEs 
\begin{equation}
\label{eq:PDEellipticRd}
 \cL u(x) + f(x,u(x),\nabla_x u(x)\sigma(x))=0, \quad \forall x \in \bR^d,
\end{equation}
where $\cL$ is the generator of the semi-group associated to the SDE \eqref{eq:SDE}.

\begin{Proposition}
\label{prop:PDE}
  Assume  \eqref{as:generatorf} and technical assumptions of Proposition \ref{prop:existence:uniqueness:BSDE}.
Then the Markovian representation $Y_t= u(X_t)$ {holds,} where $u$ is a continuous function. Moreover, if the $i$-th coordinate of $f$ depends only on the $i$-th row of the matrix $z$ then $u$ is a viscosity solution of \eqref{eq:PDEellipticRd}.
\end{Proposition}
For the proof in the multidimensional setting we refer to   \cite[Theorem 5.74]{pard:rasc:14}. For the case $d'=1$, the proof can be easily adapted: in particular, the continuity follows directly from the stability result given by \cite[Theorem 3.4]{Briand-Hu-98}.}

\subsection{Time randomized Feynman-Kac representation}
{Now, we want to obtain a probabilistic representation for $\nabla u$ in terms of $u$ and $\nabla {u}$. In order to get it, we use the following extra assumptions on $b$ and $\sigma$.
\begin{assumption}
 \label{as:b:sigma}
 {The function $\sigma$ is bounded with $\sup_{x \in \bR^d} \norm{\sigma(x)} = M_\sigma<+\infty$, moreover $\sigma \in C^1_b(\bR^d, \bR^{d \times d})$ and $b \in C^1_b(\bR^d, \bR^{d})$. 
 }
Furthermore, for all $x$, $\sigma(x)$ is invertible and $\sigma^{-1}$ is bounded: $\sup_{x \in \bR^d} \norm{\sigma^{-1}(x)} = M_{\sigma^{-1}}<+\infty$. 
\end{assumption}
Under  \eqref{as:b:sigma}, $(\Xx_t)_{t \geqslant 0}$ is differentiable with respect to $x$, recalling e.g. \cite{kuni:97}. We denote $(\nabla_x \Xx_t)_{t \geqslant 0}$ its {tangent} process and we define
$(\Ux_t)_{t\geqslant 0}$, a Malliavin weight (row vector valued) given by
\begin{align}
\Ux_t = \frac{1}{t}\int_0^t \langle \sigma^{-1}(\Xx_s)\nabla_x \Xx_s , \dW_s\rangle\ \sigma(x) = (U^{x,j}_t)_{1 \leqslant j \leqslant d}.
\label{eq:malliavin:weight}
 \end{align}

We also introduce some integrability assumptions.
\begin{assumption}
 \label{as:int:X:nablaX}
~
\begin{enumerate}
    \item There exists {$\nu\in (0,2\mu-{K_{f,z}^2}(\1_{d'>1}\vee\1_{r>0}))$} 
and $C>0$ such that, for all $x \in \bR^d$
\begin{equation}
\label{prop:representationBSDEinfinitehorizon:hyp:rho}
 \E{\int_0^{+\infty} e^{-\nu s}|\Xx_s|^{2r}\ds} \leqslant C(1+|x|^{2r}).
\end{equation}
\item 
We set a parameter $a> \mu-\frac{K_{f,z}^2}{2}(\1_{d'>1}\vee\1_{r>0})>0$.
\item There exists a parameter $\ta \geqslant  2\mu-K_{f,z}^2(\1_{d'>1}\vee\1_{r>0})>0$ such that for any {bounded}  set $K \subset \bR^d$ we have
\begin{align}
 \label{prop:representationBSDEinfinitehorizon:hyp:b}
 &\int_0^{+\infty} e^{-\ta s} \sup_{x \in K} \Exp{|\Xx_s|^{2r}} \ds +
 \sup_{x \in  {\mathbb{R}^d}} \Exp{\int_0^{+\infty} e^{-\ta s} \norm{\nabla_x \Xx_s}^2 \ds} <+\infty.
\end{align}
\end{enumerate} 
\end{assumption}
 
\begin{Proposition}
\label{prop:representationBSDEinfinitehorizon}
{Assume  \eqref{as:generatorf}, \eqref{as:b:sigma} and \eqref{as:int:X:nablaX}.}
Then there exists a function $u \in {C}^1(\bR^d,\bR^{d'})$ such that, for all $x\in \bR^d$, $\Yx=u(\Xx)$, $\nabla_x u({\Xx})\sigma(\Xx):=\bar{u}(\Xx)$ is a continuous version of $\Zx$ and we have the following probabilistic representation for $(u,\bar{u})$: for all $x\in \bR^d$,
\begin{align}
 u(x):= \Yx_0 &= \E{\int_0^\infty  e^{-as}(f(\Xx_s,u(\Xx_s),\bar{u}(\Xx_s))+a u(\Xx_s)) \ds}\\
&=\frac{1}{\theta }\E{  (f(\Xx_E,u(\Xx_E),\bar{u}(\Xx_E))+a u(\Xx_E))e^{-(a-\theta)E} },
\label{prop:eq:u}\\
 \bar{u}(x) := \Zx_0  &= \E{\int_0^\infty  e^{-\ta s}(f(\Xx_s,u(\Xx_s),\bar{u}(\Xx_s))+ \ta u(\Xx_s))\Ux_s \ds}\\
&=\frac{\sqrt{\pi}}{\sqrt{\ttheta} }\E{  (f(\Xx_\tE,u(\Xx_\tE),\bar{u}(\Xx_\tE))+ \ta u(\Xx_\tE))\sqrt{\tE}e^{-(\ta-\ttheta)\tE}\Ux_\tE }
\label{prop:eq:baru}
\end{align}
where $\theta >0$, $\ttheta >0$, $E\sim\cE(\theta)$ is independent of $W$ and $\tE \sim \Gamma(1/2,\ttheta)$ is independent of $W$. 
Lastly, we have the following growth:
\begin{align}
\label{prop:representationBSDEinfinitehorizon:growth}
 |u(x)| +  \norm{\bar{u}(x)} \leqslant C(1+|x|^{r}),\quad \forall x \in \bR^d.
\end{align}
\end{Proposition}

The proof of Proposition \ref{prop:representationBSDEinfinitehorizon} is postponed to Section \ref{proof:prop:representationBSDEinfinitehorizon}.

\smallskip
\begin{Remark}
{Here, $a$, $\ta$, $\theta$ and $\ttheta$ are free parameters (up to constraints given in  \eqref{as:int:X:nablaX} and Proposition \ref{prop:representationBSDEinfinitehorizon}), whose values will be adjusted later for the efficiency of numerical schemes.}
We introduce extra parameters $\theta$ and $\ttheta$ due to the theoretical study of the numerical error where we need sometimes $\theta \neq a$ and/or $\ttheta \neq {\ta}$. Moreover, we take $\tE \sim \Gamma(1/2,\ttheta)$ to control the explosion of order $t^{-1/2}$ in $0$ of $\Ux_t$.
\end{Remark}
\smallskip

\begin{Remark}
\label{rem:boundestimatesX}
 By using the Lipschitz assumption on $b$ and $\sigma$, and the boundedness on $\sigma$, standard estimates on $\Xx$ give us: for all $\varepsilon>0$ and $p\geqslant 1$,
 \begin{align*}
  \Exp{|\Xx_s|^{2p}} \leqslant C_{\varepsilon}(1+ |x|^{2p}) e^{2p(K_b + \varepsilon) s}, \quad \Exp{\norm{\nabla_x \Xx_s}^{2p}} \leqslant C e^{(2pK_b+p(2p-1)K_{\sigma}^2)s}, \quad \forall s\geqslant 0, x\in \bR^d.
 \end{align*}
Thus \eqref{prop:representationBSDEinfinitehorizon:hyp:rho} is satisfied as soon as $2rK_b < 2\mu-{K_{f,z}^2}(\1_{d'>1}\vee\1_{r>0})$ and \eqref{prop:representationBSDEinfinitehorizon:hyp:b} is satisfied as soon as $\ta >(2K_b+K_{\sigma}^2)\vee(2rK_b)$. It should be noted, however, that these constraints can be significantly reduced on a case-by-case basis. For example, under certain conditions of dissipativity on $b$ and $\sigma$, it may hold that \eqref{prop:representationBSDEinfinitehorizon:hyp:rho} and \eqref{prop:representationBSDEinfinitehorizon:hyp:b} are true for all $\nu>0$ and $\tilde{a}>0$: in this case, the only remaining condition is $0 < 2\mu-{K_{f,z}^2}(\1_{d'>1}\vee\1_{r>0})$.
\end{Remark}
}

\subsection{Contraction properties of the fixed point equation }\label{subsec:contraction}
{We have established that $(u,\bar{u})$ solves equations \eqref{prop:eq:u} and \eqref{prop:eq:baru} that can be seen as some fixed point equation. More precisely we define a map $\Phi$ such that, for all measurable functions $w={(w^1,w^2)}:\bR^d \to \bR^{d'}\times \bR^{d'\times d}$, $\Phi(w)$ is a function from $\bR^d$ to $\bR^{d'}\times \bR^{d'\times d}$ denoted by $(\Phi^1(w), \Phi^2(w))$, given as follows for all $x\in \bR^d$,
\begin{align}\label{ch4:phi}
\Phi(w)(x) = & \Bigg(\frac{1}{\theta }\E{  (f(\xxe, w^1(\xxe), w^2(\xxe)) + a w^1(\xxe))e^{-(a-\theta)E} },\\
& \quad \sqrt{\frac{\pi}{{\ttheta} }} \E{  \left(f(\Xx_\tE, w^1(\Xx_\tE), w^2(\Xx_\tE)) + \ta w^1(\Xx_\tE)\right)\sqrt{\tE}e^{-(\ta-\ttheta)\tE} \Ux_\tE }\Bigg).
\end{align}
{The continuity of the $\Phi$-image functions is established in Lemma \ref{lem:boundcont} below.}
Then, equations \eqref{prop:eq:u} and \eqref{prop:eq:baru} correspond to a fixed point of this map, $\Phi(v) = v$. Next, we study the contraction property of $\Phi$. First, we have to define a suitable norm on $C^0\left(\bR^d, \bR^{d'}\times \bR^{d'\times d}\right)$. Let ${\rho}:\bR^+ \to [1,+\infty)$ be a positive non decreasing weight function, such that $\lim_{x \rightarrow + \infty} {\rho}(x)=+\infty$, with a growth at least polynomial of degree $r$ and at most exponential: there exists $C>0$ such that
\begin{align} 
1 + C^{-1}x^r \leqslant {\rho}(x) \leqslant C e^{Cx}. 
\label{eq:growth:rho}
\end{align}
{Recall that the parameter $r$ is related to the growth condition of $f$ in  \eqref{as:generatorf}.}
Some standard choices will correspond to polynomial or exponential weighting, i.e. ${\rho}(x)=(1+x^\alpha)$ or ${\rho}(x)=\exp(\alpha x)$ for some parameter $\alpha\geqslant 0$. 
By a slight abuse of notation, we also define ${\rho}$ on $\bR^d$ by setting ${\rho}(x)={\rho}(|x|)$ for all $x \in \bR^d$.
The ${\rho}$-norm of a function $v:\bR^d \to \bR^{d'}\times \bR^{d'\times d}$ (or $\bR^{d'}$ or $\bR^{d'\times d}$) is defined by
\begin{align}
\label{eq:rho:norm}
\normr{{v}}:=\sup_{x\in \bR^d}\frac{\norm{{v}(x)}}{{\rho}(x)}.
\end{align}
Let us also define the following constants, which appear in the upcoming computations,
\begin{equation}\label{eq:condition:normA:rho}
\begin{aligned}
c_{\infty,\eqref{eq:condition:normA:rho}} &:= \sup_{x\in \bR^d} \E{ \frac{e^{-(a-\theta)E}}{\theta}\frac{{\rho}(X_E^x)}{ {\rho}(x)}},\\
\tilde{c}_{\infty,\eqref{eq:condition:normA:rho}} &:= \sup_{x \in \bR^d} \E{ \frac{\sqrt{\pi}\sqrt{\tE}e^{-(\ta-\ttheta)\tE}\norm{\Ux_{\tilde E}}}{\sqrt{\ttheta}}\frac{{\rho}(X_\tE^x)}{ {\rho}(x)}}.
\end{aligned}
\end{equation}
Let us also introduce the following assumption.
\begin{assumption}\label{as:suprho}
There exists $\epsilon>0$ such that for any {bounded} set $K\subset \bR^d$,
\begin{equation}
\label{eq:as:suprho}
\begin{aligned}
\sup_{x\in K} \E{\left(e^{-(a-\theta)E}{\rho}(X_E^x)\right)^{1+\epsilon}}<\infty \text{ and } \sup_{x\in K} \E{\left(\sqrt{\tE}e^{-(\ta-\ttheta)\tE}\norm{\Ux_{\tilde E}}{\rho}(X_{\tilde E}^x)\right)^{1+\epsilon}}<\infty.
\end{aligned}
\end{equation}
\end{assumption}}

{We state two preliminary lemmas before the study of the contraction of $\Phi$, {the proofs are postponed to Subsection \ref{proof:2lemmas}.}
\begin{Lemma}
  \label{lem:newLipcste}
  Let $h: \bR^{d'} \rightarrow \bR^{d'}$ a $K$-Lipschitz function that is $(-\mu)$-monotone with $0<\mu\leqslant K$. {Let $a\geqslant 0$,} then $h+a\bI_{d'}$ is $\sqrt{K^2-2\mu a+a^2}$-Lipschitz. 
 \end{Lemma}
\smallskip

\begin{Lemma}
\label{lem:boundcont} Assume  \eqref{as:generatorf}, \eqref{as:b:sigma}, \eqref{as:int:X:nablaX} and \eqref{as:suprho}. Assume additionally that   
 $a>\theta$, $\ta>\ttheta$, and that $c_{\infty,\eqref{eq:condition:normA:rho}}\lor \tilde{c}_{\infty,\eqref{eq:condition:normA:rho}} < \infty$. Let  $w = (w^1, w^2) \in C^0\left(\bR^d, \bR^{d'}\times \bR^{d'\times d}\right)$ such that $\normr{{w}}<+\infty$: then, $\normr{\Phi({w})} <+\infty$ and $\Phi({w})(\cdot)\in C^0\left(\bR^d, \bR^{d'}\times \bR^{d'\times d}\right)$.
\end{Lemma}

Then we have the following estimates, {that are our main results regarding the contraction property of $\Phi$.}}

\begin{Proposition}[$L_p$-estimate, $1< p \leqslant +\infty$] ~
\label{pr:Lpestimate}
Assume  \eqref{as:generatorf}, \eqref{as:b:sigma} and \eqref{as:int:X:nablaX}.
\begin{description} 
\item [$\rhd$ The case $p=+\infty$.] Assume \eqref{as:suprho}, $a>\theta$, $\ta>\ttheta$ and $c_{\infty,\eqref{eq:condition:normA:rho}} \vee \tilde{c}_{\infty,\eqref{eq:condition:normA:rho}} < +\infty$. Then, for any functions $w_1$, $w_2$ in $C^0(\bR^d, \bR^{d'}\times \bR^{d'\times d})$, we have
\begin{align*}
\normr{\Phi(w_1)-\Phi(w_2)}\leqslant  \kappa_\infty \normr{w_1-w_2},
\end{align*}
where 
$$\kappa_\infty = \sqrt{ c_{\infty,\eqref{eq:condition:normA:rho}}^2   {\left(K_{f,y}^2 - 2\mu a + a^2+K_{f,z}^2\right)} + \tilde{c}_{\infty,\eqref{eq:condition:normA:rho}}^2   {\left(K_{f,y}^2 - 2\mu \ta + \ta^2+K_{f,z}^2\right) }}.$$
\item [$\rhd$ The case $p\in (1,+\infty)$.] 
Assume $ap>\theta$. Assume that there exist finite positive constants $c_{p,\eqref{eq:condition:stability:rho}}$, $\tilde{c}_{p,\eqref{eq:condition:stability:rho}}$, $\tilde{c}_{p,\eqref{eq:condition:stability:rho:bis}}$ {and a probability measure $\meas$ on $\bR^d$}, satisfying $\int_{\mathbb{R}^d} |x|^{pr} \meas(\dx)<+\infty$, such that for any measurable function $w:\bR^d \to \bR^{d'}\times \bR^{d'\times d}$ with a polynomial growth $r$ we have
\begin{equation}
\begin{aligned}
\label{eq:condition:stability:rho}
\left(\int_{\bR^d} \E{ {\norm{w(\xxe)} }^p} \meas(\dx) \right)^{1/p} &\leqslant 
c_{p,\eqref{eq:condition:stability:rho}} \left(\int_{\bR^d} {\norm{w(x)}}^p \meas(\dx)\right)^{1/p}<+\infty, \\
\left(\int_{\bR^d} \E{ {\norm{w(\xxet)} }^p} \meas(\dx) \right)^{1/p} &\leqslant 
\tilde{c}_{p,\eqref{eq:condition:stability:rho}} \left(\int_{\bR^d} {\norm{w(x)}}^p \meas(\dx)\right)^{1/p}<+\infty,
\end{aligned}
\end{equation}
and 
\begin{equation}
  \label{eq:condition:stability:rho:bis}
\sup_{x \in \bR^d}\Exp{\left|\sqrt{\frac{\pi}{\ttheta}}U_\tE^x \sqrt{\tE}e^{-(\ta-\ttheta)\tE}\right|^{\frac{p}{p-1}}}^{\frac{p-1}{p}} = \tilde{c}_{p,\eqref{eq:condition:stability:rho:bis}}<+\infty.
\end{equation}

Then, for any measurable functions $w_1,w_2 : \bR^d \to \bR^{d'}\times \bR^{d'\times d}$ with a polynomial growth $r$, we have
\begin{align}
\left(\int_{\bR^d} { \norm{\Phi(w_1)(x)-\Phi(w_2)(x)}} ^p \meas(\dx)\right)^\frac{1}{p } &\leqslant 
  \kappa_p \left(\int_{\bR^d} { \norm{w_1(x)-w_2(x)}}^p \meas(\dx)\right)^\frac{1}{p},
\end{align}
where
\begin{align}
\kappa_{p,1} = & \Bigg( \dfrac{1}{\theta}\left(\dfrac{p-1}{ap-\theta}\right)^{p-1} \Bigg)^{\frac1p} c_{p,\eqref{eq:condition:stability:rho}} \left(K_{f,y}^2 - 2\mu a + a^2+K_{f,z}^2\right)^{1/2},\\
\kappa_{p,2} = & \tilde{c}_{p,\eqref{eq:condition:stability:rho:bis}} \tilde{c}_{p,\eqref{eq:condition:stability:rho}}  {\left(K_{f,y}^2 - 2\mu \ta + \ta^2+K_{f,z}^2\right)^{1/2}} 
\end{align}
and 
\begin{equation}\label{eq:kappa_p}
    \kappa_p = \begin{cases}
        (\kappa_{p,1}^p +\kappa_{p,2}^p)^{1/p} & \text{if } p \in (1,2]\\
        \sqrt{\kappa_{p,1}^2+\kappa_{p,2}^2} & \text{otherwise.}
    \end{cases}
\end{equation}
\end{description}
\end{Proposition}

 {It should be noted that only the $L_\infty$ framework was studied in \cite{Gobet-Richou-Szpruch-26}. This is because these estimates are used to study the convergence of the first numerical scheme, which is based on spatial discretization. See in particular Theorem \ref{thm:errornum:infinitehorizon}. In contrast, the study of the second numerical scheme -- which relies on neural network approximations and is not covered in \cite{Gobet-Richou-Szpruch-26} -- requires contraction results in $L_2$, which justifies this broader study. See in particular Proposition \ref{ch4:prop:contractionNN}.}

\begin{proof}~

{$\rhd$ Let $p=+\infty$. We can assume that the functions $w_1,\ w_2$ and $w_1-w_2$ have a finite ${\rho}$-norm, otherwise the above inequality is trivial. From Lemma \ref{lem:boundcont}, $\Phi(w_1)$ and $\Phi(w_2)$ have finite ${\rho}$-norms and {are continuous}.
 We denote $w_1=(w_1^1,w_1^2)$ and $w_2 = (w_2^1,w_2^2)$. Using Lemma \ref{lem:newLipcste}, we have for any $x\in \bR^d$,}
{
\begin{align}
& \frac{ |\Phi^1(w_1)(x) - \Phi^1(w_2)(x)|}{{\rho}(x)} \\
 \leqslant& \dfrac{1}{\theta {\rho}(x)}\E{e^{-(a-\theta)E}\norm{f(\xxe, w_1^1(\xxe), w_1^2(\xxe)) - f(\xxe, w_2^1(\xxe), w_2^2(\xxe)) + a \big(w_1^1(\xxe) - w_2^1(\xxe)\big)}}\nonumber\\
 \leqslant&  ({K_{f,y}^2 -2\mu a + a^2})^{1/2} \E{\dfrac{|w_1^1(\xxe) - w_2^1(\xxe)|}{{\rho}(\norm{\xxe})}\dfrac{{\rho}(\norm{\xxe})}{{\rho}(x)}\dfrac{e^{-(a-\theta)E}}{\theta}}\\
& + K_{f,z}\E{\dfrac{\norm{w_1^2(\xxe) - w_2^2(\xxe)}}{{\rho}(\norm{\xxe})}\dfrac{{\rho}(\norm{\xxe})}{{\rho}(x)} \dfrac{e^{-(a-\theta)E}}{\theta} }
\end{align}
Then, Cauchy-Schwarz inequality gives us 
\begin{align}
\frac{ |\Phi^1(w_1)(x) - \Phi^1(w_2)(x)|}{{\rho}(x)} 
&\leqslant  \left(K_{f,y}^2 - 2\mu a + a^2+K_{f,z}^2\right)^{1/2}\E{\dfrac{\norm{w_1(\xxe) - w_2(\xxe)}}{{\rho}(\norm{\xxe})}\dfrac{{\rho}(\norm{\xxe})}{{\rho}(x)} \dfrac{e^{-(a-\theta)E}}{\theta} }\\
& \leqslant c_{\infty,\eqref{eq:condition:normA:rho}}  
 {\left(K_{f,y}^2 - 2\mu a + a^2+K_{f,z}^2\right)^{1/2}}  \normr{w_1 - w_2} .
\end{align}}
{Following same steps as above for $\Phi^2$, we obtain,
{\begin{align}
\frac{ \norm{\Phi^2(w_1)(x) - \Phi^2(w_2)(x)}}{{\rho}(x)} \leqslant \tilde{c}_{\infty,\eqref{eq:condition:normA:rho}}
 {\left(K_{f,y}^2 - 2\mu \ta + \ta^2+K_{f,z}^2\right)^{1/2}} \normr{w_1 - w_2}.
\end{align}}
Therefore, the advertised result follows for $p=+\infty$.}

$\rhd$ Next, for $p\in (1, +\infty)$, we can use Lemma \ref{lem:newLipcste} and H\"older inequality to obtain for the first component,
\begin{align*}
&\int_{\bR^d} { |\Phi^1(w_1)(x)-\Phi^1(w_2)(x)|} ^p \meas(\dx)\\
 \leqslant & \dfrac{1}{\theta^p}\int_{\bR^d}  
 {\left(K_{f,y}^2 - 2\mu a + a^2+K_{f,z}^2\right)^{p/2}} \E{\norm{w_1(\xxe)-w_2(\xxe)}e^{-(a-\theta)E}}^p  \meas(\dx)\\
 \leqslant & \E{\left(\frac{e^{-(a-\theta)E}}{\theta}\right)^{\frac{p}{p-1}}}^{p-1}\int_{\bR^d}  
 {\left(K_{f,y}^2 - 2\mu a + a^2+K_{f,z}^2\right)^{p/2}} \E{\norm{w_1(\xxe)-w_2(\xxe)}^p}  \meas(\dx)\\
 \leqslant & \dfrac{1}{\theta}\left(\dfrac{p-1}{ap-\theta}\right)^{p-1}c_{p,\eqref{eq:condition:stability:rho}}^p 
 {\left(K_{f,y}^2 - 2\mu a + a^2+K_{f,z}^2\right)^{p/2}} \int_{\bR^d}   \norm{w_1(x)-w_2(x)}^p  \meas(\dx)\\
=& {\kappa_{p,1}^p \int_{\bR^d}  \norm{w_1(x)-w_2(x)}^p  \meas(\dx).}
\end{align*}

Next, similarly for the second component,
\begin{align*}
&\int_{\bR^d} { \norm{\Phi^2(w_1)(x)-\Phi^2(w_2)(x)}} ^p \meas(\dx)\\
\leq&  \left(\sqrt{\dfrac{\pi}{\ttheta}}\right)^p  {\left(K_{f,y}^2 - 2\mu \ta + \ta^2+K_{f,z}^2\right)^{p/2}}
\int_{\bR^d}   \E{\norm{w_1(\xxet)-w_2(\xxet)} \norm{\Ux_\tE \sqrt{\tE}e^{-(\ta - \ttheta)\tE} } }^p  \meas(\dx)\\
\leqslant & \left(\sqrt{\dfrac{\pi}{\ttheta}}\right)^p  {\left(K_{f,y}^2 - 2\mu \ta + \ta^2+K_{f,z}^2\right)^{p/2}} \times \\
& \int_{\bR^d}   \E{{\norm{w_1(\xxet) - w_2(\xxet)} }^p} \E{\norm{\Ux_\tE\sqrt{\tE}e^{-(\ta - \ttheta)\tE}}^\frac{p}{p-1}}^{p-1}  \meas(\dx)\\
\leqslant & 
  {\left(K_{f,y}^2 - 2\mu \ta + \ta^2+K_{f,z}^2\right)^{p/2}} \tilde{c}_{p,\eqref{eq:condition:stability:rho:bis}}^p \tilde{c}_{p,\eqref{eq:condition:stability:rho}}^p \int_{\bR^d}   {\norm{w_1(x) - w_2(x)}^p}  \meas(\dx)\\
=& {\kappa_{p,2}^p \int_{\bR^d}  \norm{w_1(x)-w_2(x)}^p  \meas(\dx).}
\end{align*}
where we have used H\"older's inequality. 
 {Thus, when $1< p \leqslant 2$, since $t \mapsto t^{p/2}$ is subadditive on $\bR^+$ we obtain 
\begin{align}
\int_{\bR^d} { \norm{\Phi(w_1)(x)-\Phi(w_2)(x)}} ^p \meas(\dx) \leqslant& \int_{\bR^d} { \norm{\Phi^1(w_1)(x)-\Phi^1(w_2)(x)}} ^p \meas(\dx)\\
&+ \int_{\bR^d} { \norm{\Phi^2(w_1)(x)-\Phi^2(w_2)(x)}} ^p \meas(\dx)
\end{align}
and the announced result follows from previous estimates. When $p >2$, the triangle inequality in $L^{p/2}$ gives us
\begin{align}
\left(\int_{\bR^d} { \norm{\Phi(w_1)(x)-\Phi(w_2)(x)}} ^p \meas(\dx)\right)^\frac{2}{p } \leqslant& \left(\int_{\bR^d} { \norm{\Phi^1(w_1)(x)-\Phi^1(w_2)(x)}} ^p \meas(\dx)\right)^\frac{2}{p }\\
&+ \left(\int_{\bR^d} { \norm{\Phi^2(w_1)(x)-\Phi^2(w_2)(x)}} ^p \meas(\dx)\right)^\frac{2}{p }
\end{align}
and once again the announced result follows from previous estimates.
}
\end{proof}

Similarly, when there is no dependence on $z$ in the generator one can prove a special case of this result, stated below.
\begin{Proposition}\label{ch4:prop:genY}
When the generator of the BSDE \eqref{eq:BSDEinfinitehorizon} does not depend on the $Z$ process, that is $f(x, y, z) = f(x, y)$, we can restrict our map $\Phi$ to its first component $\Phi^1$ and the conclusions of Proposition \ref{pr:Lpestimate} remain true for $\Phi^1$ under the same assumptions and with the following contraction constants:
\begin{align*}
\kappa_\infty &= c_{\infty,\eqref{eq:condition:normA:rho}}  \sqrt{K_{f,y}^2 - 2\mu a + a^2},\\
\kappa _p &= \left(\dfrac{1}{\theta}\right)^{\frac1p}\left(\dfrac{p-1}{ap-\theta}\right)^{\frac{p-1}{p}}c_{p,\eqref{eq:condition:stability:rho} } \sqrt{K_{f,y}^2 - 2\mu a + a^2}\quad\text{for}\quad p \in (1,+\infty).
\end{align*}
\end{Proposition}

Now a straightforward application of the Banach fixed-point theorem yields the following result.
\begin{Corollary} 
\label{cor:contraction}
Let $p\in (1,+\infty]$. 
Assume  \eqref{as:generatorf}, \eqref{as:b:sigma} and \eqref{as:int:X:nablaX}.
\begin{description} 
\item [(case $p=+\infty$)]: Assume \eqref{as:suprho}, $a>\theta$, $\ta>\ttheta$,  $\kappa_\infty < 1$ and $\normr{\Phi(0)}<+\infty$,
then there is a unique continuous function $v: \bR^d \to \bR^{d'}\times \bR^{d'\times d}$ with  $\displaystyle\sup_{x\in \bR^d}\dfrac{\norm{v(x)}}{{\rho}(x) }<+\infty$  satisfying  $v=\Phi(v),$
i.e. solving \eqref{prop:eq:u} and \eqref{prop:eq:baru}. Moreover the solution $v$ satisfies the following bound: $\displaystyle\normr{v}\leqslant \frac{1}{1- \kappa_\infty} \normr{\Phi(0)}$.
\item [(case $p \in (1,+\infty)$)]:  Assume that $ap>\theta$ and there exists a probability measure $\meas$ on $\bR^d$ satisfying $\int_{\mathbb{R}^d} |x|^{pr} \meas(\dx)<+\infty$ and such that \eqref{eq:condition:stability:rho} holds. Assume also \eqref{eq:condition:stability:rho:bis} holds, $\kappa_p < 1$ and $\left(\displaystyle\int_{\bR^d} |\Phi(0)(x)|^p \meas(\dx)\right)^\frac{1}{p } < +\infty$.
\end{description}
Then, there is a unique (up to an $\meas$-negligible set) measurable function $v: \bR^d \to \bR^{d'}\times \bR^{d'\times d}$ with   $\left(\displaystyle\int_{\bR^d}\norm{v(x)}^p \meas(\dx)\right)^\frac{1}{p }<+\infty$ satisfying  $v=\Phi(v)$, i.e. solving \eqref{prop:eq:u} and \eqref{prop:eq:baru}.
Moreover the solution $v$ satisfies the following bound: $\displaystyle\left(\int_{\bR^d}\norm{v(x)}^p \meas(\dx) \right)^\frac{1}{p }\leqslant \frac{1}{1- \kappa_p} \left(\int_{\bR^d}\norm{\Phi(0)(x)}^p \meas(\dx)\right)^\frac{1}{p }$.
\end{Corollary}

Next, in order to check in practice assumptions of Corollary \ref{cor:contraction} we must obtain some bounds on $c_{\infty,\eqref{eq:condition:normA:rho}}$, $\tilde{c}_{\infty,\eqref{eq:condition:normA:rho}}$, $c_{p,\eqref{eq:condition:stability:rho}}$, $\tilde{c}_{p,\eqref{eq:condition:stability:rho}}$ and $\tilde{c}_{p,\eqref{eq:condition:stability:rho:bis}}$. 
{This is the purpose of the following subsections. The bounds obtained can likely be improved on a case-by-case basis; our aim here is simply to show that contraction holds for certain parameter regimes, establishing that Corollary \ref{cor:contraction} is non-vacuous.}

\subsection{Sufficient assumptions to get $\kappa_\infty<1$}
\label{section:checking:BSDE}
{The aim of this subsection is to identify concrete sets of assumptions on the generator and the SDE under which $\Phi$ is a contraction for $p=+\infty$, so that Corollary \ref{cor:contraction} applies. This contraction property will then be the key ingredient for designing the first numerical scheme based on spatial discretization with convergence guarantees in the next section. 
These assumptions are often stronger than those requested in Proposition \ref{prop:representationBSDEinfinitehorizon} to get existence and uniqueness of a solution to BSDE \eqref{eq:BSDEinfinitehorizon}. In particular, they depend on parameters $a$, $\ta$, $\theta$, $\ttheta$ and they give some indications on how to choose optimally these parameters to ensure the contraction property of $\Phi$. 
 {The following computations follow the same line of reasoning as in the proof of \cite[Proposition 4]{Gobet-Richou-Szpruch-26}, but are just much more tedious.}}
We assume that  \eqref{as:generatorf} and \eqref{as:b:sigma} hold. Then, we recall and introduce following notations:
\begin{itemize}
 \item $b$ is $K_b$-Lipschitz and there exist two non-negative constants $M_b$ and $L_b$ such that 
 \begin{equation*}
  |b(x)| \leqslant M_b + L_b |x|, \quad \forall x \in \bR^d,
 \end{equation*}
\item $\sigma$ is $K_{\sigma}$-Lipschitz and
\begin{equation*}
 \sup_{x \in \bR^d} \norm{\sigma(x)} \leqslant M_{\sigma}, \quad \sup_{x \in \bR^d} \norm{\sigma^{-1}(x)} \leqslant M_{\sigma^{-1}}.
\end{equation*}
\end{itemize}

First, we establish the following elementary result:
\begin{Lemma}
\label{lem:ineq:boundmomentX} Assume  \eqref{as:generatorf} and \eqref{as:b:sigma}. 
For all $\varepsilon>0$ and $p \geqslant 2$, there exists $C_{\varepsilon}$ that only depends on $M_b,L_b,M_{\sigma},p$ and $\varepsilon$, such that for all $t,x$, we have
$$\Exp{|\Xx_t|^p} \leqslant (|x|^p+C_{\varepsilon}t)e^{(pL_b + \varepsilon)t}.$$
 \end{Lemma}
\begin{proof}
Applying It\^o formula to $|X_t^x|^p$ and Young's inequality, we get
\begin{align*}
 \Exp{|X_t^x|^p} &\leqslant |x|^p + \int_0^t p\Exp{|X_s^x|^{p-1}(M_b + L_b |X_s^x|)} \ds +\frac{p(p-1)}{2}\int_0^t d M_{\sigma}^2 \Exp{|X_s^x|^{p-2}}\ds\\
 &\leqslant |x|^p + C_{\varepsilon} t + \int_0^t (pL_b + \varepsilon) \Exp{|X_s^x|^{p}} \ds
\end{align*}
and we just have to use Gr\"onwall lemma to conclude.
\end{proof}

Due to {the growth condition} \eqref{prop:representationBSDEinfinitehorizon:growth}, it is natural to take a weight function with a growth at least equal to $1+|x|^r$. Since this weight function will have an impact on the error study  of our numerical scheme - to be precise, the error coming from the truncation of the domain, see Remark \ref{rem:comments-num-error} - we take here ${\rho}(x)=\rho_{r',\eta}(x)=1+\eta^{-1}|x|^{r'}$ with $r' \geqslant r$ and $\eta>0$.  Let us remark that  \eqref{as:suprho} holds as soon as $a-\theta$, $\ta -\ttheta$ are large enough and $\varepsilon$ (from \eqref{as:suprho}) small enough due to Remark \ref{rem:boundestimatesX}. Also, let us note that this choice implies $\normr{\Phi(0)}<+\infty$ as soon as $c_{\infty,\eqref{eq:condition:normA:rho}} <+\infty$ and $\tilde{c}_{\infty,\eqref{eq:condition:normA:rho}} <+\infty$.

\paragraph{When $f$ only depends on $y$: $f(x,y,z)=f(x,y)$.} 

We consider \eqref{prop:eq:u} without the component $\bar u$, i.e.
\begin{align}
 u(x)=\frac{1}{\theta }\E{  (f(\Xx_E,u(\Xx_E))+a u(\Xx_E))e^{-(a-\theta)E} }.
\label{prop:representationBSDEinfinitehorizon:eq:u:analysis}
\end{align}

\begin{Proposition}
\label{prop:contraction:BSDE:infinite:horizon} Assume  \eqref{as:generatorf} and \eqref{as:b:sigma}.
For all $\varepsilon>0$, $r'>0$ and $a>r'L_b+\varepsilon$, we have
 $$\kappa_\infty \leqslant \sqrt{K_{f,y}^2-2a\mu+a^2} \max\left\{ \frac{1}{a}+\frac{C_{\varepsilon}}{\eta(a-(r'L_b+\varepsilon))^{1+\frac{r'}{2} \wedge 1}} , \frac{1}{a-(r'L_b+\varepsilon)} \right\}$$
where $\kappa_{\infty}$ is defined in Proposition \ref{ch4:prop:genY} and with $C_{\varepsilon}$ that only depends on $M_b,L_b,M_{\sigma},r'$ and $\varepsilon$. If $r'=0$, we also have
$$\kappa_\infty \leqslant \frac{\sqrt{K_{f,y}^2-2a\mu+a^2}}{a}.$$
Finally, $\Phi$ is a contraction as soon as $\mu>r'L_b$ $a$ large enough and then $\eta$ large enough.
\end{Proposition}

The proof of Proposition \ref{prop:contraction:BSDE:infinite:horizon} is postponed to Section \ref{proof:2propositions}.

\begin{Remark}

\begin{itemize}
    \item As pointed out in Remark \ref{rem:boundestimatesX}, Assumption $\mu > rL_b$ is a sufficient condition required in Proposition \ref{prop:existence:uniqueness:BSDE} to get existence and uniqueness of a solution to the BSDE. If $\mu > rL_b$, then there exists $r'>r$ small enough such that $\mu > r'L_b$. Moreover, if $b$ has a sublinear growth, then we can take $L_b$ as close to $0$ as we want and then there is no upper constraint on $r' \geqslant r$. 
    \item In this special case, we do not need the smoothness of $b$, $\sigma$ and the invertibility of $\sigma$ since there is no Malliavin weight in \eqref{prop:eq:u}.
\end{itemize} 
\end{Remark}

\paragraph{When $f$ depends on $y$ and $z$.}
Due to \eqref{prop:representationBSDEinfinitehorizon:growth}, it is still natural to take the weight function ${\rho}(x)=\rho_{r'}(x)=1+|x|^{r'}$, for $x \in \bR^d$ and $r' \geqslant r$. 
\begin{Proposition}
\label{prop:condition:contraction:infinitehorizon} Assume  \eqref{as:generatorf} and \eqref{as:b:sigma}.
Let us assume that we have following growth estimates on $X$: there exist some constants $c_{1,\eqref{ass:momentXandnablaX}}$, $c_{2,\eqref{ass:momentXandnablaX}}$, $c_{3,\eqref{ass:momentXandnablaX}}$, $c_{4,\eqref{ass:momentXandnablaX}}$, $c_{5,\eqref{ass:momentXandnablaX}} \geqslant 0$ such that, for all $t\geqslant 0$,
\begin{equation}
\label{ass:momentXandnablaX}
\Exp{|X_t^x|^{2r'}}^{1/2} \leqslant (c_{1,\eqref{ass:momentXandnablaX}}+c_{2,\eqref{ass:momentXandnablaX}}|x|^{r'})e^{c_{3,\eqref{ass:momentXandnablaX}}t}, \quad \Exp{\norm{\nabla X_t^x}^{2}}^{1/2} \leqslant c_{4,\eqref{ass:momentXandnablaX}}e^{c_{5,\eqref{ass:momentXandnablaX}}t}.  
\end{equation}
Then, for $a>c_{3,\eqref{ass:momentXandnablaX}}$ and $\ta > c_{3,\eqref{ass:momentXandnablaX}}+c_{5,\eqref{ass:momentXandnablaX}}$ we have following upper-bounds:
$$c_{\infty,\eqref{eq:condition:normA:rho}} \leqslant \max \left\{\frac{1}{a} + \frac{c_{1,\eqref{ass:momentXandnablaX}}}{a-c_{3,\eqref{ass:momentXandnablaX}}} ,  \frac{c_{2,\eqref{ass:momentXandnablaX}}}{a-c_{3,\eqref{ass:momentXandnablaX}}} \right\}$$
and
\begin{align*}
  &\tilde{c}_{\infty,\eqref{eq:condition:normA:rho}}\\ \leq& M_{\sigma}M_{\sigma^{-1}} c_{4,\eqref{ass:momentXandnablaX}}\sqrt{\pi} \max \left\{ \frac{1}{\sqrt{\ta-c_{5,\eqref{ass:momentXandnablaX}}}} + \frac{c_{1,\eqref{ass:momentXandnablaX}}}{\sqrt{\ta - c_{5,\eqref{ass:momentXandnablaX}} - c_{3,\eqref{ass:momentXandnablaX}}}} , \frac{c_{2,\eqref{ass:momentXandnablaX}}}{\sqrt{\ta-c_{5,\eqref{ass:momentXandnablaX}}-c_{3,\eqref{ass:momentXandnablaX}}}}  \right\}.
  \end{align*}  
\end{Proposition}

The proof of Proposition \ref{prop:condition:contraction:infinitehorizon} is postponed to Section \ref{proof:2propositions}.

\begin{Remark}\label{ch4:rem:contraction}
  Let us set $a=\ta=K_{f,y}$ and denote $\delta = K_{f,y}-\mu\geqslant 0$. If $K_{f,y} > c_{3,\eqref{ass:momentXandnablaX}}+c_{5,\eqref{ass:momentXandnablaX}}$, then
  $$\kappa_{\infty} = \sqrt{c_{\infty,\eqref{eq:condition:normA:rho}}^2  + \tilde{c}_{\infty,\eqref{eq:condition:normA:rho}}^2} {\sqrt{2\delta K_{f,y}+K_{f,z}^2}}<1$$
  as soon as $\delta$ and $K_{f,z}$ are small enough, i.e. the $z$ dependence of $f$ is ``small'' and the $y$ dependence is nearly linear (with coefficient $-\mu$).
\end{Remark}

\subsection{Sufficient assumptions to get $\kappa_p<1$}

\label{section:checking:BSDE:bis}

\begin{Proposition}
{Let $p\geqslant 2$.}
  Assume  \eqref{as:generatorf} and \eqref{as:b:sigma}. Let us assume the growth assumption on $(\nabla X_t^x)_{t \geqslant 0}$ given by the right hand side in \eqref{ass:momentXandnablaX}.
  If $\ta p-\ttheta-c_{5,\eqref{ass:momentXandnablaX}}p>0$, then we have
  $$\tilde{c}_{p,\eqref{eq:condition:stability:rho:bis}}^p \leqslant \frac{(\sqrt{\pi}M_{\sigma}M_{\sigma^{-1}}c_{4,\eqref{ass:momentXandnablaX}})^p (p-1)^{\frac{p-1}{2}}}{{\sqrt{\ttheta}}(\ta p-\ttheta-c_{5,\eqref{ass:momentXandnablaX}}p)^{\frac{p-1}{2}}} .$$
\end{Proposition}

It is worth noting that this type of estimate was not present in \cite{Gobet-Richou-Szpruch-26}. We therefore provide the proof below.

\begin{proof}
We have 
\begin{align*}
  \Exp{\left|\sqrt{\frac{\pi}{\ttheta}}U_\tE^x \sqrt{\tE}e^{-(\ta-\ttheta)\tE}\right|^{\frac{p}{p-1}}}= & \int_0^{+\infty} \E{\left|\sqrt{\frac{\pi}{\ttheta}}U_{t}^x \sqrt{t}e^{-(\ta-\ttheta)t}\right|^{\frac{p}{p-1}}} \sqrt{\frac{\ttheta}{\pi t}}e^{-\ttheta t}\dt\\
  \leqslant &  \left(\frac{\pi}{\ttheta}\right)^{\frac{1}{2(p-1)}}\int_0^{+\infty} \E{|\Ux_t|^{\frac{p}{p-1}}} t^{\frac{1}{2(p-1)}}e^{-\frac{\ta p}{p-1}t+\frac{\ttheta}{p-1}t} \dt.
\end{align*}
Since $p\geqslant 2$, then $\frac{p}{p-1}\leqslant 2$ and Jensen inequality gives us, for all $t>0$,
\begin{align*}
  \E{|\Ux_t|^{\frac{p}{p-1}}} \leq&  \E{|\Ux_t|^{2}}^{\frac{p}{2(p-1)}} \leq\frac{(M_{\sigma}M_{\sigma^{-1}})^{\frac{p}{p-1}}}{t^{\frac{p}{2(p-1)}}}\left(\frac{1}{t}\int_0^t \E{\norm{\nabla X_s^x}^2}\ds\right)^{\frac{p}{2(p-1)}}\\
  \leq& \frac{(M_{\sigma}M_{\sigma^{-1}})^{\frac{p}{p-1}}}{t^{\frac{p}{2(p-1)}}}\left(c_{4,\eqref{ass:momentXandnablaX}}e^{c_{5,\eqref{ass:momentXandnablaX}}t}\right)^{\frac{p}{p-1}}.
\end{align*}
It follows that
\begin{align*}
  \Exp{\left|\sqrt{\frac{\pi}{\ttheta}}U_\tE^x \sqrt{\tE}e^{-(\ta-\ttheta)\tE}\right|^{\frac{p}{p-1}}}
   \leq& \left(\frac{\pi (M_{\sigma}M_{\sigma^{-1}}c_{4,\eqref{ass:momentXandnablaX}})^{2p}}{\ttheta}\right)^{\frac{1}{2(p-1)}}\int_0^{+\infty}  \frac{1}{\sqrt{t}}e^{-\frac{\ta p}{p-1}t+\frac{\ttheta}{p-1}t+\frac{c_{5,\eqref{ass:momentXandnablaX}}p}{p-1}t} \dt\\
   =&\left(\frac{\pi (M_{\sigma}M_{\sigma^{-1}}c_{4,\eqref{ass:momentXandnablaX}})^{2p}}{\ttheta}\right)^{\frac{1}{2(p-1)}} \frac{\sqrt{\pi(p-1)}}{\sqrt{\ta p-\ttheta-c_{5,\eqref{ass:momentXandnablaX}}p}}.
\end{align*}
Finally, we obtain
\begin{equation*}
  \sup_{x \in \bR^d}\Exp{\left|\sqrt{\frac{\pi}{\ttheta}}U_\tE^x \sqrt{\tE}e^{-(\ta-\ttheta)\tE}\right|^{\frac{p}{p-1}}}^{p-1} \leqslant 
\frac{(\sqrt{\pi}M_{\sigma}M_{\sigma^{-1}}c_{4,\eqref{ass:momentXandnablaX}})^p (p-1)^{\frac{p-1}{2}}}{{\sqrt{\ttheta}}(\ta p-\ttheta-c_{5,\eqref{ass:momentXandnablaX}}p)^{\frac{p-1}{2}}}
  \end{equation*}
  which gives us the announced result.
\end{proof}
\begin{Remark}
  If $p<2$, it is also possible to get an upper-bound for $\tilde{c}_{p,\eqref{eq:condition:stability:rho:bis}}$ by using same computations if we have a deterministic upper-bound for $\nabla X_t^x$ which is the case when the SDE has an additive noise.
\end{Remark}

It remains to prove the existence of $c_{p,\eqref{eq:condition:stability:rho}}$ and $\tilde{c}_{p,\eqref{eq:condition:stability:rho}}$.
The following proposition treats the special Brownian case.
\begin{Proposition}
\label{prop:estimate:cpp}
  Let us assume that $X^x=x+W$ and let us consider $\meas(dx)=\frac{c}{(1+|x|)^{pr+d+1}}\dx$ where $c$ is such that $\meas$ is a probability measure on $\bR^d$.
  Then \eqref{eq:condition:stability:rho} is satisfied for
  \begin{align*}
  c_{p,\eqref{eq:condition:stability:rho}}^p &:= \int_0^{+\infty} (t\vee 1)^{\frac{pr+d+1}{2}} \theta e^{-\theta t} \dt \int_{\bR^{d}} \left( 1+|y| \right)^{pr+d+1} \frac{1}{(2\pi)^{d/2}}e^{-\frac{|y|^2}{2}}  \dy<+\infty,\\
  \tilde{c}_{p,\eqref{eq:condition:stability:rho}}^p &:= \int_0^{+\infty} \frac{(t\vee 1)^{\frac{pr+d+1}{2}}}{\sqrt{t}} \frac{\sqrt{\ttheta}}{\sqrt{\pi}} e^{-\ttheta t} \dt \int_{\bR^{d}} \left( 1+|y| \right)^{pr+d+1} \frac{1}{(2\pi)^{d/2}}e^{-\frac{|y|^2}{2}}  \dy<+\infty.
  \end{align*}
\end{Proposition}
\begin{proof}
Firstly, let us consider a measurable function $w : \bR^d \to \bR^{d'}\times \bR^{d'\times d}$ with a polynomial growth $r$. We remark that 
$$ \int_{\bR^d} \norm{w(x)}^p\meas(\dx) \leqslant C \int_{\bR^d} \frac{(1+|x|)^{pr}}{(1+|x|)^{pr+d+1}} \dx \leqslant C \int_{\bR^d} \frac{1}{(1+|x|)^{d+1}} \dx<+\infty.$$
Now we compute
\begin{align*}
  \int_{\bR^d} \E{\norm{w(X_E^x)}^p}\meas(\dx) =& \int_0^{+\infty} \int_{\bR^{2d}} \norm{w(x+\sqrt{t}y)}^p \frac{1}{(2\pi)^{d/2}}e^{-\frac{|y|^2}{2}} \frac{c}{(1+|x|)^{pr+d+1}} \theta e^{-\theta t} \dx\dy\dt\\
  =& \int_0^{+\infty} \int_{\bR^{2d}} \norm{w(z)}^p \frac{1}{(2\pi)^{d/2}}e^{-\frac{|y|^2}{2}} \frac{c}{(1+|z-\sqrt{t}y|)^{pr+d+1}} \theta e^{-\theta t} \dz\dy\dt.
\end{align*}
Since we have
$$\frac{(1+|z|)^{pr+d+1}}{(1+|z-\sqrt{t}y|)^{pr+d+1}} \leqslant \left(\frac{1+|z-\sqrt{t}y|+\sqrt{t}|y|}{1+|z-\sqrt{t}y|}\right)^{pr+d+1} \leqslant \left( 1+\sqrt{t}|y| \right)^{pr+d+1},$$
then we get
\begin{align*}
  &\int_{\bR^d} \E{\norm{w(X_E^x)}^p}\meas(\dx)\\
   \leq& \int_{\bR^d} \norm{w(z)}^p \meas(\dz) \int_0^{+\infty} \int_{\bR^{d}} \left( 1+\sqrt{t}|y| \right)^{pr+d+1} \frac{1}{(2\pi)^{d/2}}e^{-\frac{|y|^2}{2}} \theta e^{-\theta t} \dy\dt\\
  \leqslant & \int_{\bR^d} \norm{w(z)}^p \meas(\dz) \int_0^{+\infty} (t\vee 1)^{\frac{pr+d+1}{2}} \theta e^{-\theta t} \dt \int_{\bR^{d}} \left( 1+|y| \right)^{pr+d+1} \frac{1}{(2\pi)^{d/2}}e^{-\frac{|y|^2}{2}}  \dy
\end{align*}
which gives us the first announced result. The second estimate follows same computations.
\end{proof}

\begin{Remark}
When $X$ is not the Brownian motion, it would be possible to use some Aronson estimates (in the uniformly elliptic case) in order to get the same kind of bound for $c_{p,\eqref{eq:condition:stability:rho}}$ and $\tilde{c}_{p,\eqref{eq:condition:stability:rho}}$. In particular, it could be proved using the same approach as in the proof of  \cite[Theorem 1]{gobet:hal-01904457}. 
\end{Remark}

\section{Numerical schemes}
\label{section:scheme}
The aim of this section is to define some schemes in order to provide a numerical approximation of the function $v = (u, \bar{u})$ solution of \eqref{prop:eq:u} and \eqref{prop:eq:baru}. The first proposed scheme relies on the contraction property given by Proposition \ref{pr:Lpestimate} and a space discretization through a regular grid. A full study of this scheme error is obtained, see Theorem \ref{thm:errornum:infinitehorizon}.  {This first scheme and the associated proof strategy are quite similar to those proposed in \cite{Gobet-Richou-Szpruch-26}. However, when certain significant differences arise in the following, we will highlight them.} Then, we also provide a second and a third scheme based on NNs approximation.

\subsection{A first scheme based on a regular grid}\label{section:first_scheme}


{As previously, we consider in the following ${\rho}(x)=\rho_{r',\eta}(x)=1+\eta^{-1}|x|^{r'}$ with $r' \geqslant r$, $\eta>0$ and where $r$ comes from \eqref{prop:representationBSDEinfinitehorizon:growth}.}
We denote $\Pi$ a non empty finite box subgrid of $\delta \bZ^d$, $N \in \bN^*$ its cardinality, $\delta>0$ its mesh size and $\Box$ its convex hull (in $\bR^d$). Without loss of generality we can assume that $0 \in \Box$. 
In order to define a multilinear interpolation procedure on $\Pi$, we consider the following basis functions:
$$\forall z \in \Pi, \quad\forall x \in \bR^d,\quad \psi_z(x):=\prod_{i=1}^d \left(1-|\delta^{-1}(x_i-z_i)|\right)_+,$$
where $(.)_+$ denotes the positive part function.
For a function $\phi : \Pi \rightarrow \bR^{d'}\times \bR^{d'\times d}$, we define $P\phi$ the interpolation of $\phi$ on $\bR^d$ as follows
\begin{equation*}
	P\phi(x):=
	\begin{cases}
		\sum_{z \in \Pi} \psi_z(x)\phi(z), & x \in \Box \\
		\sum_{z \in \Pi} \psi_z(\text{Proj}(x,\Box))\phi(z), & x\notin \Box\,. 
	\end{cases}
\end{equation*}
By a small abuse of notation we also consider this interpolation operator for functions $\phi :\bR^d \rightarrow \bR^{d'}\times \bR^{d'\times d}$ by defining $P\phi := P\phi_{|\Pi}$.
By the same definition, $P$ can act also on vector-valued and matrix-valued functions. 

This interpolation operator satisfies following standard properties.
\begin{Proposition}
\label{prop:properties-P}
Let  us consider $\phi, \tilde{\phi} : \bR^d \rightarrow \bR^{d'}\times \bR^{d'\times d} $. Then we have
\begin{enumerate}
    \item \label{item1} $$\norm{P\phi - P\tilde{\phi} }_{\infty} \leqslant \norm{\phi -\tilde{\phi}}_{\infty}\,.$$
    \item \label{item2} {If $\phi\in C^2(\bR^d,\bR^{d'}\times \bR^{d'\times d})$ such that $\norm{\nabla^2 \phi(x)} \leqslant C(1+|x|^{r'})$ for all $x \in \bR^d$. Then there exists $C\geqslant 0$ that does not depend on $\Pi$ such that
    $$\norm{\frac{P\phi-\phi}{\rho_{r',\eta}}}_{\infty,\Box} := \sup_{x \in \Box}\norm{\frac{P\phi(x)-\phi(x)}{\rho_{r',\eta}(x)}}\leqslant C\delta^2(1+\delta^{r'}).$$}
    \item \label{item3}
    \begin{equation}
        \label{prop:P:upperbound} 
        \norm{P\phi(x)} \leqslant P\rho_{r',\eta} (x)\sup_{z \in \Pi}\norm{\frac{\phi(z)}{\rho_{r',\eta} (z)}},\quad\forall x \in \bR^d.
    \end{equation}
    \item \label{item4} If $r' \geqslant 2$, there exists $C$ that does not depend on $\Pi$ such that
    \begin{equation}
        \label{prop:Prho/rho} 
        \sup_{x \in \bR^d} \left| \frac{P \rho_{r',\eta}(x)}{\rho_{r',\eta}(x)}\right| \leqslant 
        {1+C\delta^2(1+\delta^{r'})}.
    \end{equation}
\end{enumerate}
\end{Proposition}

    \begin{proof}
        {The first item is standard. For the second one, let us remark that $\Box$ is the union of some hypercubes (also known as cells) whose side has length $\delta$. Let us denote $\mathcal{H}_{\Box}$ the set of cells contained into $\Box$. Then, by a standard result on multilinear interpolation, we have
        \begin{align}
            \sup_{x \in \Box} \norm{\frac{P\phi(x)-\phi(x)}{\rho_{r',\eta}(x)}} &= \sup_{H \in \mathcal{H}_{\Box}} \sup_{x \in H} \norm{\frac{P\phi(x)-\phi(x)}{\rho_{r',\eta}(x)}}
             \leqslant \sup_{H \in \mathcal{H}_{\Box}}  \frac{C \delta^2 \sup_{x \in H}\norm{\nabla^2 \phi(x)}}{\inf_{x \in H}\rho_{r',\eta}(x)}
        \end{align}
        where $C$ does not depend on $H$. Then,
        \begin{align}
            \sup_{x \in \Box} \norm{\frac{P\phi(x)-\phi(x)}{\rho_{r',\eta}(x)}} &\leqslant C\delta^2 \sup_{H \in \mathcal{H}_{\Box}}  \frac{ 1+\eta^{-1}\sup_{x \in H} |x|^{r'}}{1+\eta^{-1}\inf_{x \in H} |x|^{r'}} \leqslant C\delta^2 \sup_{H \in \mathcal{H}_{\Box}}  \frac{ 1+\eta^{-1}(\inf_{x \in H} |x|+\sqrt{d}\delta)^{r'}}{1+\eta^{-1}\inf_{x \in H} |x|^{r'}}\\&\leqslant C {(2^{r'-1} \vee 1)}\delta^2(1+ {\eta^{-1}d^{r'/2}}\delta^{r'}).
        \end{align}
        }
       For the third item, we have, for all $x \in \bR^d$,
     \begin{eqnarray*}
      \norm{P\phi(x)} &=&\norm{\sum_{z \in \Pi} \psi_z(\text{Proj}(x,\Box))\phi(z)} \leqslant \sum_{z \in \Pi} \psi_z (\text{Proj}(x,\Box))\rho_{r'}(z) \norm{ \frac{\phi(z)}{\rho_{r'}(z)}}\\
      &\leqslant& \sup_{z \in \Pi} \norm{\frac{\phi(z)}{\rho_{r'}(z)}} \sum_{z' \in \Pi} \psi_{z'} (\text{Proj}(x,\Box))\rho_{r'}(z') =  P\rho_{r'} (x) \sup_{z \in \Pi}\norm{\frac{\phi(z)}{\rho_{r'} (z)}}.
     \end{eqnarray*}

So, it just remains to prove Item \ref{item4}. Since $r' \geqslant 2$, $\rho_{r',\eta}$ belongs to $C^2(\bR^d,\bR)$ and satisfies $\norm{\nabla^2 \rho_{r',\eta}(x)} \leqslant \eta^{-1}r'(r'-1)|x|^{r'-2} \leqslant C(1+|x|^{r'})$ for all $x \in \bR^d$. Hence we can apply Item \ref{item2} to $\phi=\rho_{r',\eta}$: since $\rho_{r',\eta} \geqslant 1$, we get, for all $x \in \Box$,
\begin{align}
\label{proof:item4:box}
0 \leqslant \frac{P\rho_{r',\eta}(x)}{\rho_{r',\eta}(x)} \leqslant 1 + \left|\frac{P\rho_{r',\eta}(x)-\rho_{r',\eta}(x)}{\rho_{r',\eta}(x)}\right| \leqslant 1+ C\delta^2(1+\delta^{r'}),
\end{align}
where $C$ does not depend on $\Pi$. Now, let $x \notin \Box$ and set $y:=\text{Proj}(x,\Box) \in \Box$. By definition of $P$, we have $P\rho_{r',\eta}(x)=P\rho_{r',\eta}(y)$. Since $\Box$ is a closed convex set containing $0$, the projection onto $\Box$ is $1$-Lipschitz and satisfies $\text{Proj}(0,\Box)=0$, so that $|y| = |\text{Proj}(x,\Box)-\text{Proj}(0,\Box)| \leqslant |x|$. Since $\rho_{r',\eta}$ is a non-decreasing function of $|x|$, it follows that $\rho_{r',\eta}(y)\leqslant \rho_{r',\eta}(x)$. Then, recalling that $P\rho_{r',\eta} \geqslant 0$ and applying \eqref{proof:item4:box} at the point $y$, we obtain
\begin{align}
\frac{P\rho_{r',\eta}(x)}{\rho_{r',\eta}(x)} = \frac{P\rho_{r',\eta}(y)}{\rho_{r',\eta}(x)} \leqslant \frac{P\rho_{r',\eta}(y)}{\rho_{r',\eta}(y)} \leqslant 1+ C\delta^2(1+\delta^{r'}),
\end{align}
which concludes the proof.
     
    \end{proof}

We also define a growth truncation operator $T_{R,{\rho}_r}$ by
$$T_{R,{\rho}_r} : \phi \mapsto \left\lfloor \frac{\phi}{{\rho}_r}\right\rfloor_R {\rho}_r $$
for any function $\phi : \bR^d \rightarrow \bR^{d'}\times \bR^{d'\times d}$ or $\phi : \Pi \rightarrow \bR^{d'}\times \bR^{d'\times d}$, where $\lfloor .\rfloor_{R}$ denotes the projection onto the Euclidean ball $\bar{B}(0,R)$ of $\bR^{d'}\times \bR^{d'\times d}$.

\subsubsection{Definition of the scheme} 
{We now define our numerical scheme. The construction relies on three ingredients: Picard iterations, spatial approximation of functions onto $\Pi$, and empirical estimation of expectations using $M$ i.i.d. samples.}

\begin{Definition}
    \label{def:scheme}
     We construct a sequence of random functions $v^{n}_M : \Omega \times \Pi \rightarrow \bR^{d'}\times \bR^{d' \times d}$, $n \in \bN$ such that $v^0_M=0$ and, for all $n \in \bN$, $z\in \Pi$,
\begin{equation}
 \label{def:defscheme1}
 v^{n+1}_M(z)=T_{B,{{\rho}_r}}\left(\frac1{M} \sum_{j=1}^{M} R^{(\cdot)}_{n,j}(Pv^n_M)\right)(z) , 
\end{equation}
where $B\geqslant \norm{ v}_{{\rho}_r}$ and for any $\phi = (w^1, w^2) : \bR^d \rightarrow \bR^{d'}\times \bR^{d' \times d}$ and $x \in \Pi$,  $(R_{n,j}^{x}(\phi))_{n \in \bN, j \in \bN^*}$ are i.i.d. random variables with the same distribution as 
\begin{align}
    \label{def:Rz}
R^{x}(\phi):=& \Bigg(\frac{1}{\theta } (f(\xxe, w^1(\xxe), w^2(\xxe)) + a w^1(\xxe))e^{-(a-\theta)E},\\
\nonumber
& \quad \sqrt{\frac{\pi}{{\ttheta} }}  \left(f(\Xx_\tE, w^1(\Xx_\tE), w^2(\Xx_\tE)) + \ta w^1(\Xx_\tE)\right)\sqrt{\tE}e^{-(\ta-\ttheta)\tE} \Ux_\tE\Bigg),
\end{align}
recalling that $\theta \in (0,a)$, $\ttheta \in (0,\ta)$, $E\sim\cE(\theta)$, $\tE \sim \Gamma(1/2,\ttheta)$ are independent and independent of $W$.
\end{Definition}

\begin{Remark}
At first sight, there is a weak dependence of the scheme on $v=(u, \bar{u})$, the unknown solution of \eqref{prop:eq:u} and \eqref{prop:eq:baru}, through the condition $B \geqslant \norm{ v}_{{\rho}_r}$. Nevertheless, Corollary \ref{cor:contraction} gives an upper bound on the ${\rho}_r$-norm of $v$. Another possibility to set $B$ is to track a tight upper bound for the constant $C$ in inequality \eqref{prop:representationBSDEinfinitehorizon:growth} from Proposition \ref{prop:representationBSDEinfinitehorizon}. 
{However, the truncation procedure does not appear to be necessary in practice: our numerical experiments in Section~\ref{ch4:numexperiments} exhibit convergence without it.}
\end{Remark}

\begin{Remark}
{The scheme given by Definition~\ref{def:scheme} is fully implementable, and a complete analysis of the numerical error is provided in Theorem~\ref{thm:errornum:infinitehorizon}. However, grid-based spatial approximations suffer from a well-known limitation: the grid size grows exponentially with the dimension $d$, making such schemes impractical in high-dimensional settings. A natural remedy is to replace the grid by a dimension-robust spatial approximation, such as a NN, though this comes at the cost of a more involved theoretical error analysis. This is the purpose of Section~\ref{subsection:NNnum}.}
\end{Remark}

\subsubsection{Theoretical study of the scheme}
The analysis will be done with the norm $\|.\|_{\rho_{r',\eta}}$.
In order to treat the statistical error, we consider Orlicz norms. We denote $\Psi : \bR^+ \rightarrow \bR^+$ an Orlicz function, that is a continuous non-decreasing function, vanishing in zero and with $\lim_{x \rightarrow + \infty} \Psi(x)=+\infty$ and we define the $\Psi$-Orlicz norm of a real random vector $Y$ by 
$$|Y|_{\Psi} := \inf \left\{ c >0, \Exp{\Psi\left(\frac{|Y|}{c}\right)} \leqslant 1 \right\}.$$
{The previous definition extends straightforwardly to matrix-valued random variables. We further assume that $\Psi$ is convex\footnote{convex Orlicz functions are also referred to as ``Young functions'' or ``N-functions'' in the literature.} and increasing, which ensures that $| . |_{\Psi}$ is a norm and that $\Psi^{-1}$ is a concave function defined on $\bR^+$.}

In practice we will use following convex and increasing Orlicz functions.
\begin{itemize}
 \item For $\beta \in (0,1]$, $\Psi^{LT}_{\beta}(x):=e^{(x+x_\beta)^{\beta}}-e^{x_\beta^\beta}$, where $x_\beta:=\left(\frac{1-\beta}{\beta}\right)^{1/\beta}$ is the inflection point of $x\mapsto e^{x^\beta}$, so that $\Psi^{LT}_\beta$ is convex (when $\beta=1$, $x_\beta=0$ and $\Psi^{LT}_1(x)=e^x-1$). We remark that $|Y|_{\Psi^{LT}_{\beta}}<+\infty$ implies that there exists $\varepsilon>0$ such that $\Exp{e^{\varepsilon Y^\beta}}<+\infty$ which means that $Y$ is light-tailed. 
 \item For $\beta >1$, $\Psi^{HT}_{\beta}(x):=e^{(\ln (1+x))^{\beta}}-1$. Observe that when $|Y|_{\Psi^{HT}_{\beta}}<+\infty$, $Y$ has finite polynomial moment of order $p$ for any $p>0$ but $Y^\alpha$ may not have exponential moment for any $\alpha>0$. 
 So, this Orlicz function is well suited for heavy-tailed random variables that are not fat-tailed. In particular the case $\beta=2$ fits well to log-normal tailed random variables. 
 \item For $\beta >1$, $\Psi^{FT}_{\beta}(x):=x^{\beta}$. We consider these Orlicz functions when $Y$ has some polynomial moments of finite order. This is the case for fat-tailed random variables.
\end{itemize}
 {Let us remark that only $\Psi^{LT}_{1}$ is considered in \cite{Gobet-Richou-Szpruch-26}.} All these Orlicz functions satisfy some Talagrand inequality and  maximal inequality. 

\begin{Proposition}
\label{prop:Talagrand-maximal}
$\Psi^{LT}_{\beta}$ $(\beta \in (0,1])$, $\Psi^{HT}_{\beta}$ $(\beta>1)$ and $\Psi^{FT}_{\beta}$ $(\beta >1)$ satisfy:
\begin{enumerate}
 \item {\bf [Talagrand inequality]} There exists a universal constant $C_{\Psi}$ such that, for all sequence $(Y_k)_{1 \leqslant k \leqslant K}$ of independent, mean zero, random variables satisfying $|Y_k|_{\Psi}<+\infty$ for all $1 \leqslant k \leqslant K$, we have
\begin{equation}
 \label{ineq:Talagrand}
 \left|\sum_{k=1}^K Y_k\right|_{\Psi} \leqslant C_{\Psi} \left( \Exp{\left|\sum_{k=1}^K Y_k\right|}+\left|\max_{1\leqslant k \leqslant K} |Y_k| \right|_{\Psi} \right).
\end{equation}
 \item {\bf [Maximal inequality]} There exists a universal constant $C_{\Psi}$ such that, for all sequence $(Y_k)_{1 \leqslant k \leqslant K}$ of random variables satisfying $|Y_k|_{\Psi}<+\infty$ for all $1 \leqslant k \leqslant K$, we have
 \begin{equation}
 \label{ineq:maximal}
  \left|\max_{1\leqslant k \leqslant K} |Y_k| \right|_{\Psi} \leqslant C_{\Psi} \Psi^{-1}(K) \max_{1\leqslant k \leqslant K} |Y_k|_{\Psi}.
 \end{equation}
\end{enumerate}
\end{Proposition}
\begin{proof}
Talagrand inequality for $\Psi^{LT}_{\beta}$ and $\Psi^{FT}_{\beta}$ comes from  \cite[Theorem 3 and Theorem 1]{Talagrand-89}. Talagrand inequality for $\Psi^{HT}_{\beta}$ is given by \cite[Theorem 2.1]{Chamakh-Gobet-Liu-21}. Maximal inequality is provided by  \cite[Lemma 2.2.2]{vanderVart-Wellner-96} (see also  \cite[Theorem 2.2]{Chamakh-Gobet-Liu-21} for $\Psi^{HT}_{\beta}$).
\end{proof}

\begin{Proposition}
\label{prop:numericalerror}
Assume the following hypotheses:
\begin{enumerate}
 \item {Assumptions \eqref{as:generatorf}, \eqref{as:b:sigma}, \eqref{as:int:X:nablaX} and \eqref{as:suprho} are fulfilled, $a>\theta$ and $\ta>\ttheta$. We set $\eta>0$ and $r'\geqslant 2$ such that $r'\geqslant r$ where $r$ comes from \eqref{prop:representationBSDEinfinitehorizon:growth}.
 }
 \item $\kappa_{\infty}<1$, recalling that $\kappa_{\infty}$ is defined in Proposition \ref{pr:Lpestimate} for $\rho=\rho_{r',\eta}$, so that 
 the fixed point equation $\Phi(v)=v$ has a unique solution $v \in C^0(\bR^d,\bR^{d'} \times\bR^{d \times d'})$ satisfying $\| v\|_{{\rho}_{r',\eta}}<+\infty$. Moreover, by Proposition \ref{prop:representationBSDEinfinitehorizon}, we also have $\| v \|_{{\rho}_{r}}<+\infty$ since $(u,\bar{u})$ is a continuous fixed-point of $\Phi$ with $\| (u,\bar{u})\|_{{\rho}_{r',\eta}}<+\infty$ so $v=(u,\bar{u})$ by uniqueness. We set  $\|v\|_{{\rho}_{r}} \leqslant B <+\infty$ which implies that $\norm{v(x)} \leqslant B(1+|x|^r)$.
\item {$v \in C^2(\bR^d,\bR^{d'}\times \bR^{d \times d'})$ such that $\norm{\nabla^2 v(x)} \leqslant C(1+|x|^{r'})$.}
\item There exists a convex and increasing Orlicz function $\Psi$ satisfying Talagrand inequality \eqref{ineq:Talagrand}, Maximal inequality \eqref{ineq:maximal} 
 and there exists $C>0$ that does not depend on $\Pi$ (but depends on $B$) such that, for all functions $\phi : \bR^d \rightarrow \bR^{d'} \times \bR^{d' \times d}$ continuous with a growth $\norm{\phi(x)}\leqslant B(1+|x|^r)$ and $z \in \Pi$,
 \begin{equation}
  \label{ass:finite-Orlicz}
  \left| \frac{R^z(P\phi)}{{\rho_{r',\eta}(z)}} - \Exp{\frac{R^{z}(P\phi)}{{\rho_{r',\eta}(z)}}} \right|_{\Psi} + \Exp{\norm{\frac{R^z(P\phi)}{{\rho_{r',\eta}(z)}} - \Exp{\frac{R^{z}(P\phi)}{{\rho_{r',\eta}(z)}}} }^2} \leqslant C.
 \end{equation}
\end{enumerate}
{Then there exist some constants $C>0$, $\varepsilon \in (0,\kappa_{\infty}^{-1}-1)$ and $\delta_0>0$ such that, for all $M \in \bN^*$, $N \in \bN^*$, $n \in \bN$ and $\delta \leqslant \delta_0$,
\begin{align*}
\Exp{\sup_{x \in \bR^d} \norm{\frac{ Pv^{n}_M(x) -v(x)}{\rho_{r',\eta}(x)}}}
  \leqslant &C \left(\frac{\Psi^{-1}(N)}{\sqrt{M}}\left(1+\frac{\Psi^{-1}(M)}{\sqrt{M}}\right)+\delta^2
 +  \sup_{x \in \bR^d \setminus \Box} \left(\frac{1+|x|^r}{\rho_{r',\eta}(x)}\right) +(1+\varepsilon)^n\kappa_{\infty}^n   \right).
\end{align*}}
\end{Proposition}

The proof is postponed to Section \ref{proof:prop:numericalerror}.

\begin{Remark}
\label{rem:comments-num-error}
 The upper bound obtained for the numerical error in Proposition \ref{prop:numericalerror} can be easily analyzed: 
 \begin{enumerate}
 \item The first term is the statistical error coming from the approximation of the expectation by an empirical mean. When we consider Orlicz functions $\Psi^{LT}_{\beta}$ ($\beta \in (0,1]$), $\Psi_\beta^{HT}$ ($\beta > 1$) and $\Psi_\beta^{FT}$ ($\beta \geqslant 2$), then 
 \begin{align*}
 \frac{(\Psi^{LT}_{\beta})^{-1}(N)}{\sqrt{M}}\left(1+\frac{(\Psi^{LT}_{\beta})^{-1}(M)}{\sqrt{M}}\right) & = O\left( \frac{(\ln(1+N))^{1/\beta}}{\sqrt{M}} \right),\\
 \frac{(\Psi^{HT}_{\beta})^{-1}(N)}{\sqrt{M}}\left(1+\frac{(\Psi^{HT}_{\beta})^{-1}(M)}{\sqrt{M}}\right) & = O\left( \frac{e^{(\ln(1+N))^{1/\beta}}}{\sqrt{M}} \right),\\
 \text{and } \quad \frac{(\Psi^{FT}_{\beta})^{-1}(N)}{\sqrt{M}}\left(1+\frac{(\Psi^{FT}_{\beta})^{-1}(M)}{\sqrt{M}}\right) & = O\left( \frac{N^{1/\beta}}{\sqrt{M}} \right),
 \end{align*}
 where the $O(.)$ notation means up to a constant which, here, is uniform in the various numerical parameters $(\delta, N, M, n)$.

 \item The second term is related to the space discretization by a discrete grid $\delta \bZ^d$.
 \item  The third term is a truncation error. In order to get a good control on it we must consider a weight function $\rho_{r',\eta}$ with a bigger growth than the solution $v$, i.e. $r'>r$. 
  \item The last term comes from the Picard procedure using the contraction property.
 \end{enumerate}
\end{Remark}

\begin{Remark}
 {The calculations in the proof of Proposition \ref{prop:numericalerror} follow the same line of reasoning as those in Proposition 7 of \cite{Gobet-Richou-Szpruch-26}, with the exception that we do not assume that $\nabla^2 v$ is bounded.} However, we have identified an error in the proof of the latter\footnote{In the study of error $\mathcal{E}_{\infty,2}$, page 19 in \cite{Gobet-Richou-Szpruch-26}, a term $\sup_{x \in \bR^d} \frac{P\rho_{r'}(x)}{\rho_{r'}(x)}$ must appear in the last line of inequalities.} which, once corrected, results in $(1+\varepsilon)$ appearing before $\kappa_{\infty}$, as well as the additional constraint $\delta \leqslant \delta_0$. Note that the result of Corollary 3.1 of \cite{Gobet-Richou-Szpruch-26} is also affected in the same way.
\end{Remark}



We specify now the error given by Proposition \ref{prop:numericalerror} for several sets of assumptions on $b$, $\sigma$ and $f$.
We also assume that our grid $\Pi$ is centered in $0$, and is given by
$$\left\{(i_1\delta,...,i_d \delta) \,|\, i_k \in \{-\tilde{N},...,\tilde{N}\}, k \in \{1,...,d\} \right\}$$
for a given $\tilde{N} \in \bN$. It implies that $N=(2\tilde{N}+1)^d$. 

\begin{Lemma}
\label{lem:estim:orlicz}
 Assume \eqref{as:generatorf}, \eqref{as:b:sigma} and \eqref{as:int:X:nablaX}.\\ We set $r'>r$ where $r$ comes from \eqref{prop:representationBSDEinfinitehorizon:growth}. 
 \begin{enumerate}
  \item \label{lem:estim:orlicz:item1} If $\sigma$ is constant, $\ta-\ttheta-(1+r)K_b>0$ and $a-\theta-r K_b>0$, then \eqref{ass:finite-Orlicz} is fulfilled for $\Psi_{1/(r+1)}^{LT}$.
  \item \label{lem:estim:orlicz:item2} If $f(x,y,z)=f(x,y)$ (in this case, we only consider the first component of the scheme that approximates only $u$) and $a-{\theta}-r K_b>0$, then \eqref{ass:finite-Orlicz} is fulfilled for $\Psi_{1/(r\vee 1)}^{LT}$. 
  \item \label{lem:estim:orlicz:item3} If $\beta \geqslant 2$,  $\ta-(r+1)K_b-\frac{2\beta-1}{2} K_{\sigma}^2-\frac{\beta-1}{\beta} \ttheta>0$ and $a-rK_b-\frac{\beta-1}{\beta} {\theta}>0$, then \eqref{ass:finite-Orlicz} is fulfilled for $\Psi_{\beta}^{FT}$.
 \end{enumerate}
\end{Lemma}

The proof of Lemma \ref{lem:estim:orlicz} is postponed to Section \ref{proof:lem:estim:orlicz}.

\begin{Theorem}
\label{thm:errornum:infinitehorizon}
 {Assume \eqref{as:generatorf}, \eqref{as:b:sigma}, \eqref{as:int:X:nablaX} and \eqref{as:suprho}, $a>\theta$ and $\ta>\ttheta$.\\
 We set $r'\geqslant 2$ such that $r'>r$ where $r$ comes from \eqref{prop:representationBSDEinfinitehorizon:growth}. We also assume that $\kappa_{\infty}<1$, $\norm{ v}_{{\rho}_r} \leqslant B <+\infty$ and $v \in C^2(\bR^{d},\bR^{d'}\times \bR^{d' \times d})$ with the growth $\norm{\nabla^2 v(x)} \leqslant C(1+|x|^{r'})$.
 
Then there exist some constants $\varepsilon \in (0,\kappa_{\infty}^{-1}-1)$ and $\delta_0>0$ such that the following holds for $M \in \bN^*$, $\tilde{N} \in \bN$, $n \in \bN$ and $\delta \leqslant \delta_0$.
 
  \begin{enumerate}
  \item If $\sigma$ is constant, $\ta-\ttheta-(1+r)K_b>0$ and ${a}-{\theta}-r K_b>0$, then 
  $$\Exp{ \sup_{x \in \bR^d} \norm{\frac{ Pv^{n}_M(x) -v(x)}{\rho_{r',\eta}(x)}} } = O\left( \delta^2+\frac{(\ln (2+\tilde{N}))^{(r+1)}}{\sqrt{M}} + |1 \vee (\tilde{N}\delta)|^{r-r'}+ \left(\kappa_{\infty}(1+\varepsilon)\right)^n \right).$$
  \item If $f(x,y,z)=f(x,y)$ (in this case, we only consider the first component of the scheme that approximates only $u$) and $a-{\theta}-r K_b>0$, then
  $$\Exp{ \sup_{x \in \bR^d} \norm{\frac{ Pv^{n}_M(x) -v(x)}{\rho_{r',\eta}(x)}}} = O\left( \delta^2+\frac{(\ln (2+\tilde{N}))^{r\vee 1}}{\sqrt{M}} + |1 \vee (\tilde{N}\delta)|^{r-r'} +\left(\kappa_{\infty}(1+\varepsilon)\right)^n \right).$$
  \item If $\beta \geqslant 2$,  $\ta-(r+1)K_b-\frac{2\beta-1}{2} K_{\sigma}^2-\frac{\beta-1}{\beta} \ttheta>0$ and $a-rK_b-\frac{\beta-1}{\beta} {\theta}>0$, then 
  $$\Exp{ \sup_{x \in \bR^d} \norm{\frac{ Pv^{n}_M(x) -v(x)}{\rho_{r',\eta}(x)}}} = O\left( \delta^2+\frac{(1+2\tilde{N})^{d/\beta}}{\sqrt{M}} + |1 \vee (\tilde{N}\delta)|^{r-r'} +\left(\kappa_{\infty}(1+\varepsilon)\right)^n \right).$$
 \end{enumerate}
}

\end{Theorem}

\begin{proof}

Thanks to Lemma \ref{lem:estim:orlicz} and Proposition \ref{prop:Talagrand-maximal} we are allowed to apply Proposition \ref{prop:numericalerror} and we just have to study all the terms in the upper-bound. First, we have 
\begin{align*}
\sup_{x \in \bR^d \setminus \Box} \left(\frac{1+|x|^r}{\rho_{r',\eta}(x)}\right) &= \sup_{x \in \bR^d \setminus \Box} \left(\frac{1+|x|^r}{1+\eta^{-1}|x|^{r'}}\right)   \leqslant \sup_{|x| \geqslant \tilde{N}\delta} \frac{1+|x|^r}{1+\eta^{-1}|x|^{r'}}= O\left(|\tilde{N}\delta|^{r-r'} \right).
\end{align*}
Moreover, when $\tilde{N}\delta <1$, we also have
$$\sup_{|x| \geqslant \tilde{N}\delta} \frac{1+|x|^r}{1+\eta^{-1}|x|^{r'}} \leqslant \sup_{y \in \mathbb{R}^+} \frac{1+y^r}{1+\eta^{-1}y^{r'}} \leqslant C_{\eta}.$$ 
Finally, we obtain the result using Remark \ref{rem:comments-num-error}.
\end{proof}

\subsection{Some schemes based on NNs approximation}

\label{subsection:NNnum}


\subsubsection{2nd scheme: contraction based scheme using NNs approximation}
\label{section:second_scheme}

We use once again the contraction property by considering a sequence of NNs $(\cU_{\zeta_n,n})_{n \in \bN}$, i.e. a sequence of functions $x \mapsto \left( \cU^1_{\zeta_n,n}(x), \cU^2_{\zeta_n,n}(x) \right) \in \bR^{d'} \times \bR^{d'\times d}$ parametrized by $\zeta_n \in \Theta_{n}$. For a given set of parameters $\Theta$, we denote $\cN_{\Theta}:= \{\cU_\zeta:\zeta\in \Theta \}$ the associated parametrized family of functions. We set $\cU_{\zeta_0,0}=0$. Knowing the NN $\cU_{\zeta_n,n}$, we train $\cU_{\zeta_{n+1},n+1}$ as follows:
\begin{itemize}
\item[1.] Sample $X_0 = X_{t=0} \sim \mu_0$, $E\sim \cE(\theta)$ and $\tE\sim \Gamma(1/2, \ttheta)$.
\item[2.] Get corresponding samples of the Malliavin weight $U^{X_0}_E$, $U^{X_0}_\tE$ and state processes $X^{X_0}_E$, $X^{X_0}_\tE$ at times $E$ and $\tE$ for the starting point $X_0$. 
\item[3.] Evaluate the previous NN $\cU_{\zeta_n,n}$ at each sample of $X^{X_0}_E$ and $X^{X_0}_\tE$ and compute the samples of $\varphi(\cU_{\zeta_n,n})(X_0)$ given by,
\begin{align}
\varphi(\cU_{\zeta_n,n})(X_0) = & \Bigg(\frac{1}{\theta } (f(X^{X_0}_E, \cU^1_{\zeta_n,n}(X^{X_0}_E), \cU_{\zeta_n,n}^2(X^{X_0}_E)) + a \cU^1_{\zeta_n,n}(X^{X_0}_E))e^{-(a-\theta)E} ,\\
& {\hspace{-1.5cm}} \sqrt{\frac{\pi}{{\ttheta} }}   \left(f(X^{X_0}_\tE, \cU^1_{\zeta_n,n}(X^{X_0}_\tE), \cU^2_{\zeta_n,n}(X^{X_0}_\tE)) + \ta \cU^1_{\zeta_n,n}(X^{X_0}_\tE)\right)\sqrt{\tE}e^{-(\ta-\ttheta)\tE} U^{X_0}_\tE \Bigg).
\end{align}
\item[4.] Minimize the loss function,
$$\cL(\cU_{\zeta, n+1}) := \E{\norm{\cU_{\zeta,n+1}(X_0) - \varphi(\cU_{\zeta_n,n})(X_0)}^2},\quad \zeta \in \Theta_{n+1}$$
and choose the parameters $\zeta_{n+1}$,
\begin{align}
\label{loss:function:theo}
\zeta_{n+1} \in \arg\min_{\zeta \in \Theta_{n+1}} \E{\norm{\cU_{\zeta,n+1}(X_0) - \varphi(\cU_{\zeta_n,n}(X_0)}^2}.
\end{align}
Let us remark that we have
\begin{align*}
\Exp{\norm{\cU_{\zeta,n+1}(X_0) - \varphi(\cU_{\zeta_n, n} )(X_0)}^2} = & \int_{\bR^d} \norm{\cU_{\zeta,n+1}(x) - \E{\varphi(\cU_{\zeta_n, n} )(x)}}^2 \mu_0(\dx)\\
& + \int_{\bR^d} \E{ \norm{ \E{\varphi(\cU_{\zeta_n, n} )(x)} - \varphi(\cU_{\zeta_n, n} )(x) }^2} \mu_0(\dx).
\end{align*}
Thus, we obtain
$\cU_{\zeta_{n+1},{n+1}} = \text{Proj}\left( \E{\varphi(\cU_{\zeta_n, n} )(.)}, \cN_{\Theta_{n+1}}\right)$
for the $L_{\mu_0}^2$ norm defined, for all measurable function $h: \bR^d \to \bR^k$ with $k \in \bN^*$, by 
$\norm{h}_{L^2_{\mu_0}} := \left(\int_{\bR^d} |h(x)|^2 \mu_0(\dx)\right)^{1/2}.$ In particular, if $\cN_{\Theta_{n+1}}$ is large enough, $\cU_{\zeta_{n+1},{n+1}}$ is close
to be equal to $\E{\varphi(\cU_{\zeta_n, n} )(.)}$ which is exactly what we are looking for.
\end{itemize}

The scheme can be summarized as follows.
\vspace{0.3cm}

\begin{breakablealgorithm}
\caption{Contraction based scheme using NN}\label{ch4:alg:contractionNN}
\begin{algorithmic}[1]

\noindent We start with an initial guess for the NNs for $(\cU^{1}_{\zeta_0,0}, \cU^{2}_{\zeta_0,0}) = (0, 0)$. At a given step $n+1$ of the Picard iteration, we sample a random starting point from the distribution $\mu_0$ and compute $\varphi(\cU_{\zeta_n,n})$ and we minimize the loss function $\cL(\cU_{\zeta,n+1})$.
\vspace{0.2cm}
\State{}Initialize $\displaystyle(\cU^{1}_{\zeta_0,0}, \cU^{2}_{\zeta_0,0}) \gets (0, 0)$.
\For{$n$ from 0 to $N-1$} \Comment{Picard iterations}
    \State{}Obtain $M$ samples of $X_{t=0} \sim \mu_0$, $E\sim \cE(\theta)$ and $\tE\sim \Gamma(\frac{1}{2}, \ttheta)$.
    \State{}Obtain corresponding $M$ samples of $U_E^{X_0}$, $U_\tE^{X_0}$, $X_E^{X_0}$ and $X_\tE^{X_0}$.
    \State{}Define
    $$\varphi^1(\cU_{\zeta_n,n})(X_0) \gets \dfrac{1}{\theta } (f(X^{X_0}_E, \cU^1_{\zeta_n,n}(X^{X_0}_E), \cU_{\zeta_n,n}^2(X^{X_0}_E)) + a \cU^1_{\zeta_n,n}(X^{X_0}_E))e^{-(a-\theta)E}.$$
	\State{}Define
	$$\quad\ \ \varphi^2(\cU_{\zeta_n,n})(X_0) \gets \sqrt{\dfrac{\pi}{{\ttheta} }}   \left(f(X^{X_0}_\tE, \cU^1_{\zeta_n,n}(X^{X_0}_\tE), \cU^2_{\zeta_n,n}(X^{X_0}_\tE)) + \ta \cU^1_{\zeta_n,n}(X^{X_0}_\tE)\right)\sqrt{\tE}e^{-(\ta-\ttheta)\tE} U^{X_0}_\tE. $$
	\State{}Set $\displaystyle\varphi(\cU_{\zeta_n,n})(X_0) \gets \big(\varphi^1(\cU_{\zeta_n,n})(X_0),\, \varphi^2(\cU_{\zeta_n,n})(X_0)\big).$
	
    \State{}Train the NN $\cU_{\zeta, n+1}$ to minimize
$$\cL(\cU_{\zeta, n+1}) = \dfrac{1}{M}\sum_{k=1}^M \norm{\cU_{\zeta,n+1}(X^k_0) - \varphi(\cU_{\zeta_n,n})(X^k_0)}^2,\quad \zeta \in \Theta_{n+1}$$
\quad\ \ using ADAM optimizer.
    \State{}Update $\displaystyle\zeta_{n+1} \gets \arg\min_{\zeta \in \Theta_{n+1}} \cL(\cU_{\zeta, n+1}).$
\EndFor

\end{algorithmic}
\end{breakablealgorithm}

\vspace{0.3cm}

Now we want to study the error 
 \begin{align*}
  e_n:=\left(\int_{\bR^d} \norm{ \cU_{\zeta_{n},{n}}(x)-v(x)}^2 \mu_0(\dx)\right)^{1/2}, \quad n \in \bN. 
 \end{align*}
Let us introduce $r_{n} := \E{\norm{\cU_{\zeta_{n},{n}}(X_0) - \mathbb{E}_{X_0}\left[\varphi(\cU_{\zeta_{n-1},{n-1}})(X_0)\right]}^2}^{1/2}$, for $n \in \bN$ and where $\mathbb{E}_{X_0}[.]$ denotes the expectation conditionally to $X_0$. $r_n$ is the accuracy of the NN approximation at the $n$-th Picard iteration. 
We have the following result for the convergence of this scheme to the solution of the BSDE \eqref{eq:BSDEinfinitehorizon}.
\begin{Proposition}\label{ch4:prop:contractionNN}
We assume assumptions of Proposition \ref{pr:Lpestimate} for $p=2$ and $\meas = \mu_0$. Let us suppose that $\kappa_2<1$ for $\kappa_2$ defined in Proposition \ref{pr:Lpestimate}. We also assume that we can solve exactly at each Picard step the optimization problem associated with the theoretical loss function \eqref{loss:function:theo}. Then, the numerical approximation error of the scheme at $n$-th Picard iteration is bounded by
$$e_{n} \leqslant \frac{\sup_{1 \leqslant k \leqslant n}r_k}{1-\kappa_2} + \kappa_2^{n} e_0.$$
\end{Proposition}

\begin{proof}

We denote by $\cN_{\Theta}:=\{\cU_\zeta,\zeta \in \Theta \}$ the set of functions $\bR^d \rightarrow \bR^{d'}\times \bR^{d' \times d}$ that can be mimicked by a NN with parameters in $\Theta$.  

We assume that $\kappa_{2} < 1$. Then Corollary \ref{cor:contraction} gives us the existence and uniqueness of a solution $v$ to \eqref{prop:eq:u} and \eqref{prop:eq:baru}. We also assume that, $\zeta \in \Theta_n$ is numerically computable for the loss function \eqref{loss:function:theo}, i.e. we ignore numerical optimisation errors. In particular, $\cU_{\zeta_n,n}$ is deterministic for all $n \in \bN$.
 We have
\begin{align*}
e_{n+1} \leqslant & \E{\norm{\cU_{\zeta_{n+1},{n+1}}(X_0) - \mathbb{E}_{X_0}\left[\varphi(\cU_{\zeta_{n},{n}})(X_0)\right]}^2}^{1/2} + \E{\norm{\mathbb{E}_{X_0}\left[\varphi(\cU_{\zeta_{n},{n}})(X_0)\right] - v(X_0)}^2}^{1/2}\\
\leqslant & r_{n+1} + \E{\norm{\Phi(\cU_{\zeta_{n},{n}})(X_0) - \Phi(v)(X_0)}^2}^{1/2}\leqslant  r_{n+1} + \kappa_2 e_n,
\end{align*}
which easily implies the result. 
\end{proof} 
{
 {This convergence result is purely theoretical, for two reasons. First, there is no guarantee that the numerical optimization converges to a global minimum, and even if it does converge, there is no theoretical rate of convergence to that minimum. 
Second, the function to be minimized must be approximated by its empirical counterpart (see line 8 of Algorithm \ref{ch4:alg:contractionNN}), and the resulting statistical error will also be propagated through the Picard iteration.
On the theoretical side, it is known that at each step one can find a $\Theta_n$ large enough to make $r_{n}$ arbitrarily small: 
see, for instance, \cite{suzuki2021adaptive} and \cite{YAROTSKY2017}, which provide theoretical guidelines for choosing the architecture of the neural network so that $r_{n}$ falls below a prescribed threshold.
}
Finally, this scheme shares the same drawback as the grid-based scheme: it requires a contraction condition for the Picard iterations to converge. The next NN scheme is designed to overcome this limitation.
}

\subsubsection{{3rd scheme:} direct scheme using NNs approximation}
\label{section:third_scheme}
Here we propose a scheme using a NN approximation that does not rely on a contraction property. 
We train a NN $(\cU_\zeta, \zeta \in \Theta)$, in $I$ epochs, each with $\cP$ gradient descent steps, whereby the NN is optimized for a given set of samples before we update the input samples in the next epoch. 
The proposed scheme is defined as follows.
\vspace{0.3cm}

\begin{breakablealgorithm}
\caption{Direct scheme using NN}\label{ch4:alg:directNN}
\begin{algorithmic}[1]

\vspace{\baselineskip}

\State{}  \noindent  {Initialize $\zeta_{\cP}^0$}
 \For{$i$ from $1$ to $I$}     \Comment{Iterating over epochs}
     \State{}Obtain $M_x$ samples of $X_{t=0} \sim \mu_0$.
     \State{}For each sample of $X_0$, obtain $M$ samples of $\big(E, \tE, X_E^{X_0}, X_\tE^{X_0}, U_\tE^{X_0} \big)$ where
     $${E}\sim \cE({\theta}) \text{ and } \tE\sim \Gamma(0.5, \ttheta).$$
     \State{}Set $\zeta^{i}_{n=0} \gets \zeta^{i-1}_{n=\cP}$.    \Comment{Using the NN from the end of the previous epoch}
     \State{}Define for each $k=1, \dots, M_x$ and $j=1, 2, \dots, M$,    \Comment{Targets, computed once per epoch}
     $$\varphi^{1,j}(\cU_{\zeta_{0}^i})(X^k_0) \gets \dfrac{1}{\theta } (f(X^{X^k_0}_{E^j}, \cU^1_{\zeta_{0}^i}(X^{X^k_0}_{E^j}), \cU_{\zeta_{0}^i}^2(X^{X^k_0}_{E^j})) + a \cU^1_{\zeta_{0}^i}(X^{X^k_0}_{E^j}))e^{-(a-\theta){E^j}}.$$
     \State{}Define for each $k=1, \dots, M_x$ and $j=1, 2, \dots, M$,
     $$\varphi^{2,j}(\cU_{\zeta_{0}^i})(X^k_0) \gets \sqrt{\dfrac{\pi}{{\ttheta} }}   \left(f(X^{X^k_0}_{\tE^j}, \cU^1_{\zeta_{0}^i}(X^{X^k_0}_{\tE^j}), \cU^2_{\zeta_{0}^i}(X^{X^k_0}_{\tE^j})) + \ta \cU^1_{\zeta_{0}^i}(X^{X^k_0}_{\tE^j})\right)\sqrt{\tE^j}e^{-(\ta-\ttheta)\tE^j} U^{X^k_0, j}_{\tE^j}. $$
     \State{}Set $\displaystyle\varphi^j(\cU_{\zeta_{0}^i})(X^k_0) \gets \big(\varphi^{1,j}(\cU_{\zeta_{0}^i})(X^k_0),\, \varphi^{2,j}(\cU_{\zeta_{0}^i})(X^k_0)\big)$ $\forall j=1, \dots, M,\ \forall k= 1, \dots, M_x$.
     \For{$n$ from $0$ to $\cP-1$}        \Comment{ADAM steps with fixed targets}
         \State{}Given $\zeta_{n}^i$, we get {$\zeta_{n+1}^{i}$} using one ADAM optimization step for the loss function,
         $$\Upsilon_{M, M_x}(\cU_{\zeta}) :=  \dfrac{1}{M_x}\sum_{k=1}^{M_x} \norm{\cU_\zeta(X^k_0) - \frac{1}{M}\sum_{j=1}^M \varphi^j({\cU_{\zeta^i_0}})(X^k_0)}^2.$$
     \EndFor
 \EndFor

\end{algorithmic}
\end{breakablealgorithm}

 {

\vspace{\baselineskip}
Algorithm~2 trains a single NN $\mathcal U_\zeta=(\mathcal U^1_\zeta,\mathcal U^2_\zeta)$ to satisfy the
fixed point equation $v=\Phi(v)$ of Section~2.3 directly, that is, to make the
residual
\begin{equation}
 \mathcal R_\zeta:=\mathcal U_\zeta-\Phi(\mathcal U_\zeta)                     \label{eq:residual}
\end{equation}
vanish in $L^2_{\mu_0}$. This relates the scheme to physics-informed neural networks
(PINNs)~\cite{RPK19}, which penalise the residual of the PDE, here
$\mathcal L\mathcal U^1_\zeta+f(\cdot,\mathcal U^1_\zeta,\nabla_x\mathcal U^1_\zeta\sigma)$ in \eqref{eq:PDEellipticRd}, and to
the losses of~\cite{NR23}, which interpolate between BSDE-based and PINN losses for
elliptic and parabolic boundary value problems. However, the residual~\eqref{eq:residual} comes from the integral
representation~\eqref{prop:eq:u}--\eqref{prop:eq:baru} rather than from the PDE: since $\bar u$ is a separate
output of the NN and is represented through the Malliavin weight
$U^{X_0}_{\tilde E}$, computing the targets requires no spatial derivative of the NN.
The derivative-free method of~\cite{HNS20} also trains NNs on targets given by a
probabilistic representation, for elliptic boundary value problems. 

In terms of its optimization procedure, Algorithm~\ref{ch4:alg:directNN} does not minimize
$\int_{\bR^d} \norm{ \mathcal R_\zeta(x)}^2 \mu_0(\dx)$
 directly. Indeed, the targets $\frac1M\sum_{j=1}^M\varphi^j(\cU_{\zeta^i_0})(X^k_0)$ are computed with the parameters $\zeta^i_0$ available at the beginning of epoch $i$ and are kept fixed during its $\cP$ optimization steps, the NN being initialized with the parameters obtained at the end of the previous epoch. Each epoch can therefore be seen as an inexact Picard iteration, in which the regression problem is only partially solved, similarly to fitted value iteration \cite{MS08}.
 Heuristically, this partial resolution acts as a relaxation of Picard iterations, $\cU_{i+1}\approx(1-\gamma)\cU_i+\gamma\Phi(\cU_i)$ for some $\gamma\in(0,1]$. A linearization around the fixed point suggests that such a relaxation may converge under the weaker condition that the spectrum of the derivative of $\Phi$ at $v$ lies in $\{\lambda\in\mathbb C:\ \mathrm{Re}\,\lambda<1\}$, whereas Picard iterations require a contraction. This could explain the numerical behavior observed in Section~\ref{ch4:numexperiments}, but a rigorous convergence analysis of Algorithm~\ref{ch4:alg:directNN} is left for future work.

}


\section{Numerical experiments}
\label{ch4:numexperiments}
In this section, we provide some numerical experiments \footnote{The numerical experiments are available as Python code on the following GitHub link - \href{https://github.com/CharuShardul/Inf_BSDE_solvers}{https://github.com/CharuShardul/Inf\_BSDE\_solvers}.} for the three schemes described above. In general, for the state process $\Xx_t$, we consider only the Brownian motion $W^x_t$ starting from $x\in \bR^d$, i.e., $b(x) = 0$ and $\sigma(x) =\bI_d$ ($d$-dimensional identity matrix) in \eqref{eq:SDE:intro}. However, the numerical experiments could be extended to a general SDE with non-constant $b$ and $\sigma$, which we illustrate for the grid based contraction scheme by taking $\sigma$ as a function of the state process and $b=0$.  {The case $b\neq 0$ can add a systematic drift which can be taken into account by not centering the grid at $0$, so as to prevent the truncation error from becoming excessive.}

We design examples with analytical solutions in the following way: We start with a given function $u(x)\in C^2(\bR^d,\bR^{d'})$ and construct the generator $f$ such that the corresponding BSDE is satisfied by $\Yx_t = u(X_t^x)$ and $\Zx_t = \nabla_x u(X_t^x)\sigma(X^x_t)$. This can be done by choosing a suitable function $f_0(x,y,z)$ that satisfies  \eqref{as:generatorf}, in particular, with $\mu$ as the $y$-monotonicity constant and $K_{f,z}$ as the $z$-Lipschitz continuity constant. Then, define $f$ as follows:
\begin{align*}
    f(x, y, z) = f_0(x,y,z) - \dfrac{1}{2}\text{Tr}\big(\nabla_x^2 u(x)\sigma(x)\sigma(x)^\top\big) - f_0(x, u(x), \nabla_x u(x)\sigma(x)) 
\end{align*}
so that $u$ satisfies the PDE \eqref{eq:PDEellipticRd} which is equivalent to the BSDE \eqref{eq:BSDEinfinitehorizon}. Here $\nabla_x^2$ denotes the Hessian operator.

\subsection{First Scheme - Grid based contraction scheme}

\subsubsection{Constant $\sigma$ with $X_t^x = x + W_t$}
We use the following example with $d'=1$ for all three numerical schemes:
\begin{align}\label{eq:numexp_constsig}
    u(x) &= \dfrac{1}{d}\sum_{i=1}^d \tan^{-1}(x_i)\\
    \Bar{u}(x) &= \dfrac{1}{d}\left(\dfrac{1}{1+x_1^2},\ \dots \ ,\dfrac{1}{1+x_d^2} \right)\\
    f_0(x, y, z) &= -cy + \cos{(y+|x|)} + K_z \sin(|z|)\\
    f(x, y, z) &= f_0(x, y, z) + \dfrac{1}{d}\sum_{i=1}^d \dfrac{x_i}{(1+x_i^2)^2} - f_0(x, u(x), \Bar{u}(x)). 
\end{align}

 {
In this subsection, we take $d=1$ and the Malliavin weight can be calculated as per \eqref{eq:malliavin:weight}, which gives $\Ux_t = \frac{W_t}{t}$ that does not depend on $x$. In addition, we have taken the following values for the parameters: $c = 2$ and $K_z = 0.5$ with $a = 2$, $\ta = 2$, $\theta = 1.5$ and $\ttheta = 1.5$ being the parameters defined in Proposition \ref{prop:representationBSDEinfinitehorizon}. Let us note that the existence and uniqueness of the solution to the BSDE are guaranteed twice by Proposition~\ref{prop:existence:uniqueness:BSDE}: first because $d'=1$ and $r=0$, and second because $2\mu-K_{f,z}^2=2-0.25>0$. It is also interesting to check whether the theoretical contraction bounds hold in this example. Here, we have $\mu=1$, $K_{f,y}=3$ and $r=0$, so that we can take $\rho=1$. We get
$$c_{\infty,\eqref{eq:condition:normA:rho}}=\int_0^{+\infty}e^{-at}\dt=0.5
\quad\text{and}\quad
\tilde{c}_{\infty,\eqref{eq:condition:normA:rho}}=\int_0^{+\infty}e^{-\ta t}\,\E{\norm{\Ux_t}}\dt=\sqrt{\frac{2}{\pi}}\int_0^{+\infty}\frac{e^{-\ta t}}{\sqrt{t}}\dt=1.$$
Then the theoretical bound given by Proposition~\ref{pr:Lpestimate} becomes
$$\kappa_{\infty}=\sqrt{1.25\,(9+K_z^2)},$$
which is larger than $1$ for every value of $K_z$. Even if we replace the general Lipschitz constant given by Lemma~\ref{lem:newLipcste} by the exact Lipschitz constant of $y\mapsto f(x,y,z)+ay$, which is equal to $1$ in this example, we only get $\kappa_\infty=\sqrt{1.25\,(1+K_z^2)}>1$. Since convergence is observed in all our numerical experiments, this shows that our theoretical bound is very pessimistic.}

We provide some numerical experiments for the grid-based contraction scheme defined in Definition \ref{def:scheme}. We have not used the truncation operator $T_{B,{\rho}_r}$ as, from an empirical point of view, it was not necessary to obtain a numerically stable solution and it had no impact on accuracy. 
{Unless otherwise specified: $NbP = 10$, $\tilde N = 10$, $M=40000$ and $R=3$, where $NbP$ is the number of Picard iterations, $(2\tilde{N} + 1)^d$ is the number of points in the $d$-dimensional grid $\Pi$ centered at $0$, $M$ is the number of samples and $2R$ is the side length of the $d$-dimensional hypercube spanned by $\Pi$, with $\delta=\frac{R}{\tilde{N}}$ denoting the mesh size.} 

Since we use a grid approximation for this scheme with a finite grid over the domain, it is natural to expect this choice to affect the errors, especially as we get closer to the boundary. Hence, we introduce a boundary truncation parameter $p$ and compute the solution on a larger grid with $(2(\tilde N + p) + 1)^d$ points and discard the $(2(\tilde N + p) + 1)^d$ - $(2\tilde N + 1)^d$ points close to the boundary and we denote the resulting truncated domain by $\Pi_p$.

We denote the deviation of the solution at the $n$-th Picard iteration from the analytical solution as follows:
\begin{align}
\Delta u^n_{\tilde{N}, p}(x) &:= |u^n_{\tilde{N}, p}(x) - u(x)|,\quad x \in \Pi_p,\\
\Delta \bar{u}^n_{\tilde{N}, p}(x) &:= |\bar{u}^n_{\tilde{N}, p}(x) - \bar{u}(x)|,\quad x \in \Pi_p.
\end{align}
Wherever clear from the context, we will shorten the notation to $\Delta u^n_p$ for instance. The effect of boundary truncation on the overall errors and errors close to the boundary is illustrated in Figure \ref{ch4:fig:box_plot_grid_r} where we plot the errors $\Delta u^n_p(x)$ and $\Delta \bar u^n_p(x)$ on the grid for different values of $p$. 
By standard convention, the interquartile range (IQR) for the box plot is between $Q_1 = q(0.25)$ to $Q_3 = q(0.75)$ quantiles and the whiskers are at $Q_1-1.5 IQR$ and $Q_3 + 1.5 IQR$ with the hollow circles being the outliers. 
When the grid of the starting points for $\Xx_t$ is larger and the boundary is truncated, we get smaller overall errors, especially close to the boundary. 

Next, in Figure \ref{ch4:fig:max_log_err_grid}, we observe the dependence of the maximum logarithmic errors on the Picard iterations for different values of the boundary truncation parameter $p$. We note that the errors stagnate beyond $p=2$ and $n = 7$.

Next, we test Theorem \ref{thm:errornum:infinitehorizon} using our numerical example which has a constant $\sigma$. For simplicity, we have not used a weight function $\rho_{r',\eta}$ in our example because we have taken the functions $(u, \bar u)$ which are asymptotically constant outside a finite boundary. Instead, we consider the $L_\infty$ norm for the errors of $v^n = (u^n, \bar{u}^n)$ on the grid. We fix the number of Picard iterations to $10$ and only consider the impact of the mesh size and the statistical error, as the study of the error contributions of the Picard iterations is trivial. As we can see from Theorem \ref{thm:errornum:infinitehorizon}, these errors depend on $\tilde N$, $M$ and $\delta$ while we keep the size of the domain fixed by taking constant $R$. The function $v$ has a growth $C(1+|x|^r)$ (see equation \eqref{prop:representationBSDEinfinitehorizon:growth}) where $r=0$ in our case with $X_t^x = x + W_t$. With the constant $\sigma$ case for Theorem \ref{thm:errornum:infinitehorizon} without the $\rho_{r',\eta}$ function, we can write 
\begin{equation}
    \mathbb E\left[ \sup_{x\in\bR^d} |Pv^n(x) - v(x)| \right] \leq C \left( \frac{R^2}{\tilde{N}^2} +  \frac{\ln (2+ \tilde{N})}{\sqrt{M}} + (1+\varepsilon)^n\kappa_{\infty}^n \right),
\end{equation}

where $\kappa_\infty<1$ is the contraction constant, $R = \tilde N \delta$ and for practical purposes we can assume that the number of Picard iterations $n$ is large enough to ignore the third term and that $(\ln (2+\tilde{N}))^{1/2}$ can be taken as a constant for the range of $\tilde N$ where we will plot the errors. Next, to keep the first two error contributions comparable in size, $M$ should be $M=k\frac{\tilde N^4}{R^4}$ for any constant $k$ ($k=200$ in our experiment). Hence, we expect the errors to have an inverse square dependence on $\tilde N$,
$$\mathbb E\left[\sup_{x\in\bR^d} |Pv^n(x) - v(x)|\right] \propto \frac{1}{\tilde N^2}.$$


In Figure \ref{ch4:fig:erruntilde} we plot the supremum of the log errors against $\log (2+\tilde N)$ for the $10$-th Picard iteration over $10$ experiments. We also observe a slope close to the theoretical value of $-2$ for the best fit line, omitting the last two values of $\tilde N$ because for larger $\tilde{N}$, the overall error becomes so small that other sources of error have noticeable contributions.

This grid-based contraction scheme works well for a low number of dimensions but suffers from the curse of dimensionality, as the time complexity in the case of a grid of size $2\tilde{N}+1$ in each dimension is of order $O((2\tilde{N}+1)^d)$ where $d$ is the dimension of the process $X$ (Brownian motion in this case). This motivates us to consider schemes based on NNs, which we illustrate in later sections.

\subsubsection{Non-constant $\sigma$ with $X_t^x = x + \int_0^t \sigma(X^x_s) \dd W_s$}

In this subsection, we extend the grid scheme to a more general SDE with non-constant drift and volatility. To simplify matters, we will limit ourselves to the case of a null drift. With a non-constant $\sigma$, the corresponding system of equations is the following
\begin{align}\label{eq:numexp_nonconstsig}
    u(x) &= \dfrac{1}{d}\sum_{i=1}^d \tan^{-1}(x_i)\\
    \bar{u}(x) &= \nabla_x u(x)\sigma(x) =  \dfrac{1}{d}\left(\dfrac{1}{1+x_1^2},\ \dots \ ,\dfrac{1}{1+x_d^2} \right) \sigma(x)\\
    \nabla_x^2 u(x) &= \dfrac{-2}{d} \begin{pmatrix}
        \dfrac{x_1}{(1+x_1^2)^2}  & 0 & \dots & 0\\
        0 & \dfrac{x_2}{(1+x_2^2)^2} & \dots & 0 \\
        0 & 0 & \ddots & 0 \\
        0 & 0 & \dots & \dfrac{x_d}{(1+x_d^2)^2}
        \end{pmatrix}\\
    f_0(x, y, z) &= -cy + \cos{(y+|x|)} + K_z \sin(|z|)\\
    f(x, y, z) &= f_0(x,y,z) - \dfrac{1}{2}\text{Tr}\big(\nabla_x^2 u(x)\sigma(x)\sigma(x)^\top\big) - f_0(x, u(x), \nabla_x u(x)\sigma(x)).
\end{align}

Next, we take $d=1$ and $\sigma(x) = 1 + \varepsilon \tanh(x)$, with $\varepsilon = 0.9$. Note that $\varepsilon$ needs to be smaller than $1$ so that $\sigma$ is non-zero and satisfies the assumption \eqref{as:b:sigma}. In Figure \ref{ch4:fig:genSDE}, we plot various Picard iterations of the approximation against the analytical solution for $(u, \bar{u})$ and take $R = 4$, $r=3$, $M=8000$, $NbP = 10$, $\Delta t = 0.003$ and $K_z = 0.1$. 


In order to compute the Malliavin weights \eqref{eq:malliavin:weight}, we need to simulate the paths of both $X_t^x$ and $\nabla_x X_t^x$, whose dynamics can be shown to be
$$d(\nabla_x X_t^x) = \varepsilon\nabla_x X_t^x (1 - \tanh^2(X_t^x))dW_t,\quad \nabla_x X_0^x = 1.$$
Next, we proceed as follows: First, we fix a time step $\Delta t$ and obtain $M$ samples of $(E, \tilde E)$ for each point $x$ on the grid, which serve as the samples of the terminal time; Then, to vectorize the computations, we sample the paths of both processes $(X_t^x, \nabla_x X_t^x)$ until the maximum of all samples of $(E, \tilde E)$ for each $x$ and approximate their values at $E$ and $\tilde E$ using the values of $(X_t^x, \nabla_x X_t^x)$ at the discretized points on the time grid; Finally, we can proceed using the same approach we followed for constant $\sigma$.

Note that time discretization would introduce another source of error in our theoretical study of the errors in Theorem \ref{thm:errornum:infinitehorizon} whose impact could be studied, but we will not do it here.

\begin{figure}[htbp]
    \centering
    \includegraphics[scale=0.6]{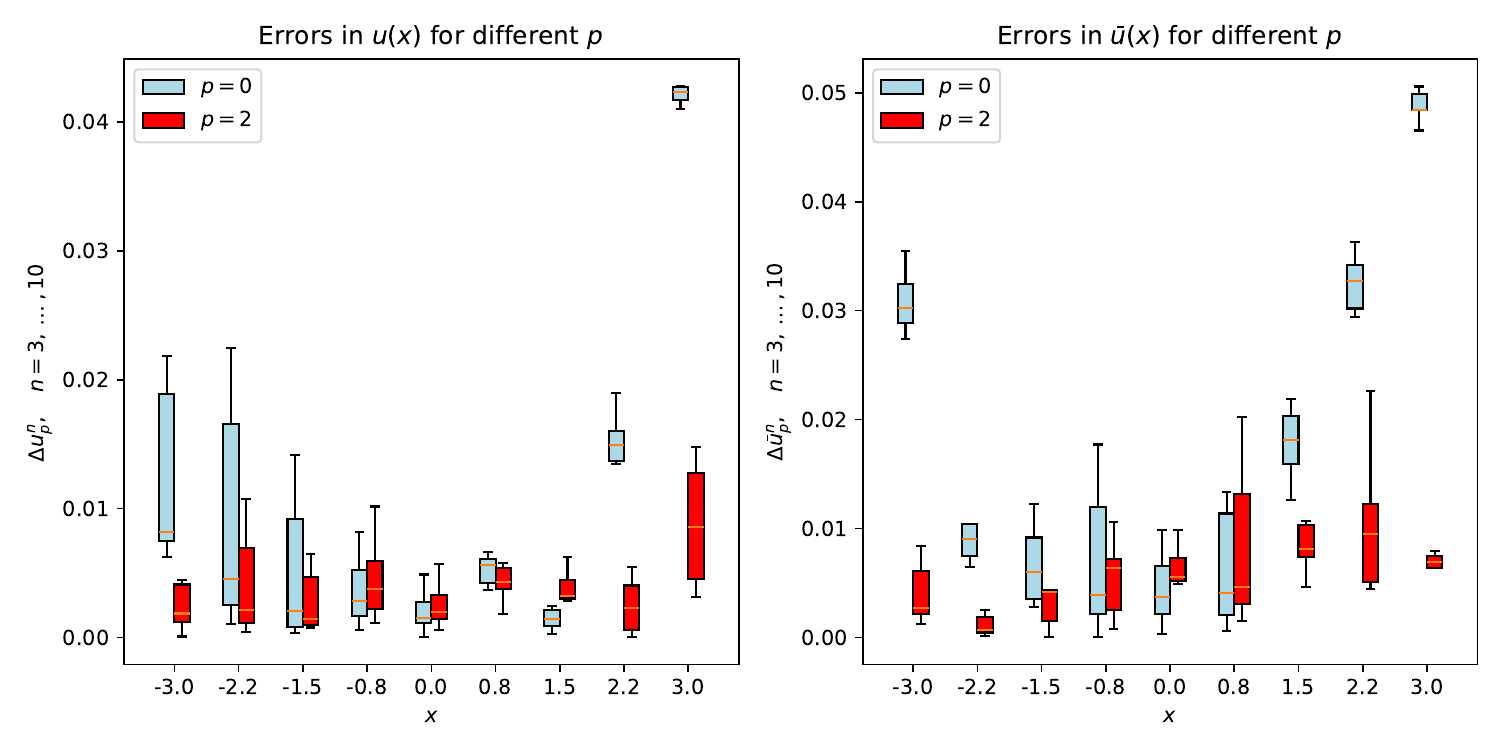}
    \caption{Box plot for errors in $u(x)$ and $\bar{u}(x)$ for $p=0$ and $2$ for Picard iterations $n=3,\dots,10$ plotted against the grid points.}\label{ch4:fig:box_plot_grid_r}
\end{figure}

\begin{figure}[htbp]
    \centering
	\includegraphics[scale=0.6]{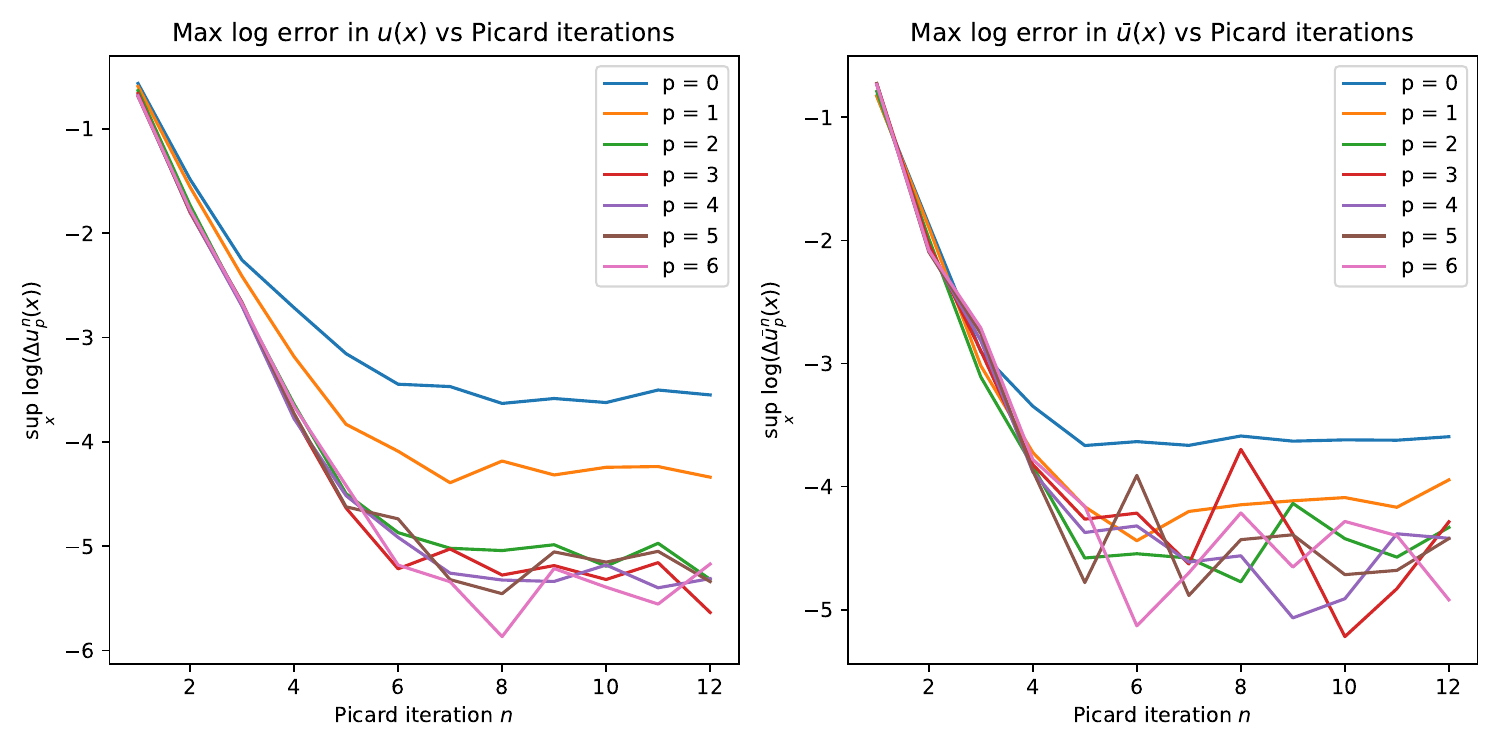}
    \caption{Plot of the max log errors on the grid against Picard iterations for different values of boundary truncation p.}
    \label{ch4:fig:max_log_err_grid}
\end{figure}

\begin{figure}[htbp]
    \centering
    \includegraphics[scale=0.62]{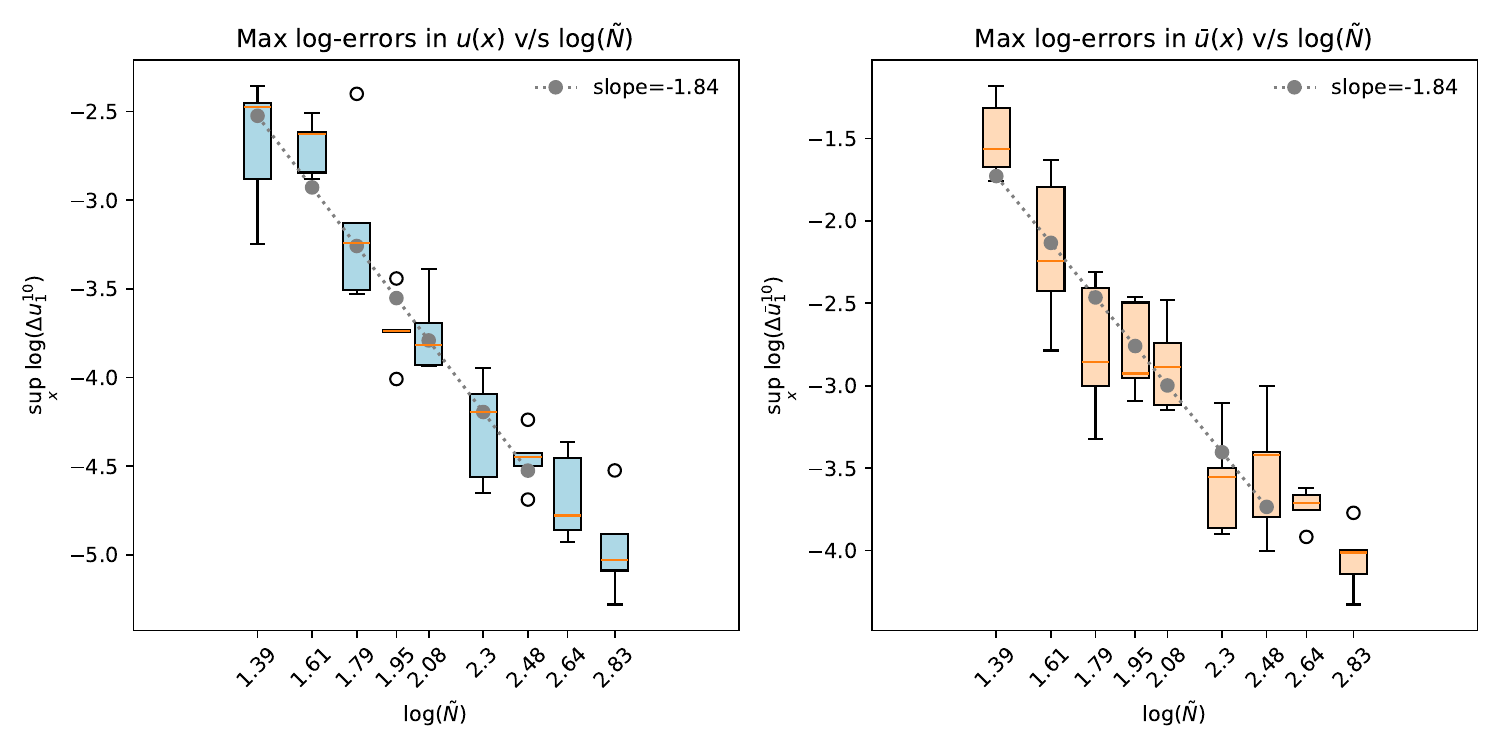}
    \caption{Box plot for the supremum of the log errors of $u(x)$ and $\bar{u}(x)$ over the grid for $p=2$ and different $\tilde N$, plotted for $10$ experiments.}\label{ch4:fig:erruntilde}
\end{figure}

\begin{figure}[htbp]
    \centering
    \includegraphics[scale=0.57]{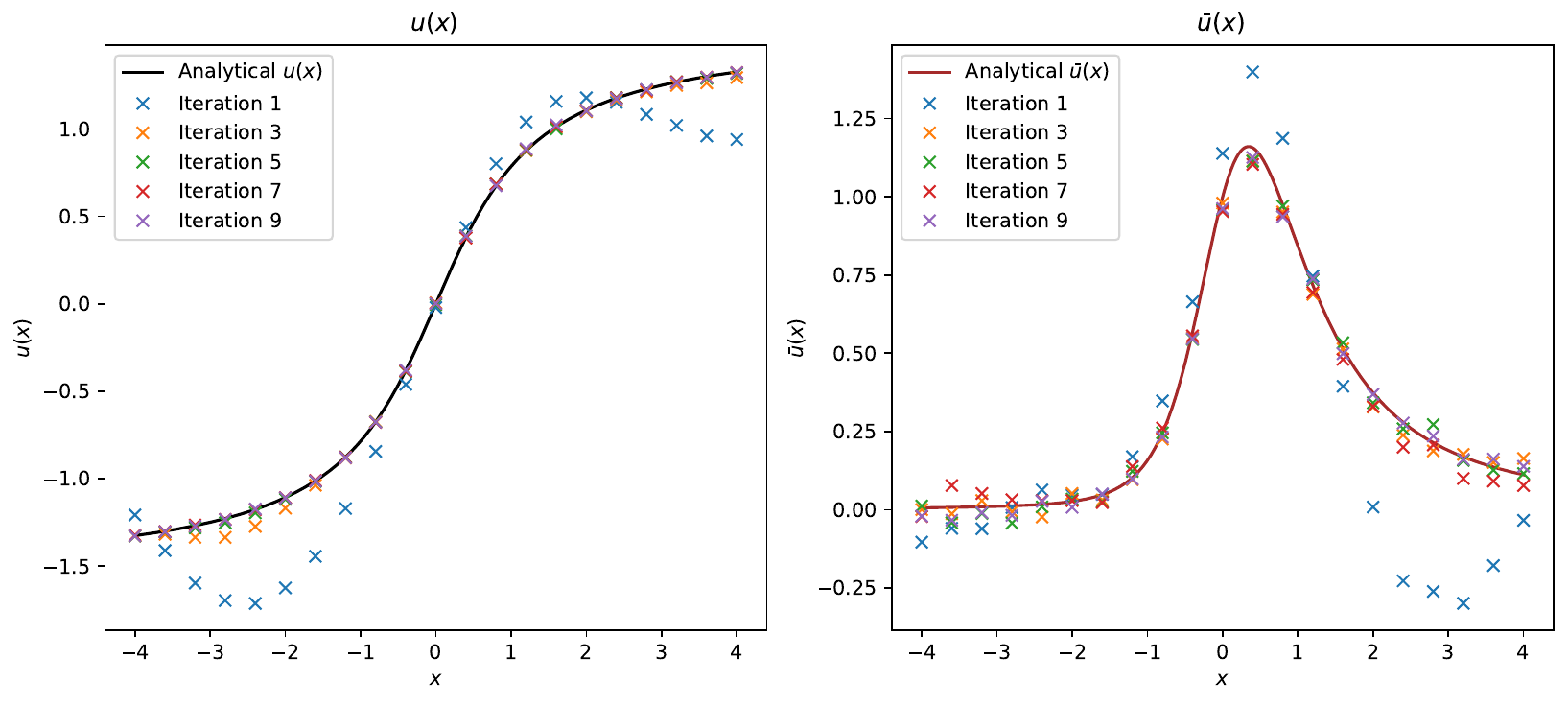}
    \caption{Picard iterations for the general SDE with $\sigma(x) = 1 + \varepsilon \tanh(x)$, $\varepsilon = 0.9$}
    \label{ch4:fig:genSDE}
\end{figure}


\subsection{Second scheme - Contraction based NN scheme}

The second scheme also uses contraction, but relies on NNs for the approximation procedure (see Algorithm \ref{ch4:alg:contractionNN}). We have used the following values for the parameters: $c = 2$, $K_z = 0.1$, $a = 2$, $\ta = 2$, $\theta = 1.5$ and $\ttheta = 1.5$ unless specified otherwise. For the distribution of the starting point $X_0$, we use a normal distribution $\cN(0, {\sigma^2 \bI_{d}})$ with $\sigma = 2.0$ since we obtain the plots within the domain $[-3, 3]$ for each dimension and $\sigma=2$ gives us a good balance of coverage of the domain and ensures that enough samples are outside this domain so that the errors close to the observation boundary are small. We use a fully connected NN with two hidden layers, each with $20 + d$ nodes and we use ReLU activation function. The initial learning rate is $5\times 10^{-4}$ which decays by a factor of $0.9$ after $1000$ steps. In Figure \ref{ch4:fig:NN_Pic_iter}, we provide a plot of the solution for the first five Picard iterations for the one-dimensional case, i.e. $d=1$ with the sample size $M = 30000$.


In Figure \ref{ch4:fig:NN_Pic_iter_2d}, we plot the odd numbered Picard iterations from $n=1$ to $n=7$ for the solution corresponding to $d=2$ and have used a sample size of $M=40000$ at each iteration. Compared to the one-dimensional case, we need a larger sample size and more Picard iterations for comparable errors which is to be expected. 




To study the effect of $d$ on the errors, we used $M=40000$ and obtained $M_{err} = 1000$ test samples with starting points $X_0$ sampled from $\cN(0, \sigma^2_{err}{\bI_{d}})$ and plotted the errors for different values of $d$, as illustrated in Figure \ref{ch4:fig:box_err}. We took $\sigma_{err} = 2.0$ which is the same as the one used for training. If we use a smaller $\sigma_{err}$ for testing, we observe smaller errors. This effect is similar to the truncation of boundary illustrated for the grid based scheme. The $L^2_{\mu_0}$-norm of the errors is approximated as follows:
$$\norm{u_d^n - u}_{L^2_{\mu_0}} \approx \left(\frac{1}{M_{err}}\sum_{i=1}^{M_{err}} |u^n(x_i)-u(x_i)|^2\right)^{1/2}$$
and
$$\norm{\bar{u}_d^n - \bar{u}}_{L^2_{\mu_0}} \approx \left(\frac{1}{M_{err}}\sum_{i=1}^{M_{err}} \sum_{k=1}^d |\bar{u}^{n,k}(x_i)-\bar{u}^k(x_i)|^2\right)^{1/2}.$$

Next, using the expressions for the analytical solutions, we can show that,
$$\norm{u}_{L^2_{\mu_0}} =  \left(\int_{\bR^d} |u(x)|^2\mu_0(\dx)\right)^{1/2} = \frac{1}{\sqrt{d}} \left( \int_\bR |\tan^{-1}(x)|^2\frac{1}{\sqrt{2\pi\sigma_{err}^2}}e^{-\frac{x^2}{2\sigma_{err}^2}}\dx\right)^{1/2},$$
$$\norm{\bar{u}}_{L^2_{\mu_0}} =  \left(\int_{\bR^d} |\bar{u}(x)|^2 \mu_0(\dx)\right)^{1/2} = \frac{1}{\sqrt{d}} \left( \int_\bR \left(\frac{1}{1+x^2}\right)^2 \frac{1}{\sqrt{2\pi\sigma_{err}^2}}e^{-\frac{x^2}{2\sigma_{err}^2}}\dx\right)^{1/2}, $$
implying $1/\sqrt{d}$ dependence on $d$ for the norm of the solution. Hence, we consider the relative $L^2_{\mu_0}$ errors in this experiment, defined as follows,
\begin{equation}\label{eq:rel_L2norm_err}
\Delta u^n_d := \frac{\norm{u_d^n - u}_{L^2_{\mu_0}}}{\norm{u}_{L^2_{\mu_0}}}\quad \text{and} \quad \Delta \bar{u}^n_d := \frac{\norm{\bar{u}_d^n - \bar{u}}_{L^2_{\mu_0}}}{\norm{\bar{u}}_{L^2_{\mu_0}}}.
\end{equation}


Next, we take $d=1$ and investigate the effect of $K_z$ on the $L^2_{\mu_0}$ norm of the errors, denoted by $\Delta u^n_{K_z}$ and $\Delta \bar{u}^n_{K_z}$. See Remark \ref{ch4:rem:contraction} to see how increasing $K_z$ can deny the contraction of the scheme, which is why we have taken small values of $K_z$ so far. 
{Theoretically speaking, the existence and uniqueness of the solution is guaranteed for all $K_z \geq 0$ since $d'=1$ and $r=0$, see Proposition \ref{prop:existence:uniqueness:BSDE}. } 
In Figure \ref{ch4:fig:box_err_Kz}, we plot the logarithmic errors for different values of $K_z$, with $M=30000$, in order to identify the value of $K_z$ beyond which the scheme breaks down. We observe that from $K_z=1.6$ and onward, the errors increase significantly.

\begin{figure}
    \centering
	\includegraphics[scale=0.56]{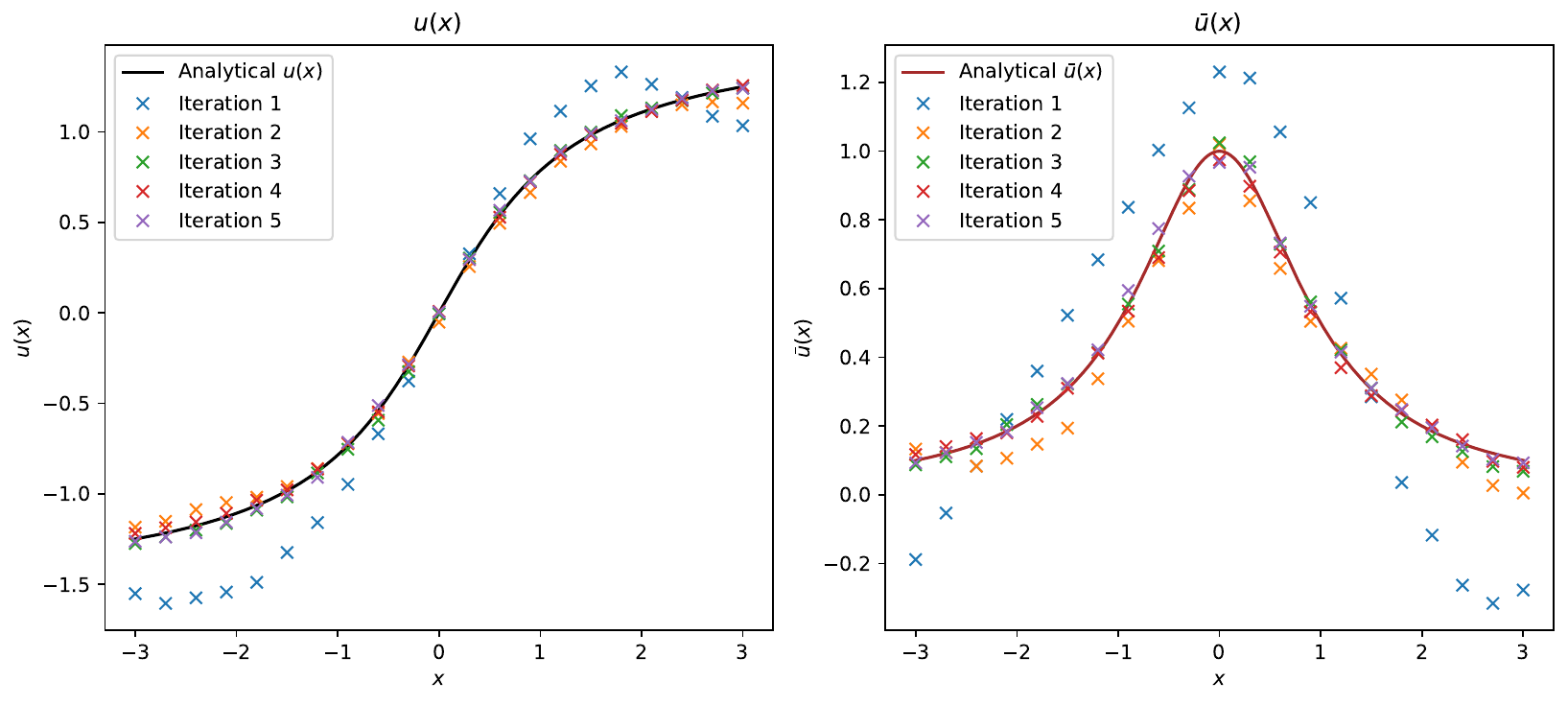}
    \caption{Picard iterations for Algorithm \ref{ch4:alg:contractionNN} for $d=1$.}
    \label{ch4:fig:NN_Pic_iter}
\end{figure}

\begin{figure}
\centering
\includegraphics[scale=0.61]{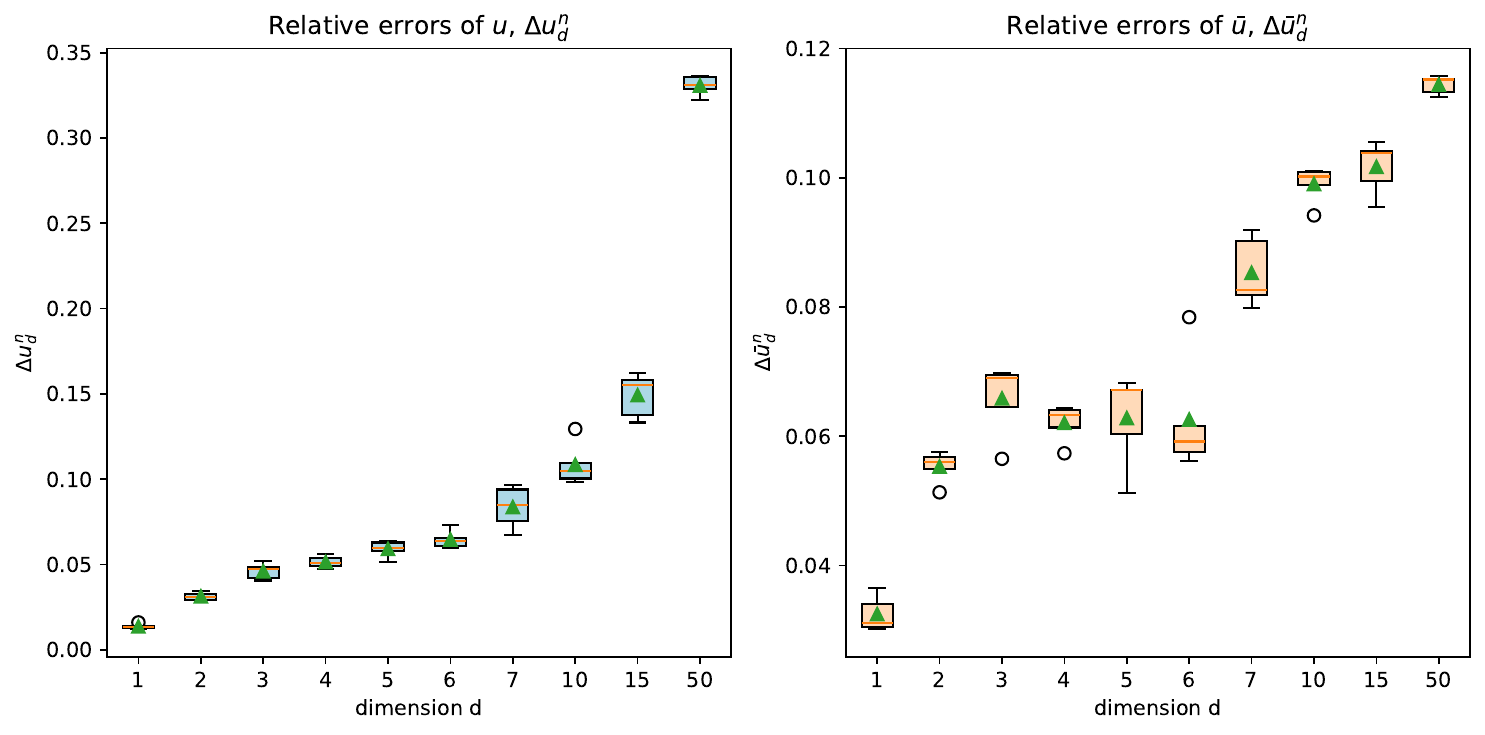}
\caption{Relative errors $\Delta u^n_d$ and $\Delta \bar{u}^n_d$ (see \eqref{eq:rel_L2norm_err}) for $n=5$ and different values of $d$ between $1$ and $50$ for 5 experiments.} 
\label{ch4:fig:box_err}
\end{figure}

\begin{figure}
\centering
\includegraphics[scale=0.61]{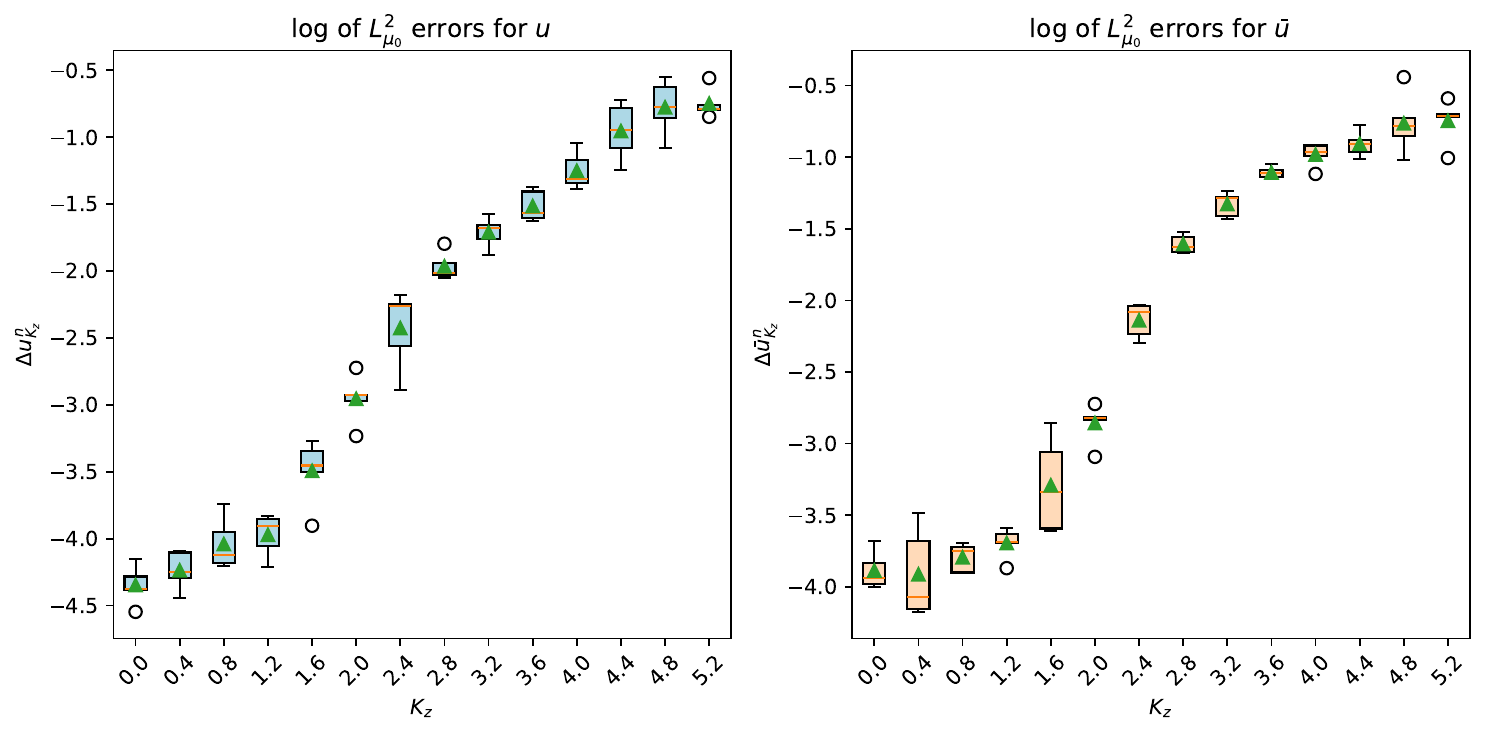}
\caption{Logarithm of $L^2_{\mu_0}$ errors $\Delta u^n_{K_z}$ and $\Delta \bar{u}^n_{K_z}$ for $n=5$ with different values of $K_z \in \{0, 0.4, \dots, 5.2 \}$ for $5$ experiments for the NN based iterative contraction scheme, Algorithm \ref{ch4:alg:contractionNN}.}
\label{ch4:fig:box_err_Kz} 
\end{figure}

\begin{figure}
    \centering
	\includegraphics[scale=0.34]{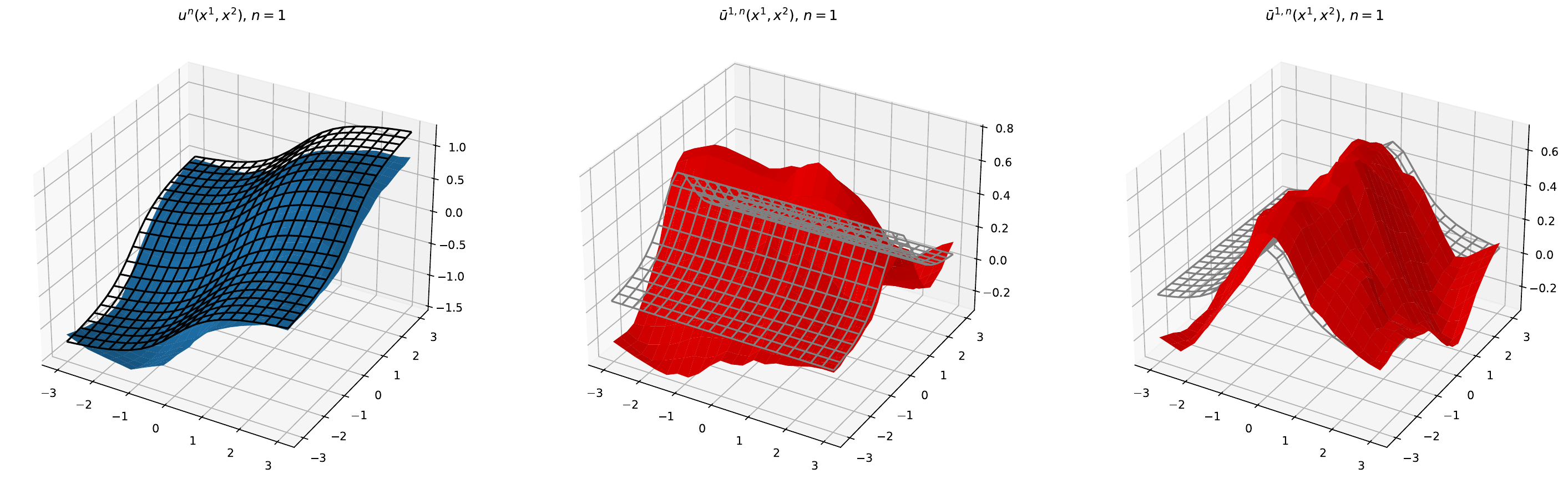}
	\includegraphics[scale=0.34]{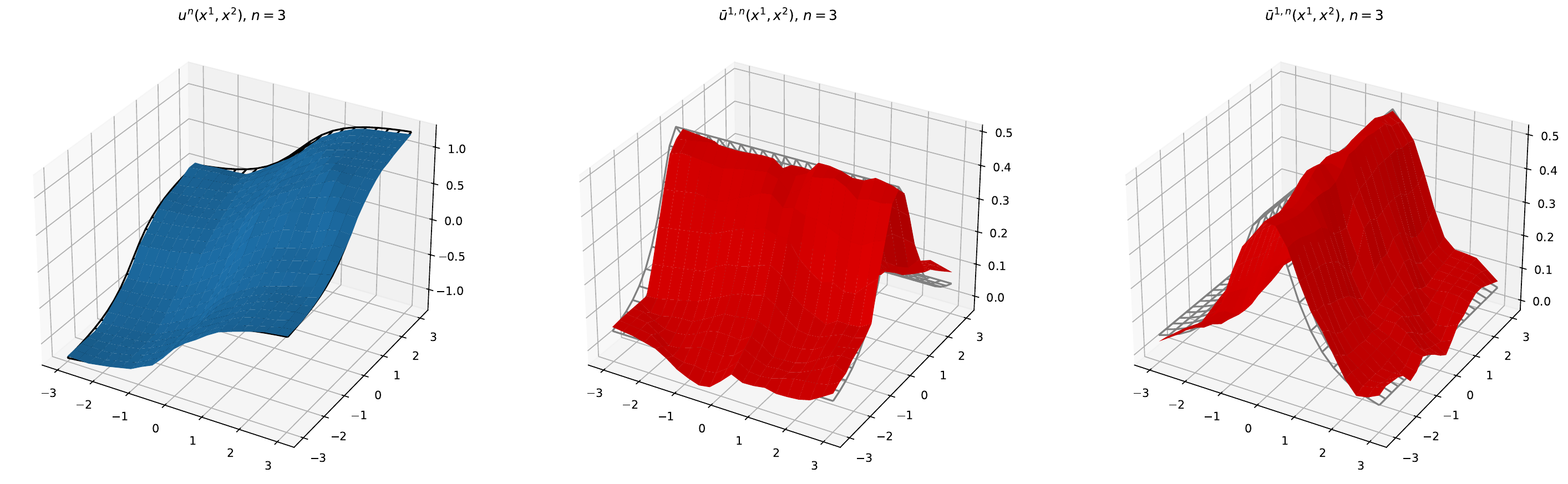}
	\includegraphics[scale=0.34]{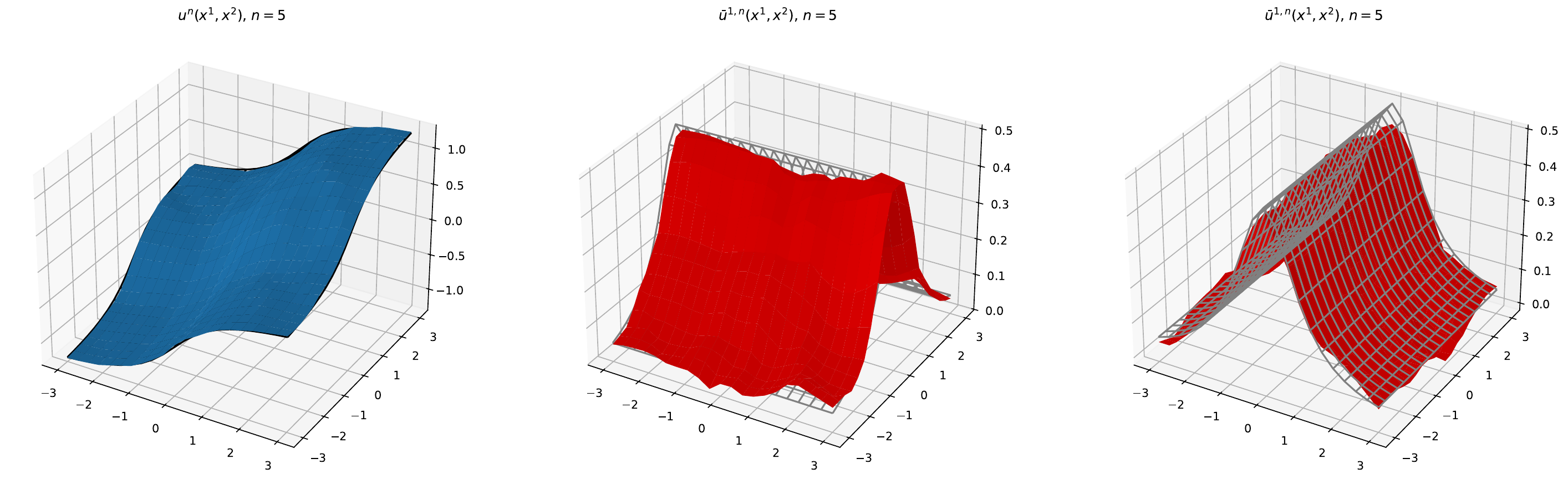}
	\includegraphics[scale=0.34]{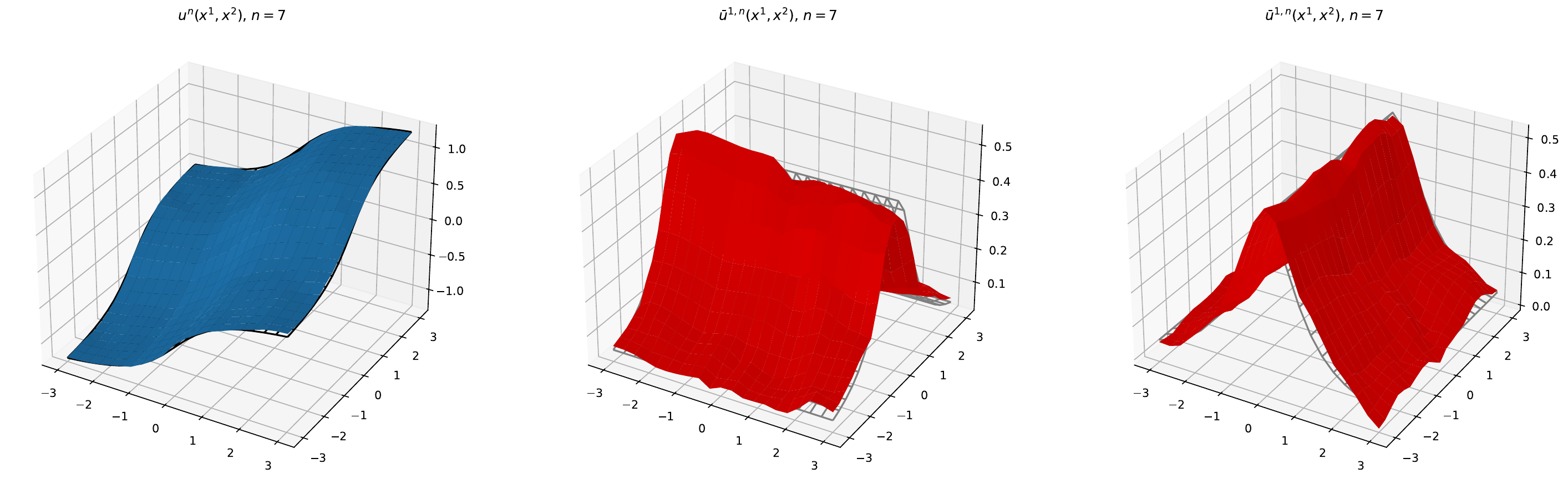}
    \caption{Picard iterations (Algorithm \ref{ch4:alg:contractionNN}) for $u^n$, $\bar{u}^{1,n}$ and $\bar{u}^{2,n}$ for $d=2$ and $n=1, 3, 5$ and $7$ (from top to bottom). The wire-frame plots are used for the analytical solution.}
    \label{ch4:fig:NN_Pic_iter_2d}
\end{figure}


\subsection{Third scheme - Direct scheme using NNs}


The third scheme is based on a direct approach that does not rely on a direct Picard scheme (see Algorithm \ref{ch4:alg:directNN}). We use the same NN architecture as the previous scheme and use an initial learning rate $5\times 10^{-4}$ which decays by a factor of $0.6$ after $500$ steps.  {We use an epoch length of $100$, which is the number of optimization updates between the input and target refreshes, and the epoch count is fixed to $30$.} We use $M_x = 512$ samples of the starting point $x$ and for each $x$, we use $M=5000$ independent samples of the exponential and gamma times along with the corresponding state process and the Malliavin weight evaluated at these time samples. The values of the system parameters are: $c = 2$, $K_z = 0.1$, $a = 2$, $\ta = 2$, $\theta = 1.5$ and $\ttheta = 1.5$ unless otherwise specified. In Figure \ref{ch4:fig:NN_dir_2d}, we plot the solution for $d=2$. 


{Similarly to the previous scheme, Figure~\ref{ch4:fig:err_Kz_direct}
displays the logarithm of the $L^2_{\mu_0}$ error norms $\Delta u_{K_z}$ and $\Delta \bar{u}_{K_z}$ (with no index $n$, as there are no Picard iterations) for values of $K_z$ ranging from $0$ to $5.2$, well beyond the theoretical guarantee for the contraction.}
 {Unlike the second scheme, the error increases much less sharply with $K_z$, even well beyond the theoretical guarantee for the contraction. We do not have a theoretical explanation for this behaviour, which may be due to the partial resolution of the regression problem at each epoch and to the warm start of the NN.}


\begin{figure}
\centering
\includegraphics[scale=0.34]{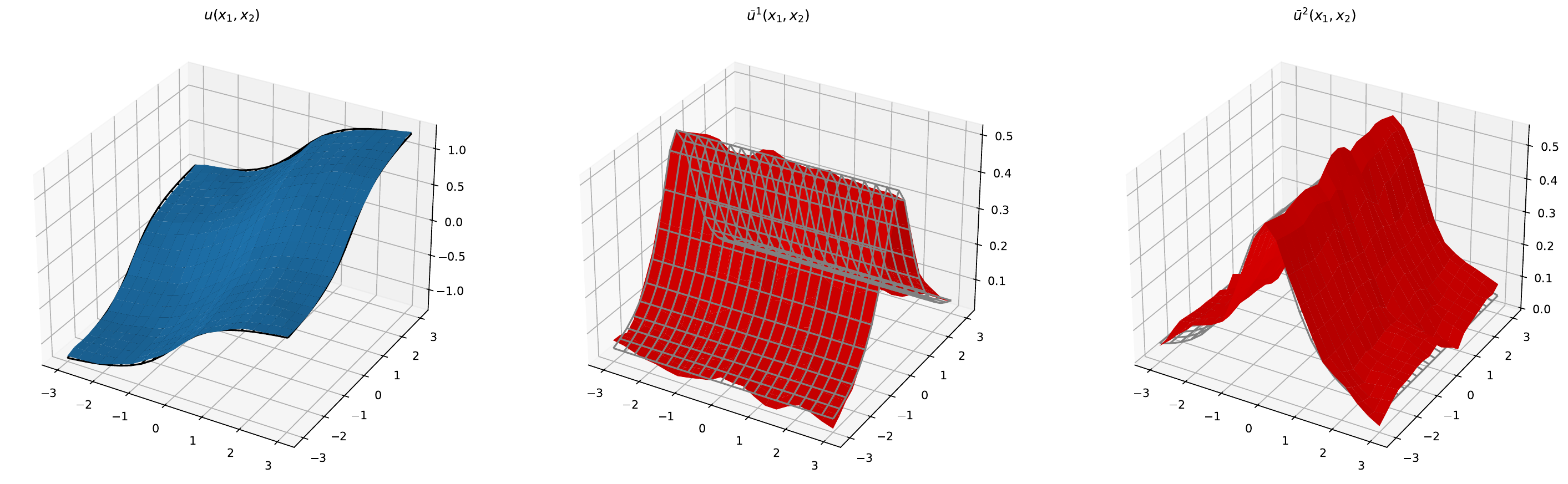}
\caption{Solution $u(x)$ and $\bar{u}(x)$ for $d=2$ for the NN based direct scheme (Algorithm \ref{ch4:alg:directNN}).}
\label{ch4:fig:NN_dir_2d}
\end{figure}

\begin{figure}
\centering
\includegraphics[scale=0.61]{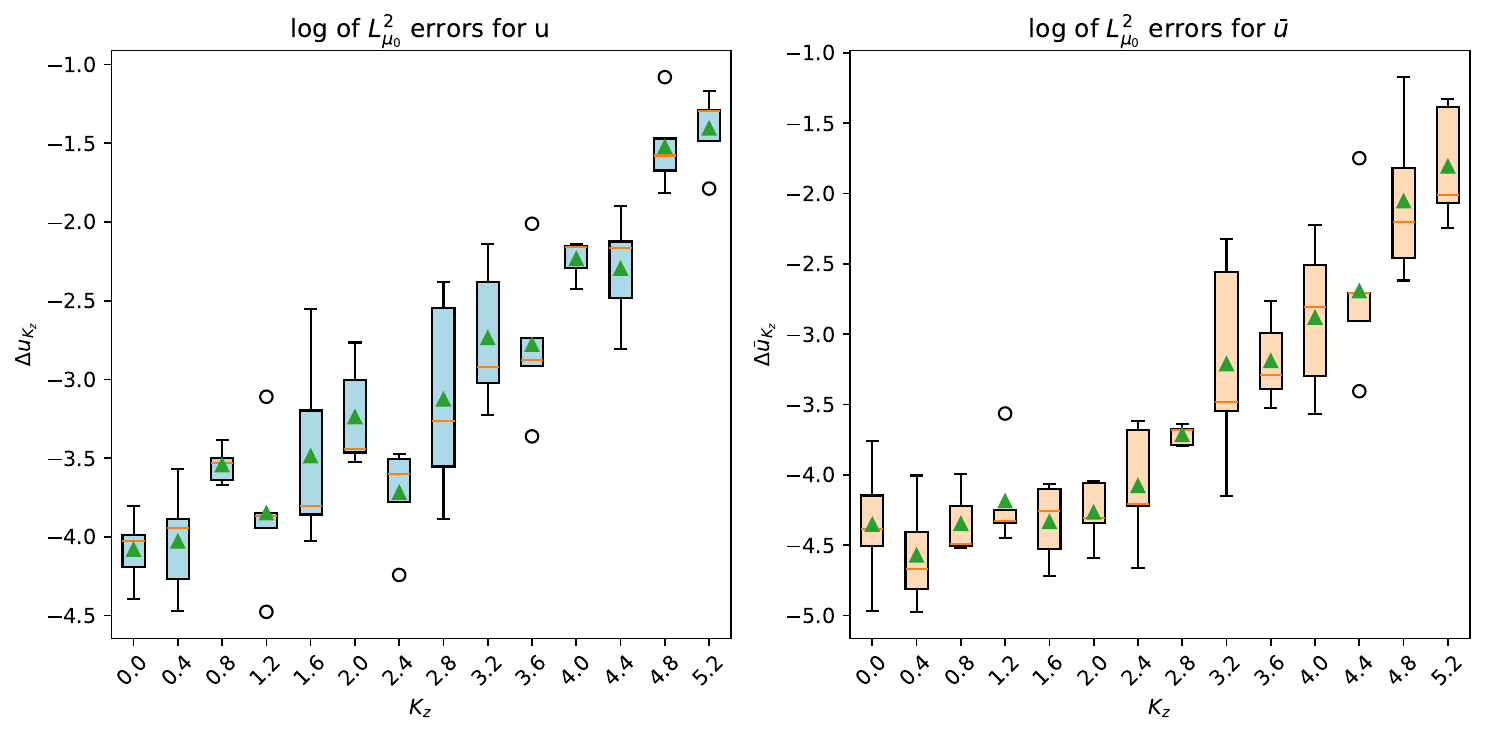}
\caption{ {Logarithm of $L^2_{\mu_0}$ errors $\Delta u_{K_z}$ and $\Delta \bar{u}_{K_z}$ for $M=3000$ and $M_x=512$ with different values of $K_z \in \{0, 0.4, \dots, 5.2 \}$ for $5$ experiments for the NN based direct scheme, Algorithm \ref{ch4:alg:directNN}.}}
\label{ch4:fig:err_Kz_direct}
\end{figure}

 {
\subsection{Accuracy and computational cost}

Figure~\ref{fig:scheme_comparison} compares the three schemes for
the Brownian example with $d=1$ and $K_z=0.1$, using solve-time
budgets of $5$, $10$, $20$, $30$, and $100$ seconds. For each
scheme and budget, we first perform one pilot run for each of
three preset parameter configurations. Grid Picard tests
$(\tilde N,M)\in\{(10,2000),(20,10000),(40,40000)\}$;
NN Picard tests $M\in\{10000,30000,60000\}$; and direct NN tests
$M\in\{300,1000,3000\}$. We select the configuration with the
smallest $\max(\Delta u,\Delta\bar u)$ on $10^4$ held-out
validation points. We then perform five fresh runs of that
configuration, each with a different random seed and the same
time budget. Their errors are measured on a separate, common
set of $10^5$ test points; both validation and test locations
are drawn from $\mathcal N(0,4)$. Table~\ref{tab:cpu_configs}
lists the selected configurations. Here, $M$ counts Monte Carlo
samples per grid node for grid Picard, training points per
Picard iteration for NN Picard, and inner Monte Carlo samples
per starting point for direct NN. Both networks have two
21-unit ReLU hidden layers with batch normalization and two
linear outputs. NN Picard uses batches of $1000$ and trains
for $200$ passes through each dataset before advancing to the
next Picard iteration, unless the time budget expires earlier.
Direct NN uses batches of $512$ and an epoch length of $100$
optimizer updates between input and target refreshes.
All solves use one CPU core and float64 arithmetic. Plotted
times include construction, tracing, sampling, target generation,
and training; pilot-run costs, runtime initialization, and
evaluation are excluded. Each box summarizes the five final
test errors using the last completed checkpoint within the
budget.

\begin{table}[t]
    \centering
    \caption{Configurations selected for each solve-time budget.
    Grid Picard uses $2\tilde N+1$ nodes on $[-6,6]$.}
    \label{tab:cpu_configs}
    \begin{tabular}{r r r r r}
        \hline
        Budget (s) & \multicolumn{2}{c}{Grid Picard}
                   & NN Picard & Direct NN \\
        \cline{2-3}
                   & $\tilde N$ & $M$ & $M$ & $M$ \\
        \hline
          5 & 20 & $10\,000$ & $60\,000$ &   300 \\
         10 & 40 & $40\,000$ & $10\,000$ &   300 \\
         20 & 40 & $40\,000$ & $10\,000$ &   300 \\
         30 & 40 & $40\,000$ & $10\,000$ &   300 \\
        100 & 40 & $40\,000$ & $30\,000$ & $3000$ \\
        \hline
    \end{tabular}
\end{table}

Grid Picard has the lowest median error at every tested budget,
reaching approximately $1\%$ within $10$ seconds. At $100$ seconds,
the median errors are $0.68\%$, $3.97\%$, and $4.43\%$ for grid
Picard, direct NN, and NN Picard, respectively. The grid's eventual
plateau is consistent with Theorem~3.10: contraction reduces the
Picard error, while sampling, interpolation, and spatial truncation
errors remain. For NN Picard, Proposition~3.11 additionally retains
the approximation contribution
$\sup_{1\leq k\leq n}r_k/(1-\kappa_2)$, even under exact optimization.
Thus, contraction alone does not ensure comparable accuracy across
the two representations. The direct scheme avoids the contraction
requirement, but no convergence rate is established here; its
advantage beyond the contraction regime is not assessed by this
small-$K_z$ experiment.

\begin{figure}
\centering
\includegraphics[scale=0.61]{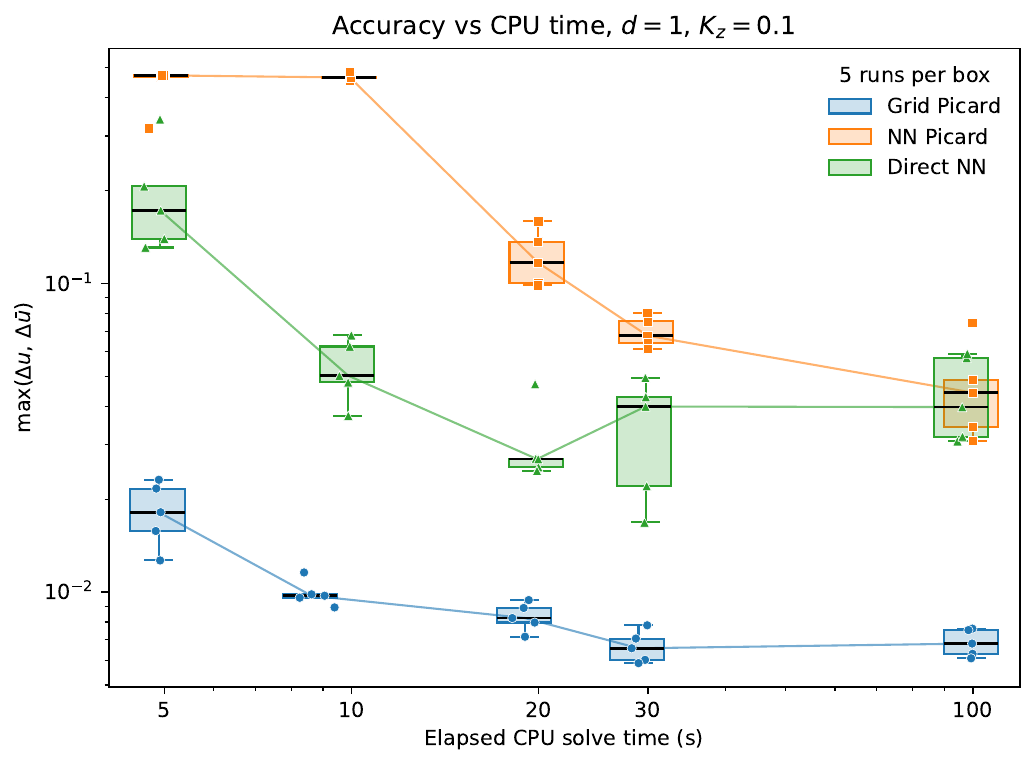}
\caption{Accuracy versus elapsed solve time for the three schemes
in the Brownian example with $d=1$ and $K_z=0.1$, using one CPU core
and float64 arithmetic. Boxes show the interquartile range of
$\max(\Delta u,\Delta\bar{u})$ over five independent runs, with black
median bars and whiskers following the $1.5\,\mathrm{IQR}$ rule.
Markers show individual runs; colored lines connect medians.
Runtime initialization is excluded from solve time.}
\label{fig:scheme_comparison}
\end{figure}
}

\appendix

\section{Proofs}
\label{section:proofs}
\subsection{Proof of Proposition \ref{prop:representationBSDEinfinitehorizon}}
\label{proof:prop:representationBSDEinfinitehorizon}
{1. Let us start by proving a probabilistic representation for $u$. Applying It\^o formula to $e^{-at}\Yx_t$ we get, for all $0 \leqslant T$, 
 \begin{equation}
 \label{proof:representationBSDEinfinitehorizon:eq1}
 u(x)=\Yx_0 =\Exp{ e^{-aT}\Yx_T+ \int_0^T e^{-as} \left( f(\Xx_s,\Yx_s,\Zx_s)+a\Yx_s \right) \ds }.
 \end{equation}
By using Proposition \ref{prop:existence:uniqueness:BSDE} and Young's inequality, we have 
\begin{align*}
  & \Exp{\int_0^{+\infty} e^{-as} |f(\Xx_s,\Yx_s,\Zx_s)+a\Yx_s|\ds}\\
 \leqslant & C\Exp{ \int_0^{+\infty} e^{-(2a-\lambda)s} + e^{-\lambda s}\left(|\Yx_s|^2 + \norm{\Zx_s}^2 + |\Xx_s|^{2r}\right)\ds} <+\infty,
\end{align*}
where we have set $\lambda:=2\mu-{K_{f,z}^2}(\1_{d'>1}\vee \1_{r>0})$. The Lebesgue theorem gives us that the integral term in \eqref{proof:representationBSDEinfinitehorizon:eq1} tends to   
\begin{align*}
 \Exp{\int_0^{+\infty} e^{-as} (f(\Xx_s,\Yx_s,\Zx_s)+a\Yx_s)\ds}
\end{align*}
when $T \rightarrow + \infty$. Moreover, by Proposition \ref{prop:existence:uniqueness:BSDE} we get 
$$\lim_{T \rightarrow +\infty} \Exp{e^{-aT}|\Yx_T|} \leqslant \lim_{T \rightarrow +\infty} \left(e^{-(2a-\lambda)T}+ \Exp{e^{- \lambda T}|\Yx_T|^2}\right)  =0.$$
Thus, we have
\begin{align*}
 u(x) &= \E{\int_0^\infty  e^{-as}(f(\Xx_s,u(\Xx_s),\Zx_s)+a u(\Xx_s)) \ds}
\end{align*}
and the growth on $u$ given by \eqref{prop:representationBSDEinfinitehorizon:growth} follows from Proposition \ref{prop:existence:uniqueness:BSDE} and \eqref{prop:representationBSDEinfinitehorizon:hyp:rho}.}

{2.a. Now let us assume {temporarily} that $f \in C^1_b(\bR^{d} \times \bR^{d'} \times \bR^{d' \times d},\bR^{d'})$ and let us consider the following finite horizon BSDE, with a null terminal condition,}
\begin{align*}
 Y_t^{N,x} =  {0 +} \int_t^N f(\Xx_s,Y_s^{N,x},Z_s^{N,x})\ds -\int_t^N Z_s^{N,x} \dW_s, \quad 0 \leqslant t \leqslant N,
\end{align*}
with solution $(Y^{N,x},Z^{N,x}) \in \mathscr{S}^2_N \times \mathscr{M}^2_N$ defined on $[0,N]$ for any $x \in \bR^{d}$ and $N>0$. We can set $Y^{N,x}_t:=0$ and $Z^{N,x}_t:=0$ for all $t>N$. Let us remark that for any $\ta>0$ we also have 
\begin{align}
\label{eq:BSDE:finite:horizon}
 Y_t^{N,x} = \int_t^N e^{-\ta(s-t)} \left(f(\Xx_s,Y_s^{N,x},Z_s^{N,x}) +\ta Y_s^{N,x}\right)\ds -\int_t^N e^{-\ta(s-t)} Z_s^{N,x} \dW_s, \quad 0 \leqslant t \leqslant N.
\end{align}
By applying  \cite[Theorem 4.2]{ma:zhan:02:1}, we have the representation
$$Y^{N,x}_t = u_N(t,\Xx_t), \quad \forall t \in [0,N],$$
with $u_N(t,.) \in C^{1}( \bR^d,\bR^{d'})$ for all $t \in [0,N)$ and there exists a continuous version of $Z^{N,x}$ given by
$$Z^{N,x}_t = \nabla_x u_N(t,\Xx_t)\sigma (\Xx_t) := \bar{u}_N(t,\Xx_t), \quad t \in [0,N),$$
and
\begin{align}
\bar{u}_N(t,\Xx_t) =& \bE\bigg[ \int_t^N e^{-\ta(s-t)}\left(f(\Xx_s,Y_s^{N,x},Z_s^{N,x})+\ta Y_s^{N,x}\right)\\
& \times\frac{1}{s-t}\int_t^s \langle \sigma^{-1}(X_r^x)\nabla_x X_r^x (\nabla \Xx_t)^{-1}, \dW_r \rangle \sigma(\Xx_t) \ds\big|\cF_t \bigg].
\end{align}

Finally, from proofs of  \cite[Theorem 5.57]{pard:rasc:14} and  \cite[Lemma 3.1]{Briand-Hu-98} we also have a uniform growth estimate for $u_N$: there exists a constant $C>0$ such that
\begin{align}
 \label{ineq:growth:uN}
 |u_N(t,x)| \leqslant C(1+|x|^r), \quad \forall t \geqslant 0, x \in \bR^d, N>0. 
\end{align}
Now we {pass to the limit as $N$ goes} to infinity. In order to do it we denote, for all $x \in \bR^d$, 
\begin{align*}
 \bar{u}(x)= \Exp{ \int_0^{+\infty} e^{-\ta s}\left(f(X_s^{x},Y_s^{x},Z_s^{x})+\ta Y_s^{x}\right)\Ux_s\ds},
\end{align*}
and we firstly show that $\bar{u}$ is well defined. Recalling the growth of $u$ (resp. $u_N$), we can use the Bismut-Elworthy representation for $\Zx$ (resp. $Z^{N,x}$) on the time interval $[0,1]$, see e.g.  \cite[Corollary 3.11]{Fuhrman-Tessitore-02}, to get that 
\begin{equation}
\label{bound:Z:ZN:01}
\norm{\Zx_s} + \norm{Z^{N,x}_s} \leqslant C(1+|\Xx_s|^{r}),\quad  \forall (s,x) \in [0,1] \times \bR^d, N \geqslant 2.
\end{equation}

 {Then, we obtain
\begin{align*}
& \Exp{ \int_0^{+\infty} e^{-\ta s}\norm{\left(f(X_s^{x},Y_s^{x},Z_s^{x})+\ta Y_s^{x}\right)\Ux_s}\ds }\\
\leqslant &C\Exp{ (1+\sup_{s \in [0,1]} |\Xx_s|^{r})\int_0^{1} \frac1s \norm{\int_0^s \langle \sigma^{-1}(\Xx_u)\nabla_x \Xx_u , \dW_u\rangle}\ds }\\
&+ C\Exp{\int_1^{+\infty} e^{-\ta s}\left(1+|\Xx_s|^r+\norm{\Zx_s}\right)\norm{\int_0^s \langle \sigma^{-1}(\Xx_u)\nabla_x \Xx_u , \dW_u\rangle}\ds}
\end{align*}
by applying the growth of $f$, the growth of $u$ given by \eqref{prop:representationBSDEinfinitehorizon:growth} and the estimate \eqref{bound:Z:ZN:01}. 
By using Cauchy-Schwarz inequality, Fubini theorem and the boundedness of $\sigma^{-1}$, the previous inequality becomes
\begin{align*}
& \Exp{ \int_0^{+\infty} e^{-\ta s}\norm{\left(f(X_s^{x},Y_s^{x},Z_s^{x})+\ta Y_s^{x}\right)\Ux_s}\ds }\\
\leqslant & C\left(1+\Exp{\sup_{s \in [0,1]} |\Xx_s|^{2r}}^{1/2}\right)\left(\int_0^1 \frac{\ds}{s^{1/2}}\right)^{1/2}\left(\int_0^1 \frac{1}{s^{3/2}} \int_0^s \E{\norm{\nabla_x X_u^x}^2}  \du \ds \right)^{1/2}\\
&+ C\Exp{\int_0^{+\infty} e^{-\ta s}(1+|\Xx_s|^{2r}+\norm{\Zx_s}^2)\ds}^{1/2}\Exp{\int_0^{+\infty} e^{-\ta s}\norm{\nabla_x \Xx_s}^2\ds}^{1/2}.
\end{align*}
Finally, we just have to apply estimates on $Z^x$ given by Proposition \ref{prop:existence:uniqueness:BSDE}, \eqref{prop:representationBSDEinfinitehorizon:hyp:b}, \eqref{prop:representationBSDEinfinitehorizon:hyp:rho} and standard estimates on $\Xx$ and $\nabla_x \Xx$ to obtain 
$$\Exp{ \int_0^{+\infty} e^{-\ta s}\norm{\left(f(X_s^{x},Y_s^{x},Z_s^{x})+\ta Y_s^{x}\right)\Ux_s}\ds } \leqslant C(1+|x|^r).$$}

{Now, we estimate the error between $\bar{u}_N(0,.)$ and $\bar{u}$.
Thanks to  \eqref{as:generatorf} we have
\begin{align*}
 & \norm{\nabla_x u_N(0,x)-\bar{u}(x)\sigma^{-1}(x)} \\
 \leqslant & C\norm{\sigma^{-1}(x)}\Exp{\int_0^{+\infty} e^{-\ta s}\left(|Y_s^{N,x}-Y_s^{x}|+\norm{Z_s^{N,x}-Z^{x}_s}\right)|U_s^{x}|\ds}\\
 &+C\norm{\sigma^{-1}(x)}\Exp{\int_N^{+\infty} e^{-\ta s} |f(X_s^{x},0,0)||U_s^{x}|\ds}
\end{align*}
and by using same calculations as previously, we have, for all $x \in \bR^d$, $N\geqslant 2$ and $\varepsilon \in (0,1]$,
\begin{align*}
 & \norm{\nabla_x u_N(0,x)-\bar{u}(x)\sigma^{-1}(x)} \\
 \leqslant & C \Exp{\sup_{s \in [0,1]} (|\Yx_s|+|Y_s^{N,x}|+\norm{\Zx_s}+\norm{Z_s^{N,x}})\int_0^{\varepsilon} |\Ux_s| \ds}\\
 &+C\Exp{\int_{\varepsilon}^{+\infty} e^{-\ta s} |\Ux_s|^2 \ds }^{1/2} \Exp{\int_0^{+\infty} e^{-\ta s}\left(|Y_s^{N,x}-Y_s^{x}|^2+\norm{Z_s^{N,x}-Z^{x}_s}^2\right)\ds}^{1/2}\\
 &+C \Exp{\int_1^{+\infty} e^{-\ta s} \norm{\nabla_x \Xx_s}^2 \ds }^{1/2} \Exp{\int_N^{+\infty} e^{-\ta s} (1+|\Xx_s|^{2r})\ds}^{1/2}\\
 \leqslant & C(1+|x|^r)\sqrt{\varepsilon} \\
 & + C_{\varepsilon}\Exp{\int_{0}^{+\infty} e^{-\ta s} \norm{\nabla_x \Xx_s}^2 \ds }^{1/2} \Exp{\int_0^{+\infty} e^{-\ta s}\left(|Y_s^{N,x}-Y_s^{x}|^2+\norm{Z_s^{N,x}-Z^{x}_s}^2\right)\ds}^{1/2}\\
 &+C \Exp{\int_0^{+\infty} e^{-\ta s} \norm{\nabla_x \Xx_s}^2 \ds }^{1/2} \Exp{\int_N^{+\infty} e^{-\ta s} (1+|\Xx_s|^{2r})\ds}^{1/2}.
\end{align*}
Thus \eqref{prop:representationBSDEinfinitehorizon:hyp:b} gives us, for any {bounded} set $K \subset \bR^d$,
\begin{align*}
 \sup_{x \in K} \norm{\nabla_x u_N(0,x)-\bar{u}(x)\sigma^{-1}(x)}
 \leqslant &C_K \sqrt{\varepsilon}+ C_K\left(\int_N^{+\infty} e^{-\ta s}(1+\sup_{x \in K} \Exp{|\Xx_s|^{2r}})\ds\right)^{1/2}\\
 &+C_{\varepsilon}  \sup_{x \in K} \Exp{\int_0^{+\infty} e^{-\ta s}\left(|Y_s^{N,x}-Y_s^{x}|^2+\norm{Z_s^{N,x}-Z^{x}_s}^2\right)\ds}^{1/2}.
\end{align*}}
{By using \cite[(5.24)]{pard:rasc:14} and by checking carefully the proof of  \cite[Lemma 3.1]{Briand-Hu-98}, we can easily obtain 
\begin{align*}
 \lim_{N \rightarrow + \infty}\sup_{x \in K} \Exp{\int_0^{+\infty} e^{-\ta s}\left(|Y_s^{N,x}-Y_s^{x}|^2+\norm{Z_s^{N,x}-Z_s^{x}}^2\right)\ds}^{1/2}=0, 
\end{align*}
recalling that $\ta \geqslant 2\mu-{K_{f,z}^2}(\1_{d'>1}\vee\1_{ r>0})$. Then, we get
$$\limsup_{N \rightarrow +\infty} \sup_{x \in K} \norm{\nabla_x u_N(0,x)-\bar{u}(x)\sigma^{-1}(x)} \leqslant C_K\sqrt{\varepsilon}, \quad \forall \varepsilon \in (0,1],$$
which implies that
$$\lim_{N \rightarrow +\infty} \sup_{x \in K} \norm{\nabla_x u_N(0,x)-\bar{u}(x)\sigma^{-1}(x)} = 0.$$
From proofs of  \cite[Theorem 5.57]{pard:rasc:14} and  \cite[Lemma 3.1]{Briand-Hu-98}, we also have  that 
$$\lim_{N \rightarrow + \infty} u_N(0,x)=u(x),\quad \text{ for all } x \in \bR^d.$$ 
Finally, we can use a theorem of interchange between limit and differentiation to get that $(u_N(0,.),\nabla_x u_N(0,.))$ tends to $(u,\bar{u}\sigma^{-1})$  uniformly on all compacts, $u \in C^1(\bR^{d},\bR^{d'})$ and $\nabla_x u=\bar{u}\sigma^{-1}$. By standard arguments, see e.g.  \cite[Corollary 4.1]{ElKaroui-Peng-Quenez-97}, $\bar{u}(\Xx)$ is also a continuous version of $\Zx$.

2.b. All calculations done in part 2.a. and estimates obtained do not depend on the derivative of $f$. Then, we can consider a smooth simple approximation of $f$ by a sequence $(f_n)_{n \in \bN}$ and use the same strategy of proof as in 2.a. to tackle the non-smooth setting by sending $n$ to infinity.\qed}

\subsection{Proof of Lemmas  \ref{lem:newLipcste} and \ref{lem:boundcont}}
\label{proof:2lemmas}
{\paragraph{Proof of Lemma \ref{lem:newLipcste}.}
 For any $x,y\in \bR^{d'}$, we just have to remark that
 \begin{align}
 |h(y)+a y-(h(x)+a x)|^2&=|h(y)-h(x)|^2+
 2\ a\ \langle h(y)-h(x),y-x \rangle+a^2|y-x|^2\\
 &\leqslant |y-x|^2\left(K^2-2 \mu a+a^2\right),
 \end{align}
 and the proof is finished. \qed }

\paragraph{Proof of Lemma \ref{lem:boundcont}.}
Throughout the proof, $C$ denotes a constant that may change from line to line and that does not depend on $x$ or $n$. We set
$$F(y):=f\big(y,w^1(y),w^2(y)\big)+a\,w^1(y),\qquad \tilde F(y):=f\big(y,w^1(y),w^2(y)\big)+\ta\,w^1(y),\qquad y \in \bR^d,$$
and $V:=\sqrt{\tE}e^{-(\ta-\ttheta)\tE}$, so that
$$\Phi^1(w)(x)=\E{\frac{e^{-(a-\theta)E}}{\theta}F(\xxe)}\quad\text{and}\quad \Phi^2(w)(x)=\sqrt{\frac{\pi}{\ttheta}}\,\E{\tilde F(\xxet)\,V\,\Ux_\tE}.$$
The functions $F$ and $\tilde F$ are continuous. Moreover, by the growth of $f$ (Assumption \eqref{as:generatorf}), $\normr{w}<+\infty$ and \eqref{eq:growth:rho}, which gives $1+|y|^r \leqslant C\rho(y)$, we have
\begin{equation}
\label{eq:growth:F}
\norm{F(y)}+\norm{\tilde F(y)} \leqslant C\left(1+|y|^r+\norm{w^1(y)}+\norm{w^2(y)}\right) \leqslant C\left(1+\normr{w}\right)\rho(y),\qquad y\in \bR^d.
\end{equation}

\emph{Finiteness of $\normr{\Phi(w)}$.} By \eqref{eq:growth:F} and the definition \eqref{eq:condition:normA:rho} of $c_{\infty,\eqref{eq:condition:normA:rho}}$ and $\tilde{c}_{\infty,\eqref{eq:condition:normA:rho}}$, we get, for all $x \in \bR^d$,
$$\frac{\norm{\Phi^1(w)(x)}+\norm{\Phi^2(w)(x)}}{\rho(x)} \leqslant C\left(1+\normr{w}\right)\left(c_{\infty,\eqref{eq:condition:normA:rho}}+\tilde{c}_{\infty,\eqref{eq:condition:normA:rho}}\right)<+\infty .$$

\emph{Continuity of $\Phi(w)$.} Let $(x^n)_{n\in\bN}$ be a sequence converging to $x$ and let $K:=\{x^n,\,n\in\bN\}\cup\{x\}$, which is a bounded set. We introduce
$$\xi^n:=\frac{e^{-(a-\theta)E}}{\theta}F(\xxen),\quad \xi:=\frac{e^{-(a-\theta)E}}{\theta}F(\xxe),\quad \zeta^n:=\tilde F(\xxetn)\,V\,U^{x^n}_\tE,\quad \zeta:=\tilde F(\xxet)\,V\,\Ux_\tE,$$
so that $\Phi^1(w)(x^n)-\Phi^1(w)(x)=\E{\xi^n-\xi}$ and $\Phi^2(w)(x^n)-\Phi^2(w)(x)=\sqrt{\pi/\ttheta}\,\E{\zeta^n-\zeta}$. By Vitali's convergence theorem, it is enough to show that $(\xi^n)_{n\in\bN}$ and $(\zeta^n)_{n\in\bN}$ are bounded in $L^{1+\epsilon}$, and that $\xi^n\to\xi$ and $\zeta^n\to\zeta$ in probability.

\emph{Step 1: bounds in $L^{1+\epsilon}$.}  \eqref{eq:growth:F} and Assumption \eqref{as:suprho} applied on $K$ give
\begin{align*}
\sup_{n\in\bN}\E{\norm{\xi^n}^{1+\epsilon}}&\leqslant C\sup_{y\in K}\E{\left(e^{-(a-\theta)E}\rho(X^y_E)\right)^{1+\epsilon}}<+\infty,\\
\sup_{n\in\bN}\E{\norm{\zeta^n}^{1+\epsilon}}&\leqslant C\sup_{y\in K}\E{\left(\sqrt{\tE}e^{-(\ta-\ttheta)\tE}\norm{U^y_\tE}\rho(X^y_\tE)\right)^{1+\epsilon}}<+\infty.
\end{align*}

\emph{Step 2: convergence in probability.} We first recall a consequence of the independence between $(E,\tE)$ and $W$. Let $\tau$ be a $(0,+\infty)$-valued random variable independent of $W$, with law $\gamma$, and let $(Y^n_t)_{t>0}$, $(Y_t)_{t>0}$ be continuous processes adapted to the filtration of $W$ such that $Y^n_t\to Y_t$ in probability for every $t>0$. By the disintegration of $\mathbb{P}$ with respect to $\tau$ (or Fubini's theorem), for all $\delta>0$,
\begin{equation}
\label{eq:disintegration}
\mathbb{P}\left(\norm{Y^n_\tau-Y_\tau}>\delta\right)=\int_0^{+\infty}\mathbb{P}\left(\norm{Y^n_t-Y_t}>\delta\right)\gamma(\dt).
\end{equation}
The integrand is bounded by $1$ and goes to $0$ for every $t>0$, so the dominated convergence theorem gives $Y^n_\tau\to Y_\tau$ in probability.

Since $b$ and $\sigma$ are Lipschitz, we have $\E{|X^{x^n}_t-\Xx_t|^4}\leqslant Ce^{Ct}|x^n-x|^4$ for all $t\geqslant 0$. Applying \eqref{eq:disintegration} with $\tau=E$ and $\tau=\tE$, we get $\xxen\to\xxe$ and $\xxetn\to\xxet$ in probability. By continuity of $F$ and $\tilde F$, it follows that $\xi^n\to\xi$ and $\tilde F(\xxetn)\to\tilde F(\xxet)$ in probability. It remains to show that $V\,U^{x^n}_\tE\to V\,\Ux_\tE$ in probability.

Let us set $A_s:=\sigma^{-1}(\Xx_s)\nabla_x\Xx_s$ and $A^n_s:=\sigma^{-1}(X^{x^n}_s)\nabla_xX^{x^n}_s$. By definition \eqref{eq:malliavin:weight} of the Malliavin weight, for all $t>0$,
$$U^{x^n}_t-\Ux_t=\frac1t\int_0^t\langle A^n_s-A_s,dW_s\rangle\,\sigma(x^n)+\frac1t\int_0^t\langle A_s,dW_s\rangle\big(\sigma(x^n)-\sigma(x)\big),$$
and It\^o's isometry gives
$$\E{\norm{U^{x^n}_t-\Ux_t}^2}\leqslant \frac{C}{t^2}\left(\int_0^t\E{\norm{A^n_s-A_s}^2}\ds+\norm{\sigma(x^n)-\sigma(x)}^2\int_0^t\E{\norm{A_s}^2}\ds\right).$$
By Remark \ref{rem:boundestimatesX}, $\sup_{y\in\bR^d,\,s\in[0,t]}\E{\norm{\nabla_xX^y_s}^4}<+\infty$. Since $\sigma^{-1}$ is bounded, it follows that $\E{\norm{A^n_s-A_s}^2}+\E{\norm{A_s}^2}$ is bounded uniformly in $n$ and $s\in[0,t]$. Since $\sigma$ is continuous, the second term goes to $0$. For the first one, by dominated convergence on $[0,t]$, it is enough to show that $\E{\norm{A^n_s-A_s}^2}\to0$ for all $s\geqslant0$. Since $\sigma^{-1}(y)-\sigma^{-1}(y')=\sigma^{-1}(y)\big(\sigma(y')-\sigma(y)\big)\sigma^{-1}(y')$ for all $y,y'\in\bR^d$, $\sigma$ is $K_\sigma$-Lipschitz and $\sigma^{-1}$ is bounded, $\sigma^{-1}$ is Lipschitz with constant $M_{\sigma^{-1}}^2K_\sigma$, and
$$\E{\norm{A^n_s-A_s}^2}\leqslant 2M_{\sigma^{-1}}^4K_\sigma^2\E{|X^{x^n}_s-\Xx_s|^4}^{1/2}\E{\norm{\nabla_xX^{x^n}_s}^4}^{1/2}+2M_{\sigma^{-1}}^2\E{\norm{\nabla_xX^{x^n}_s-\nabla_x\Xx_s}^2}.$$
The first term goes to $0$. For the second one, $\Delta^n:=\nabla_xX^{x^n}-\nabla_x\Xx$ solves a linear SDE with bounded coefficients $\nabla b(X^{x^n})$, $(\nabla\sigma_k(X^{x^n}))_{1\leqslant k\leqslant d}$ ($\sigma_k$ being the $k$-th column of $\sigma$), with $\Delta^n_0=0$ and with source terms $(\nabla b(X^{x^n}_u)-\nabla b(\Xx_u))\nabla_x\Xx_u$ and $(\nabla \sigma_k(X^{x^n}_u)-\nabla \sigma_k(\Xx_u))\nabla_x\Xx_u$. It\^o's formula and Gr\"onwall's lemma give
$$\E{\norm{\Delta^n_s}^2}\leqslant Ce^{Cs}\int_0^s\E{\Big(\norm{\nabla b(X^{x^n}_u)-\nabla b(\Xx_u)}^2+\sum_{k=1}^d\norm{\nabla \sigma_k(X^{x^n}_u)-\nabla \sigma_k(\Xx_u)}^2\Big)\norm{\nabla_x\Xx_u}^2}\mathrm{d}u,$$
and the right-hand side goes to $0$ by dominated convergence on $[0,s]\times\Omega$, since $\nabla b$ and $\nabla \sigma$ are continuous and bounded and $X^{x^n}_u\to\Xx_u$ in probability. Thus, for all $t>0$, $\E{\norm{U^{x^n}_t-\Ux_t}^2}\to0$. Applying \eqref{eq:disintegration} with $\tau=\tE$, $Y^n_t=\sqrt{t}e^{-(\ta-\ttheta)t}U^{x^n}_t$ and $Y_t=\sqrt{t}e^{-(\ta-\ttheta)t}\Ux_t$ gives $V\,U^{x^n}_\tE\to V\,\Ux_\tE$ in probability, and therefore $\zeta^n\to\zeta$ in probability.

\emph{Step 3: conclusion.} By Step 1, $(\xi^n)_{n\in\bN}$ and $(\zeta^n)_{n\in\bN}$ are uniformly integrable. Together with Step 2, Vitali's convergence theorem gives $\E{\xi^n}\to\E{\xi}$ and $\E{\zeta^n}\to\E{\zeta}$, that is $\Phi(w)(x^n)\to\Phi(w)(x)$. Hence $\Phi(w)$ is continuous.
\qed

\subsection{Proof of Propositions  \ref{prop:contraction:BSDE:infinite:horizon} and \ref{prop:condition:contraction:infinitehorizon}}
\label{proof:2propositions}
{\paragraph{Proof of Proposition \ref{prop:contraction:BSDE:infinite:horizon}.}

Using the Proposition \ref{ch4:prop:genY}, we just need a bound on $c_{\infty,\eqref{eq:condition:normA:rho}}$. If $r'\geqslant 2$, then we apply Lemma \ref{lem:ineq:boundmomentX} to get, for $a>r'L_b+\varepsilon$,
\begin{align*}
 c_{\infty,\eqref{eq:condition:normA:rho}} &\leqslant \sup_{x \in \bR^d} \frac{\int_0^{+\infty} e^{-as} (1+\eta^{-1}\Exp{|\Xx_s|^{r'}})\ds}{1+\eta^{-1}|x|^{r'}}\\
 &\leqslant \sup_{x \in \bR^d} \frac{\int_0^{+\infty} e^{-as} (1+\eta^{-1}(C_{\varepsilon}s+|x|^{r'})e^{(r'L_b+\varepsilon)s})\ds}{1+\eta^{-1}|x|^{r'}}\\
 &\leqslant \sup_{x \in \bR^d} \frac{\frac{1}{a}+\frac{\eta^{-1}C_{\varepsilon}}{(a-(r'L_b+\varepsilon))^2} + \frac{\eta^{-1}|x|^{r'}}{a-(r'L_b+\varepsilon)}}{1+\eta^{-1}|x|^{r'}}\\
 &\leqslant \max\left\{ \frac{1}{a}+\frac{C_{\varepsilon}}{\eta(a-(r'L_b+\varepsilon))^2} , \frac{1}{a-(r'L_b+\varepsilon)} \right\}.
\end{align*}
When $0 < r' <2$, we just have to apply Jensen's inequality and Lemma \ref{lem:ineq:boundmomentX} (with a different $\varepsilon$) to get
\begin{align*}
 c_{\infty,\eqref{eq:condition:normA:rho}} &\leqslant \sup_{x \in \bR^d} \frac{\int_0^{+\infty} e^{-as} (1+\eta^{-1}\Exp{|\Xx_s|^2}^{r'/2})\ds}{1+\eta^{-1}|x|^{r'}}\\
 &\leqslant \sup_{x \in \bR^d} \frac{\frac{1}{a}+\int_0^{+\infty}  \eta^{-1}(C_{\varepsilon}s+|x|^2)^{r'/2}e^{-(a-(r'L_b+\varepsilon))s}\ds}{1+\eta^{-1}|x|^{r'}}\\
 &\leqslant \sup_{x \in \bR^d} \frac{\frac{1}{a}+\frac{\eta^{-1}}{a-(r'L_b+\varepsilon)}\left(\int_0^{+\infty}  (a-(r'L_b+\varepsilon))(C_{\varepsilon}s+|x|^2)e^{-(a-(r'L_b+\varepsilon))s})\ds\right)^{r'/2} }{1+\eta^{-1}|x|^{r'}}
\end{align*}
and we conclude as previously. When $r'=0$ the result is obvious.

Finally, when $\mu>r'L_b$, we set $0<\varepsilon < (\mu-r'L_b)/2$.
For $a$ large enough 
$$\sqrt{K_{f,y}^2-2a\mu+a^2} = a - \mu + o(1) \leqslant a-\mu+\varepsilon,$$
and 
$$\kappa_\infty \leqslant \max\left\{ \frac{a-\mu+\varepsilon}{a}+\frac{C_{\varepsilon}(a-\mu+\varepsilon)}{\eta(a-(r'L_b+\varepsilon))^{1+\frac{r'}{2} \wedge 1}} , \frac{a-\mu+\varepsilon}{a-(r'L_b+\varepsilon)} \right\}.$$ 
Since
$$\frac{a-\mu+\varepsilon}{a} \leqslant \frac{a-\mu+\varepsilon}{a-(r'L_b+\varepsilon)} <1,$$
we just have to consider $\eta$ large enough in order to conclude.
\qed

{\paragraph{Proof of Proposition \ref{prop:condition:contraction:infinitehorizon}.}

Let us start to study  $c_{\infty,\eqref{eq:condition:normA:rho}}$. We have
\begin{align*}
  \E{\frac{e^{-(a-\theta)E}}{\theta}(1+|X_E^x|^{r'})} =& \int_0^{+\infty} e^{-at} \left(1+\E{|\Xx_t|^{r'}}\right)\dt \leqslant \int_0^{+\infty} e^{-at} \left(1+\E{|\Xx_t|^{2r'}}^{1/2}\right)\dt\\ 
  \leqslant & \frac{1}{a} + \frac{c_{1,\eqref{ass:momentXandnablaX}}}{a-c_{3,\eqref{ass:momentXandnablaX}}} + \frac{c_{2,\eqref{ass:momentXandnablaX}}|x|^{r'}}{a-c_{3,\eqref{ass:momentXandnablaX}}}.
\end{align*}
Thus, we easily obtain that
$$c_{\infty,\eqref{eq:condition:normA:rho}} \leqslant \max \left\{\frac{1}{a} + \frac{c_{1,\eqref{ass:momentXandnablaX}}}{a-c_{3,\eqref{ass:momentXandnablaX}}} ,  \frac{c_{2,\eqref{ass:momentXandnablaX}}}{a-c_{3,\eqref{ass:momentXandnablaX}}} \right\}.$$
Now, we study $\tilde{c}_{\infty,\eqref{eq:condition:normA:rho}}$. We have
\begin{align*}
  &\E{\sqrt{\frac{\pi}{\ttheta}} \sqrt{\tE}e^{-(\ta-\ttheta)\tE}|U_\tE^x|(1+|X_\tE^x|^{r'})}\\ 
  =& \int_0^{+\infty} e^{-\ta t}\E{|\Ux_t|(1+|\Xx_t|^{r'})} \dt\\
  \leq& \int_0^{+\infty} e^{-\ta t}\E{|\Ux_t|^2}^{1/2}(1+\E{|\Xx_t|^{2r'}}^{1/2}) \dt\\
  \leqslant & \int_0^{+\infty} e^{-\ta t} \frac{M_{\sigma}M_{\sigma^{-1}}}{\sqrt{t}} \E{ \frac{1}{t} \int_0^t \norm{\nabla_x X_s}^2 \ds}^{1/2} \left(1+ (c_{1,\eqref{ass:momentXandnablaX}}+c_{2,\eqref{ass:momentXandnablaX}}|x|^{r'})e^{c_{3,\eqref{ass:momentXandnablaX}}t}\right) \dt\\
  \leqslant & M_{\sigma}M_{\sigma^{-1}} c_{4,\eqref{ass:momentXandnablaX}}\int_0^{+\infty} e^{-\ta t} \frac{e^{c_{5,\eqref{ass:momentXandnablaX}} t}}{\sqrt{t}}  \left(1+ (c_{1,\eqref{ass:momentXandnablaX}}+c_{2,\eqref{ass:momentXandnablaX}}|x|^{r'})e^{c_{3,\eqref{ass:momentXandnablaX}}t}\right) \dt\\
  \leqslant & M_{\sigma}M_{\sigma^{-1}} c_{4,\eqref{ass:momentXandnablaX}}\sqrt{\pi}\left( \frac{1}{\sqrt{\ta-c_{5,\eqref{ass:momentXandnablaX}}}} + \frac{c_{1,\eqref{ass:momentXandnablaX}}}{\sqrt{\ta - c_{5,\eqref{ass:momentXandnablaX}} - c_{3,\eqref{ass:momentXandnablaX}}}} + \frac{c_{2,\eqref{ass:momentXandnablaX}}|x|^{r'}}{\sqrt{\ta-c_{5,\eqref{ass:momentXandnablaX}}-c_{3,\eqref{ass:momentXandnablaX}}}} \right)
\end{align*}
which gives us the announced result.  
\qed

\subsection{Proof of Proposition \ref{prop:numericalerror}}
\label{proof:prop:numericalerror}

Let us denote, for $n \geqslant 0$,
$$e_{\infty,n+1}:=\Exp{ \sup_{x \in \bR^d} \norm{\frac{ Pv^{n+1}_M(x) -v(x)}{\rho_{r',\eta}(x)}}}.$$
We have
\begin{align*}
 e_{\infty,n+1} \leqslant  \cE_{\infty,1} +  \cE_{\infty,2} +  \cE_{\infty,3}
\end{align*}
with
\begin{align*}
 \cE_{\infty,1} &:= \Exp{ \sup_{x \in \bR^d} \norm{ \frac{Pv^{n+1}_M(x) -P(T_{B,{{\rho}_r}}(\bE_{v^n_M}[R^{(\cdot)}(Pv^n_M)]))(x)}{\rho_{r',\eta}(x)}}},\\
 \cE_{\infty,2} &:= \Exp{ \sup_{x \in \bR^d} \norm{\frac{Pv(x)-P(T_{B,{{\rho}_r}}(\bE_{v^n_M}[R^{(\cdot)}(Pv^n_M)]))(x)}{\rho_{r',\eta}(x)}}},\\
 \cE_{\infty,3} &:= \sup_{x \in \bR^d} \norm{\frac{Pv(x)-v(x)}{\rho_{r',\eta}(x)}},
\end{align*}
where $\mathbb{E}_{v^n_M}$ is the expectation conditionally to  {the $\sigma$-algebra generated by the random function $v_M^n$ and, for any $\phi = (w^1, w^2) : \bR^d \rightarrow \bR^{d'}\times \bR^{d' \times d}$ and $x \in \Pi$, $R^x(\phi)$ is independent with $v_M^n$. In particular, we have $\bE_{v^n_M}[R^{x}(Pv^n_M)]=\E{R^{x}(\phi)}_{|\phi = Pv^n_M}$.}

\paragraph{Error $\cE_{\infty,1}$:}
Recalling that $P$ is linear, {$\lfloor . \rfloor_{B}$} is $1$-Lipschitz and applying Item \ref{item3} of Proposition \ref{prop:properties-P}, we get
\begin{align}
 \nonumber \cE_{\infty,1} &= \Exp{ \sup_{x \in \bR^d} \norm{ \frac{P\left(T_{B,{{\rho}_r}}(\frac{1}{M}\sum_{j=1}^M R_{n,j}^{(\cdot)}(Pv^n_M)) -T_{B,{{\rho}_r}}(\bE_{v^n_M}[R^{(\cdot)}(Pv^n_M)])\right)(x)}{\rho_{r',\eta}(x)}}}\\
 &\leqslant  {\Exp{\sup_{z \in \Pi} \frac{{\rho}_r(z)}{\rho_{r',\eta}(z)}\norm{ \left\lfloor \frac{1}{M}\sum_{j=1}^M\frac{R_{n,j}^{z}(Pv^n_M)}{{\rho}_r(z)}\right\rfloor_{B} - \left\lfloor \bE_{v^n_M}\left[\frac{R^{z}(Pv^n_M)}{{\rho}_r(z)} \right]\right\rfloor_{B} }}\sup_{x \in \bR^d} \left|\frac{P\rho_{r',\eta}(x)}{\rho_{r',\eta}(x)} \right|}\\
 &\leqslant \sup_{x \in \bR^d} \left|\frac{P\rho_{r',\eta}(x)}{\rho_{r',\eta}(x)} \right| \Exp{\sup_{z \in \Pi} \norm{H_{z}}}\,,
 \label{ineq:error1}
\end{align}
where 
$$H_z := \frac{1}{M} \sum_{j=1}^M \frac{R^{z}_{n,j}(Pv^n_M)}{\rho_{r',\eta}(z)} - \bE_{v^n_M}\left[\frac{R^{z}(Pv^n_M)}{\rho_{r',\eta}(z)} \right].$$

Since $\Psi$ is a convex function, Jensen inequality gives us
\begin{align*}
 \Exp{\sup_{z \in \Pi} \norm{H_{z}}} &\leqslant  \left| \sup_{z \in \Pi} \norm{H_z} \right|_{\Psi} \Psi^{-1}\left( \Exp{\Psi \left(\frac{\sup_{z \in \Pi} \norm{H_z}}{\left| \sup_{z \in \Pi} \norm{H_z} \right|_{\Psi}}\right)}\right)\leqslant  \Psi^{-1}(1) \left| \sup_{z \in \Pi} \norm{H_z} \right|_{\Psi}.
\end{align*}
Then we upper bound the right hand side of the previous inequality by applying Maximal inequality \eqref{ineq:maximal} to get
\begin{align}
\label{ineq:EsupH}
 \Exp{\sup_{z \in \Pi} \norm{H_z}} &\leqslant C \Psi^{-1}(N)\sup_{z \in \Pi} \left| H_{z} \right|_{\Psi}.
\end{align}
Since, conditionally to $v^n_M$, $H_z$ is a sum of centered independent r.v., we can apply a conditional version of Talagrand inequality \eqref{ineq:Talagrand}:
\begin{align*}
 \left| H_{z} \right|_{\Psi,v^n_M} & \leqslant C\left( \mathbb{E}_{v_M^n}\left[ \norm{H_z}\right]+ \left| \sup_{1 \leqslant j \leqslant M} \frac{\norm{ \frac{R^z_{n,j}(Pv^n_M)}{\rho_{r',\eta}(z)} - \bE_{v^n_M}\left[\frac{R^{z}(Pv^n_M)}{\rho_{r',\eta}(z)} \right]} }{M} \right|_{\Psi,v_M^n} \right)
\end{align*}
where the notation $|.|_{\Psi,v_M^n}$ means that the Orlicz norm is computed conditionally to $v_M^n$.
It gives us, using that $H_z$ is a sum of centered i.i.d. r.v. conditionally to $v^n_M$, Maximal inequality \eqref{ineq:maximal} and Assumption \eqref{ass:finite-Orlicz},
{\begin{align*}
 \norm{H_{z}}_{\Psi,v^n_M}  \leqslant & C \frac{1}{\sqrt{M}}\bE_{v_M^n}\left[ \norm{ \frac{R^z(Pv^n_M)}{\rho_{r',\eta}(z)} - \bE_{v^n_M}\left[\frac{R^{z}(Pv^n_M)}{\rho_{r',\eta}(z)} \right]}^2 \right]^{1/2}\\
 &+ CC_{\Psi}\frac{\Psi^{-1}(M)}{M}\left| \frac{R^z(Pv^n_M)}{\rho_{r',\eta}(z)} - \bE_{v^n_M}\left[\frac{R^{z}(Pv^n_M)}{\rho_{r',\eta}(z)} \right] \right|_{\Psi,v^n_M}\\
 \leqslant & C'(M^{-1/2}+M^{-1} \Psi^{-1}(M)).  
\end{align*}}

Since the previous bound is deterministic, we just have to put it in \eqref{ineq:EsupH} and \eqref{ineq:error1} to get
$$ \cE_{\infty,1} \leqslant C\sup_{x \in \bR^d} \left|\frac{P\rho_{r',\eta}(x)}{\rho_{r',\eta}(x)} \right| \frac{\Psi^{-1}(N)}{\sqrt{M}}\left(1+\frac{\Psi^{-1}(M)}{\sqrt{M}}\right).$$

\paragraph{Error $\cE_{\infty,2}$:}
{Recalling Proposition \ref{pr:Lpestimate}, we have that $v$ is the unique solution of the fixed point equation $\Phi(v)=v$. Moreover, $B \geqslant \norm{v}_{{\rho}_r}$ implies that $T_{B, {\rho}_r}(\Phi(v)) = v$. Then we get, by using the linearity of $P$ and the third inequality in Proposition \ref{prop:properties-P}, 
\begin{equation*}
	\begin{split}
	\cE_{\infty,2} 
		 = & \Exp{ \sup_{x \in \bR^d} \norm{\frac{P(T_{B,{{\rho}_r}}\Phi(v(\cdot))-T_{B,{{\rho}_r}}\Phi(Pv^n_M(\cdot)))(x)}{\rho_{r',\eta}(x)}}} \\
		 \leqslant  & \left(\sup_{x \in \bR^d} \frac{P\rho_{r',\eta}(x)}{\rho_{r',\eta}(x)}\right)\Exp{ \sup_{x \in \bR^d} \norm{\frac{T_{B,{{\rho}_r}}\Phi(v(\cdot))(x)-T_{B,{{\rho}_r}}\Phi(Pv^n_M(\cdot))(x)}{\rho_{r',\eta}(x)}}}\,\\
         \leqslant  & \left(\sup_{x \in \bR^d} \frac{P\rho_{r',\eta}(x)}{\rho_{r',\eta}(x)}\right)\Exp{ \sup_{x \in \bR^d} \frac{{\rho}_r(x)}{\rho_{r',\eta}(x)}\norm{\left\lfloor \frac{\Phi(v(\cdot))(x)}{{\rho}_r(x)}\right\rfloor_B-\left\lfloor \frac{\Phi(Pv^n_M(\cdot))(x)}{{\rho}_r(x)}\right\rfloor_B }}.
	\end{split}
\end{equation*}
Then, the fact that {$\lfloor . \rfloor_{B}$} is $1$-Lipschitz and Proposition \ref{pr:Lpestimate} give us
\begin{align}
    \cE_{\infty,2} \leqslant & \left(\sup_{x \in \bR^d} \frac{P\rho_{r',\eta}(x)}{\rho_{r',\eta}(x)}\right)\Exp{ \sup_{x \in \bR^d} \norm{ \frac{\Phi(v(\cdot))(x)-\Phi(Pv^n_M(\cdot))(x)}{\rho_{r',\eta}(x)} }}\,\\
    \leqslant & \left(\sup_{x \in \bR^d} \frac{P\rho_{r',\eta}(x)}{\rho_{r',\eta}(x)}\right)\kappa_\infty \Exp{ \sup_{x \in \bR^d} \norm{\frac{v(x)-Pv^n_M(x)}{\rho_{r',\eta}(x)}}} = \left(\sup_{x \in \bR^d} \frac{P\rho_{r',\eta}(x)}{\rho_{r',\eta}(x)}\right)\kappa_{\infty} e_{\infty,n}.
\end{align}
}

\paragraph{Error $\cE_{\infty,3}$:}
We have assumed that $v$ is $C^2$. Then, by using the second inequality in Proposition \ref{prop:properties-P} thanks to the assumption on the growth of $\norm{\nabla^2 v}$, and the growth of $v$, we have
\begin{align*}
 \sup_{x \in \bR^d} \norm{\frac{Pv(x)-v(x)}{\rho_{r',\eta}(x)}} & \leqslant \sup_{x \in \Box} \norm{\frac{Pv(x)-v(x)}{\rho_{r',\eta}(x)}} + \sup_{x \in \bR^d \setminus \Box} \norm{\frac{Pv(x)-v(x)}{\rho_{r',\eta}(x)}}\\
 & \leqslant C\delta^2(1+\delta^{r'})+ C\sup_{x \in \bR^d \setminus \Box} \left(\frac{1+|x|^r}{\rho_{r',\eta}(x)}\right).
\end{align*}
where $C$ does not depend on $\Pi$. 

\paragraph{Error $e_{\infty,n}$:} 
Now we just have to collect previous estimates to get
\begin{align*}
 e_{\infty,n+1} \leqslant& C\sup_{x \in \bR^d} \left|\frac{P\rho_{r',\eta}(x)}{\rho_{r',\eta}(x)} \right|\frac{\Psi^{-1}(N)}{\sqrt{M}}\left(1+\frac{\Psi^{-1}(M)}{\sqrt{M}}\right)+\left(\sup_{x \in \bR^d} \frac{P\rho_{r',\eta}(x)}{\rho_{r',\eta}(x)}\right)\kappa_{\infty}e_{\infty,n}
 +C\delta^2(1+\delta^{r'})\\ 
 &+  C\sup_{x \in \bR^d \setminus \Box} \left(\frac{1+|x|^r}{\rho_{r',\eta}(x)}\right).
\end{align*}
{Since $r' \geqslant 2$, we can use inequality 4. of Proposition \ref{prop:properties-P} to get that $\sup_{x \in \bR^d} \frac{P\rho_{r',\eta}(x)}{\rho_{r',\eta}(x)} \leqslant 1+C\delta^2(1+\delta^{r'})$. Then we just have to set $\varepsilon>0$ and $\delta_0>0$ small enough such that 
\begin{equation}
    \sup_{0<\delta \leqslant \delta_0} \sup_{x \in \bR^d} \frac{P\rho_{r',\eta}(x)}{\rho_{r',\eta}(x)} \leqslant 1+\varepsilon <\kappa_{\infty}^{-1}.
\end{equation}
}
Finally, our estimates on $e_{\infty,n}$, gives us, for $\delta \leqslant \delta_0$,
\begin{align*}
 e_{\infty,n}
 \leqslant &C\left(\sum_{k=0}^{n-1} (1+\varepsilon)^k\kappa_{\infty}^k\right) \left(\frac{\Psi^{-1}(N)}{\sqrt{M}}\left(1+\frac{\Psi^{-1}(M)}{\sqrt{M}}\right)+\delta^2+\sup_{x \in \bR^d \setminus \Box} \left(\frac{1+|x|^r}{\rho_{r',\eta}(x)}\right)\right)
 +(1+\varepsilon)^n\kappa_{\infty}^n e_{\infty,0}
\end{align*}
and we just have to use that $0 \leqslant (1+\varepsilon)\kappa_{\infty} <1$ in order to conclude.

\subsection{Proof of Lemma \ref{lem:estim:orlicz}}
\label{proof:lem:estim:orlicz}
1. We have for all $c>0$, by using the growth of $\phi$, the inequality $(x+y)^{\frac{1}{r+1}} \leqslant x^{\frac{1}{r+1}}+y^{\frac{1}{r+1}}$ and Young's inequality, 
\begin{align*}
 \Exp{ e^{c\norm{\frac{R^z(P\phi)}{\rho_{r',\eta}(z)}}^{\frac{1}{r+1}}}}
 \leqslant& \Exp{ e^{cC\left(\frac{|U_\tE^z|\sqrt{\tE}e^{-(\ta-\ttheta)\tE}(1+|X_\tE^z|^r)+e^{-(a-\theta)E}(1+|X_E^z|^r)}{\rho_{r',\eta}(z)}\right)^{\frac{1}{r+1}}}}\\
 \leqslant& \Exp{ e^{cC\left(\frac{|U_\tE^z|\sqrt{\tE}e^{-(\ta-\ttheta)\tE}(1+|X_\tE^z|^r)}{\rho_{r',\eta}(z)}\right)^{\frac{1}{r+1}}} e^{cC\left(\frac{e^{-(a-\theta)E}(1+|X_E^z|^r)}{\rho_{r',\eta}(z)}\right)^{\frac{1}{r+1}}}}\\
 \leqslant & \frac12 \Exp{ e^{cC\left(\frac{|U_\tE^z|\sqrt{\tE}e^{-(\ta-\ttheta)\tE}(1+|X_\tE^z|^r)}{\rho_{r',\eta}(z)}\right)^{\frac{1}{r+1}}}} + \frac12 \Exp{ e^{cC\left(\frac{e^{-(a-\theta)E}(1+|X_E^z|^r)}{\rho_{r',\eta}(z)}\right)^{\frac{1}{r+1}}}},
\end{align*} 
where, as usual, $C$ may change from one term to another but it does not depend on $c$, $\Pi$, $z$ and $\phi$ (even if it depends on $B$).
Young's inequality gives us, if $r>0$,
$$\left(\frac{|U_\tE^z|\sqrt{\tE}e^{-(\ta-\ttheta)\tE}(1+|X_\tE^z|^r)}{\rho_{r',\eta}(z)}\right)^{\frac{1}{r+1}} \leqslant e^{-\frac{(\ta-\ttheta)}{r+1}\tE} \left( \frac{|U_\tE^z|\sqrt{\tE}}{r+1} + \frac{r}{r+1}\left(\frac{(1+|X_\tE^z|^r)}{\rho_{r',\eta}(z)}\right)^{\frac1r} \right).$$
Moreover, when $r=0$ the previous upper-bound stays true without the second term and with a coefficient $2$ in front.
Then,  {the Cauchy-Schwarz inequality gives us,} 
\begin{align}
  \Exp{ e^{c\norm{\frac{R^z(P\phi)}{\rho_{r',\eta}(z)}}^{\frac{1}{r+1}}}} \leqslant& \frac{e^{cC}}{2}\Exp{ e^{cC|U_\tE^z|\sqrt{\tE}e^{-\frac{(\ta-\ttheta)}{1+r}\tE}}}^{ {1/2}}\Exp{ e^{cC\1_{r>0}\frac{|X_\tE^z| }{(1+|z|^{r'})^{1/r}}e^{-\frac{(\ta-\ttheta)}{1+r}\tE}}}^{ {1/2}}\\
  &+\frac{e^{cC}}{2}\Exp{e^{cC\frac{|X_{{E}}^z|^{r/(r+1)}}{(1+|z|^{r'})^{1/(r+1)}}e^{-\frac{(a-\theta)}{r+1} E}}}.
  \label{bound E exp R}
\end{align}
Let us upper bound three previous terms (only two if $r=0$).
\begin{align}
 \Exp{ e^{cC|U_\tE^z|\sqrt{\tE}e^{-\frac{(\ta-\ttheta)}{1+r} \tE}}} \leqslant \int_0^{+\infty} \frac{\sqrt{\ttheta}}{\sqrt{\pi}\sqrt{t}} \Exp{e^{cC \frac{1}{\sqrt{t}}M_{\sigma^{-1}}\left|\int_0^t \langle \nabla_x X_s^z , \dW_s\rangle\right| M_{\sigma} e^{-\frac{(\ta-\ttheta)}{1+r} {t}}}}e^{-\ttheta t} \dt.  
 \label{bound-exp-UEexp}
\end{align}
Since $\norm{\nabla_x X_s^z} \leqslant e^{K_b t}$ for $0 \leqslant s \leqslant t$, it follows from the Dambis-Dubins-Schwarz representation of the continuous martingale $t \mapsto \int_0^t \langle \nabla_x X_s^z , \dW_s\rangle$ that
\begin{align}
 \Exp{e^{cC \frac{1}{\sqrt{t}}\left|\int_0^t \langle \nabla_x X_s^z , \dW_s \rangle \right|  e^{-\frac{(\ta-\ttheta)}{1+r} {t}}}} &\leqslant \Exp{\sup_{0 \leqslant s \leqslant e^{2K_b t}t} e^{cC\frac{1}{\sqrt{t}}\left|W_s\right|  e^{-\frac{(\ta-\ttheta)}{1+r} {t}}}}\\
 &\leqslant \Exp{\sup_{0 \leqslant s \leqslant 1} e^{cC \left|W_s\right| e^{-\frac{(\ta-\ttheta-(1+r)K_b)}{1+r} {t}}}}\\
 &\leqslant \Exp{ e^{cC \sup_{0 \leqslant s \leqslant 1}W_s}e^{cC \sup_{0 \leqslant s \leqslant 1}(-W_s)}}
 \leqslant \Exp{ e^{2cC |W_1|}}
 \label{bound-exp-U}
\end{align}
using assumption $\ta-\ttheta-(1+r)K_b>0$, Cauchy-Schwarz inequality and the fact that $\sup_{0 \leqslant s \leqslant 1}W_s$, $\sup_{0 \leqslant s \leqslant 1} (-W_s)$ and $|W_1|$ have the same distribution.
 {Let us remark that for any $\alpha\geqslant 0$ we have
\begin{align}
\label{bound-Laplace-absgauss}
    \Exp{e^{\alpha|W_1|}} = \left(1+\int_{-\alpha}^\alpha \frac{1}{\sqrt{2\pi}} e^{-y^2/2}\dy\right)e^{\alpha^2/2} \leqslant \left(1+\frac{2\alpha}{\sqrt{2\pi}}\right)e^{\alpha^2/2} \leqslant e^{\frac{\alpha^2}{2}+\frac{2\alpha}{\sqrt{2\pi}}} .
\end{align}
}
Thus, by plugging \eqref{bound-Laplace-absgauss} into \eqref{bound-exp-U} and then \eqref{bound-exp-U} into \eqref{bound-exp-UEexp}, we get that
$$\Exp{ e^{cC|U_\tE^z|\sqrt{\tE}e^{-\frac{(\ta-\ttheta)}{1+r}\tE}}}\leqslant e^{(c+c^2)C'}.$$
By mimicking computations of \cite[page 563]{Briand-Hu-08} we have
\begin{align*}
 \Exp{ e^{cC\frac{|X_\tE^z| }{(1+|z|^{r'})^{1/r}}e^{-\frac{(\ta-\ttheta)}{1+r}\tE}}} &\leqslant  \int_0^{+\infty} \frac{\sqrt{\ttheta}}{\sqrt{\pi}\sqrt{t}} \Exp{e^{cC \frac{|z|+K_b t+ M_{\sigma}\sup_{0 \leqslant s \leqslant t} |W_s|}{(1+|z|^{r'})^{1/r}}e^{K_b t} e^{-\frac{(\ta-\ttheta)}{1+r} t}}}e^{-\ttheta t} \dt\\
 &\leqslant   \int_0^{+\infty} \frac{\sqrt{\ttheta}}{\sqrt{\pi}\sqrt{t}} \Exp{\sup_{0\leqslant s\leqslant 1} e^{cC(1+K_b t+ M_{\sigma}\sqrt{t}|W_s|) e^{-\frac{(\ta-\ttheta-(1+r)K_b)}{1+r} t}}}e^{-\ttheta t} \dt.
\end{align*}
 {Since  $\ta-\ttheta-(1+r) K_b>0$, $t \mapsto t e^{-\frac{(\ta-\ttheta-(1+r)K_b)}{1+r} t}$ and $t \mapsto \sqrt{t} e^{-\frac{(\ta-\ttheta-(1+r)K_b)}{1+r} t}$ are bounded on $\bR^+$ which gives us 
\begin{align*}
    \Exp{\sup_{0\leqslant s\leqslant 1} e^{cC(1+K_b t+ M_{\sigma}\sqrt{t}|W_s|) e^{-\frac{(\ta-\ttheta-(1+r)K_b)}{1+r} t}}} \leqslant \Exp{\sup_{0\leqslant s\leqslant 1} e^{cC'(1+|W_s|) }}\leqslant e^{cC'}\Exp{ e^{2cC'|W_1| }}
\end{align*}
and by using \eqref{bound-Laplace-absgauss}, we get
$$\Exp{ e^{cC\frac{|X_\tE^z| }{(1+|z|^{r'})^{1/r}}e^{-\frac{(\ta-\ttheta)}{1+r}\tE}}} \leqslant  e^{(c+c^2)C'}.$$
}

In the same way, recalling that $a-\theta-r K_b>0$, we also have
$$\Exp{ e^{cC\frac{|X_{{E}}^z|^{r/(r+1)}}{(1+|z|^{r'})^{1/(r+1)}}e^{-\frac{(a-\theta)}{r+1} E}}} \leqslant e^{(c+c^2)C'}.$$
By collecting previous estimates and plugging them into \eqref{bound E exp R}, we finally get
\begin{equation}
 \label{ineq:proof:orlicz}
 \Exp{ e^{c\norm{\frac{R^z(P\phi)}{\rho_{r',\eta}(z)}}^{\frac{1}{r+1}}}} \leqslant e^{(c+c^2)C}.
\end{equation}

Moreover, we also have, using the growth of $f$, $w_1$ and $w_2$, and the Cauchy-Schwarz inequality,
\begin{align*}
\norm{\Exp{ \left(\frac{R^z(P\phi)}{\rho_{r',\eta}(z)}\right)}} \leqslant & \Exp{\norm{\frac{R^z(P\phi)}{\rho_{r',\eta}(z)}}^{2}}^{1/2} \\
 \leqslant & C\left(\int_0^{+\infty} \frac{1}{\sqrt{t}}\Exp{\frac{t|U_{t}^z|^2 (1+|X_{t}^z|^{2r})}{\rho_{r',\eta}(z)^2}} e^{-(2\ta-\ttheta)t}\dt \right)^{1/2}\\
 & +C\left(\int_0^{+\infty} \Exp{\frac{1+|X_t^z|^{2r}}{\rho_{r',\eta}(z)^2} } e^{-(2a-{\theta})t}\dt \right)^{1/2}\\
 \leqslant & C\left(\int_0^{+\infty} \frac{1}{\sqrt{t}}\Exp{t^2|U_{t}^z|^4}^{1/2} \sqrt{2} \Exp{\frac{1+|X_{t}^z|^{4r}}{\rho_{r',\eta}(z)^4}}^{1/2} e^{-(2\ta-\ttheta)t}\dt \right)^{1/2}\\
 & +C\left(\int_0^{+\infty} \Exp{\frac{1+|X_t^z|^{2r}}{\rho_{r',\eta}(z)^2} } e^{-(2a-{\theta})t}\dt \right)^{1/2}.
\end{align*}
Thanks to Burkholder-Davis-Gundy and Jensen's inequalities, we have, for all $ p \geqslant 1$,
\begin{align}
\label{bound:E[tUt]}
  \hspace{-4mm}\Exp{t^p|U_{t}^z|^{2p}} \leq& C \Exp{\left(\frac{1}{t}\int_0^t \norm{\nabla_x X_s}^2 \ds \right)^{p}} \leqslant C \frac{1}{t}\int_0^t \Exp{\norm{\nabla_x X_s}^{2p}} \ds \leqslant C \sup_{s \in [0,t]} \Exp{\norm{\nabla_x X_s}^{2p}}.\quad 
\end{align}
Then, using \eqref{bound:E[tUt]} with $p=2$ and Remark \ref{rem:boundestimatesX}, we get, for all $\varepsilon>0$,
\begin{align*}
\Exp{ \norm{\frac{R^z(P\phi)}{\rho_{r',\eta}(z)}}^{2}}^{1/2} 
 \leqslant &  C_\varepsilon \left(\int_0^{+\infty} \frac{1}{\sqrt{t}}e^{2K_bt} e^{(2rK_b+\varepsilon)t} e^{-(2\ta-\ttheta)t}\dt \right)^{1/2}\\
 & +C_\varepsilon \left(\int_0^{+\infty} e^{(2rK_b+\varepsilon)t} e^{-(2a-{\theta})t}\dt \right)^{1/2}.
\end{align*}
This upper bound is finite for $\epsilon$ small enough, recalling that  $\ta-\ttheta-(1+r)K_b>0$ and $a-\theta-r K_b>0$. In particular, it implies the second part of \eqref{ass:finite-Orlicz}. Finally, by using $\Psi^{LT}_{\beta}(x) \leqslant e^{x_\beta^\beta}\left(e^{x^\beta}-1\right)$ with $\beta=1/(r+1)$, the previous bound and \eqref{ineq:proof:orlicz}, we also have
\begin{align*}
\Exp{\Psi^{LT}_{1/(r+1)}\left(c\frac{R^z(P\phi)}{\rho_{r',\eta}(z)} - c\Exp{\frac{R^{z}(P\phi)}{\rho_{r',\eta}(z)} }\right)} \leqslant e^{x_{1/(r+1)}^{1/(r+1)}}\left(e^{(c+c^2)C}-1\right).
\end{align*}
Since $c \mapsto e^{(c+c^2)C}-1$ is a continuous increasing function from $\bR^+$ to $\bR^+$, we can conclude that the first part of \eqref{ass:finite-Orlicz} also holds true.

2. When $\sigma$ is not constant, the proof of Item \ref{lem:estim:orlicz:item1} becomes false only due to the fact that $\nabla_x X^z$ has not a deterministic upper-bound, which hinder the treatment of the term
$$\Exp{ e^{cC|U_\tE^z|\sqrt{\tE}e^{-\frac{(\ta-\ttheta)}{1+r} \tE}}}.$$
Nevertheless, this problem disappears when $f(x,y,z)=f(x,y)$. Then the proof of Item \ref{lem:estim:orlicz:item2} follows the same proof as the proof of Item \ref{lem:estim:orlicz:item1} remarking that it is enough to take the power $1/(r\vee 1)$ instead of $1/(r+1)$.

3. As in the proof of Item \ref{lem:estim:orlicz:item1} we have
\begin{align*}
\norm{\Exp{ \left(\frac{R^z(P\phi)}{\rho_{r',\eta}(z)}\right)}} \leqslant & \Exp{ \norm{\frac{R^z(P\phi)}{\rho_{r',\eta}(z)}}^{2}}^{1/2} 
\leqslant \Exp{ \norm{\frac{R^z(P\phi)}{\rho_{r',\eta}(z)}}^{\beta}}^{1/\beta} \\
 \leqslant & C\left(\int_0^{+\infty} \frac{1}{\sqrt{t}}\Exp{\frac{t^{\beta/2}|U_{t}^z|^\beta (1+|X_{t}^z|^{\beta r})}{\rho_{r',\eta}(z)^\beta}} e^{-(\beta \ta-(\beta-1)\ttheta)t}\dt \right)^{1/\beta}\\
 & +C\left(\int_0^{+\infty} \Exp{\frac{1+|X_t^z|^{\beta r}}{\rho_{r',\eta}(z)^\beta} } e^{-(\beta a-(\beta-1){\theta})t}\dt \right)^{1/\beta}\\
 \leqslant & C\left(\int_0^{+\infty} \frac{1}{\sqrt{t}}\Exp{t^\beta|U_{t}^z|^{2\beta}}^{1/2} \Exp{\frac{1+|X_{t}^z|^{2\beta r}}{\rho_{r',\eta}(z)^{2\beta}}}^{1/2} e^{-(\beta \ta-(\beta-1)\ttheta)t}\dt \right)^{1/\beta}\\
 & +C\left(\int_0^{+\infty} \Exp{\frac{1+|X_t^z|^{\beta r}}{\rho_{r',\eta}(z)^\beta} } e^{-(\beta a-(\beta-1){\theta})t}\dt \right)^{1/\beta}.
\end{align*}

Then, using \eqref{bound:E[tUt]} with $p=\beta$ and Remark \ref{rem:boundestimatesX} we get, for all $\varepsilon>0$,
\begin{align*}
\Exp{ \norm{\frac{R^z(P\phi)}{\rho_{r',\eta}(z)}}^{\beta}}^{1/\beta} 
 \leqslant &  C_\varepsilon \left(\int_0^{+\infty} \frac{1}{\sqrt{t}}e^{(\beta K_b+\frac{\beta(2\beta-1)}{2} K_{\sigma}^2)t} e^{(\beta rK_b+\varepsilon)t} e^{-(\beta \ta-(\beta-1)\ttheta)t}\dt \right)^{1/\beta}\\
 & +C_\varepsilon \left(\int_0^{+\infty} e^{(\beta rK_b+\varepsilon)t} e^{-(\beta a-(\beta-1){\theta})t}\dt \right)^{1/\beta}.
\end{align*}
This upper bound is finite for $\varepsilon$ small enough, recalling that  $\ta-(r+1)K_b-\frac{2\beta-1}{2} K_{\sigma}^2-\frac{\beta-1}{\beta} \ttheta>0$ and $a-rK_b-\frac{\beta-1}{\beta} {\theta}>0$. In particular, it implies \eqref{ass:finite-Orlicz}.

\bibliographystyle{alpha}
\bibliography{bibliography.bib}

\end{document}